\documentclass[11pt,letterpaper]{article}
   \usepackage{
     amsmath,
     amsfonts,%  AMS fonts
     amssymb,%    AMS extra symbols
     amsthm, stmaryrd}%     AMS theorems>>>3

     \usepackage[all]{xy}
\newcommand{\PP}{\ensuremath{\mathbb{P}}}   % Projective spaces
\renewcommand{\~}{\widetilde}
\DeclareMathOperator{\Ker}{ker}      % kernel
\DeclareMathOperator{\im}{im}        % image           
\DeclareMathOperator{\Hom}{Hom}      % Hom 
         
\DeclareMathOperator{\codim}{codim}  % codimension
\DeclareMathOperator{\rk}{rk}        % rank        
\newcommand{\OO}{\ensuremath{\mathcal{O}}} %Structure sheaf
\DeclareMathOperator{\cliff}{Cliff}                            	% Cliff
\DeclareMathOperator{\sym}{Sym}                            	% Sym
\DeclareMathOperator{\Sec}{Sec}                            	% Sec
\newcommand{\maps}{\ensuremath{\longrightarrow}}                % long arrow

\newcommand{\set}[1]{\ensuremath{\left\{{#1}\right\}}} 		% Set
\renewcommand{\H}[1]{\ensuremath{H^0\left({#1}\right)}} 	% H_0
\newcommand{\Hd}[1]{\ensuremath{H^0\left({#1}\right)^\vee}} 	% H_0^*
\newcommand{\Hi}[1]{\ensuremath{H^1\left({#1}\right)}} 		% H_1
\newcommand{\Hid}[1]{\ensuremath{H^1\left({#1}\right)\vee}} 	% H_1^*
\newcommand{\h}[1]{\ensuremath{h^0\left({#1}\right)}} 		% h_0
\newcommand{\hi}[1]{\ensuremath{h^1\left({#1}\right)}} 		% h_1
\DeclareMathOperator{\ch}{char}
\usepackage{amsfonts}
\usepackage{amssymb}
\usepackage{amscd}
\usepackage{xypic}

\def\imp{\Longrightarrow}
\def\to{\longrightarrow}
\def\iff{\Leftrightarrow}

\def\C{\mathbb{C}}

\def\P{\mathbb{P}}
\def\cO{\mathcal O}

\def\sec{\mbox{Sec}}

\def\s{\sigma}

\def\p{\pi}

\def\t{\tau}

\theoremstyle{plain}
\newtheorem{thm}{Theorem}[section]
\newtheorem{prop}[thm]{Proposition}
\newtheorem{lemma}[thm]{Lemma}
\newtheorem{cor}[thm]{Corollary}
\newtheorem{conj}[thm]{Conjecture}
\theoremstyle{definition}
\newtheorem{defn}{Definition}[section]
\begin{document}
%<<< Title, abstract, etc.
\title{A Generalization of the Clifford Index and Determinantal Equations for Curves and Their Secant Varieties
}
\author{Adam Ginensky}
\date{December 2008}
%\department{Mathematics}
%\division{Physical Sciences}
%\degree{Doctor of Philosophy}
\maketitle
%\dedication
%\begin{center}

%This paper is dedicated to the memory of  my first three teachers, my parents 
%Richard and Miriam Ginensky and  Prof. Soo Bong Chae.
%\end{center}
\begin{abstract}
This paper gives a geometric characterization of the Clifford index of a curve $C$, in terms of the existence of determinantal
 equations for $C$ and its secant varieties in the bicanonical embedding.  The key idea  is the generalization of Shiffer deformations
 to Shiffer variations supported on an arbitrary effective divisor $D$.  The definition of  Shiffer variations and Clifford index is then
 generalized to an arbitrary very ample line bundle, $L$,  on $C$.  This  allows one to give geometric characterizations of the
generalized  Clifford 
index in many cases.  In particular it allows one to show  the existence of determinantal equations for $C$ and $\sec^k(C)$ for
 $ k< \cliff(C,L)$ in the embedding of $C$ in the projective $ \PP(\H{L^{\otimes 2}}^{\vee})$ .  We then explain how the results can be 
generalized to embeddings of $C$ in  $\PP(\H{L_1 \otimes L_2}^{\vee})$ where $L_1$ and $L_2$ are very ample.
For line bundles of large degree this  proves a conjecture of Eisenbud,  Koh, and Stillman  relating the existence of a determinantal
 presentation for $\sec^k(C)$ to $\deg(L)$.  

\end{abstract}

\tableofcontents
%>>>
%\mainmatter
\section{Introduction}\label{S:i}
When Brendan Hassett was a post doc at the University of Chicago I spoke to him a lot about a paper of Griffiths \cite{Gri}. 
which involved the concept of Shiffer variations. These are infinitesimal deformations of a curve, which have rank one when 
viewed (via cup product) as homomorphisms from $H^0(K_C) \to H^1(\mathcal{O}_C)$.  They are in fact parametrized by points 
of the curve.   Initially I was trying to understand on which curves $C$, every 
rank $j$ infinitesimal deformations is just
the sum of $j$ rank one Shiffer variations.   Theorem \ref{T:mthm}  shows this holds only for $j < \cliff(C)$
 
 In hindsight, it became  clear to me that the construction was of a geometric 
nature and really had nothing to do with infinitesimal deformations per se.  The construction could be made
on any very ample line bundle.  In its most general form the idea is that
 whenever one can factor a line bundle $L$ as $L_1 \otimes L_2$ then one can consider $C$ as embedded in a space of 
matrices $\Hom(\Hd{C,L_1},\H{C,L_2}=\H{C,L_1} \otimes \H{C,L_2} \to \H{C,L}$ and that  this gives one equations  defining 
the curve. Roughly,  because if we let $t_{ij} $ be the image of $u_i \otimes v_j$ in $\H{C,L}$, then we get the equality 
$t_{ij}t_{kl}- t_{ik} t_{jl}=0$ .  One may recognize these as equations defining a 'Segre embedding' . 
This certainly is not a new idea. It was explained for example in \cite{E88}.   What is new, 
is the notion of a Clifford index for any line bundle on a curve and the explanation of how it controls whether the 
curve and its secant varities are defined by  equations of 'Segre type'.  One can show that 
for $j< c$, where $c$ is a 'Clifford index' defined for $L_1 \otimes L_2$,  it is the case that the equations defining
$\sec^j(C)$ are the equations of  $\sec^j(\PP(\Hd{C,L_1}) \times \PP(\Hd{C,L_2}))$ (restricted 
to the linear subspace $\PP(\Hd{C,L})$.   In general tersm,the goal of this thesis 
is to analyze when the equations defining $\sec^j(C)$ in $\PP(\H{C,L_1\otimes L_2})$ are the equations defining of 
$\sec^j(\PP(\Hd{C,L_1}\otimes \PP( \H{C,L_2})) )$ in $\PP(\H{C,L_1} \otimes \H{C,L_2})$.  It is standard to say that
$\sec^j(C)$ is determinantally defined in this circumstance.

The arrangement of the thesis  is as follows. In the first section we review the definition of the Clifford index and 
its relationship to the geometry of a curve $C$.  Shiffer variations are defined and then generalized to  Shiffer variations 
supported on a divisor $D$ .  We then prove a theorem (\ref{T:2} ) which explains how 
Shiffer variations are related to the Clifford index.  The main result of this section 
is a geometric characterization of the Clifford index .  Namely we show that in the bicanonical embedding, $\sec^j(C)$
 is in fact set theoretically the locus of all infinitesimal deformations or rank $j+1$ 
if and only if $j< c-1$.
Recall that points in $\sec^j(C) $ are the linear combinations of $j+1$ points of $C$ and hence are
 the sum of $j+1$ Shiffer variations. Since the sum of $j+1$ rank one
matrices is of rank at most $j+1$ , $\sec^j(C)$ consists of deformations of rank at most $j+1$.
Scheme theoretic equality which implies that $\sec^j(C)$
is defined by equations of degree $j+2$ for $j < (\cliff(C)-1)$ is true and is proven in section 6.  
This implies that  $\sec^j(C)$ is determinantally defined . The methods used are different and deferring
the proof until  later allows us to prove a more general result.

Sections 3 and 4 are devoted to generalizing the machinery of the Clifford index and Shiffer variations to arbitrary line
bundles.  The definitions make sense for arbitrary line bundles, but are only useful as far as we can tell, when the line bundle is
very ample.  One sign that this machinery is of general use is the fact that it  gives a very nice way
of describing line bundles in terms of  their embedding properties.  As another application we give a very short proof
in a special case of the existence of $d$ pointed $d-2$ planes.  This result is in (\cite{ACGH}) but the proof contained in there
is exceedingly technical.  This proof, which perhaps can be generalized to other cases is very elementary.
 Sections 5 and 6 are the heart of the matter.  If L is a very ample
quadratically normal line bundle with $h^1(L)=0$
 then one has
an injection $\PP(\H{C,L^{\otimes 2}}) \hookrightarrow \PP(\sym^2(\H{C,L})$ realizing elements of $\PP(\H{C,L^{\otimes 2}})$ as symetric
matrices. This is true once $\deg(L) \geq 2g+1$
 From our work on Shiffer variations we show  that points on the curve represent rank one transformations and hence
(since the sum of $j$ rank one matrices is a matrix of rank at most $j$) we get an inclusion of $\sec^{j-1}(C) \hookrightarrow
R^j$ where $R^j $ represents the locus of matrices in  $\PP(\H{C,L^{\otimes 2}})$ of rank $\leq j$.  In section 5 applying the 
ideas of sections 3 and 4 we show that these two schemes are the same set theoretically as long as $j < \cliff(C,L)$ where
$\cliff(C,L)$ is the clifford index of $L$ as defined in section 3.  Section 6 is devoted to proving the scheme theoretic
equality of $\sec^j(C)$ and $R^j$.  Essentially by definition the variety $\sec(C)$ is reduced, so the trick is to show
that $R^j$ is also reduced.  The hard work is to show that in fact $R^j$ is normal (and hence reduced!) which seems to be of
independent interest.  The technique used  is to find a resolution $\~R^j$ which in fact is a vector bundle over $\sym^j(C)$.
This idea is one I learned from (\cite{ACGH}), but in fact seems to be related to the `Basic Theorem' of \cite{W}. Since the
varieties $R^j$ are determinantal by definition, it follows that $\sec^{j-1}(C) $ is determinantal.

Up until this point the results have been on embeddings in $L^{\otimes 2}$ and it's obvious factorization as
$L \otimes L$.  In section 7 we take up the question of what happens when $L$ factors as $L_1 \otimes L_2$.
All the machinery developed in sections 3 and 4  generalize and one  can prove a theorem stating  that for $j< \cliff(C,L_1,L_2)$
the secant varieties are determinantally defined.  Rarely is it the case that a divisor is 
'special' for both $L_1$ and $L_2$ and hence one gets improved bounds on
the circumstances under which  one can say that the secant
varieties are determinantal.  In particular for $j=0$ that is to say for the case of the curve $C$ itself,
one recovers the main result of  \cite{eks}  in the case of smooth curves, which is if $\deg(L_i) \geq 2g+1$ and
$L_1 \neq L_2$ if $\deg(L_1)=\deg(L_2)= 2g+1$ then for $L= L_1 \otimes L_2$,   $C$ is determinantally defined in $L$.

In section 8 we discuss the results that occur when $\hi{L}=1$ . First as an application of the ideas of Clifford index
we give another proof of the theorem of Green and Lazarsfeld giving a bound in terms of $\deg(L)$ and $\cliff(C)$
as to when an imbedding 
by a very ample line bundle is quadratically normal.  In essence, a very ample line bundle satisfying $\cliff(C,L) >0$ is 
quadratically normal, and one can only realize a  line bundle of  clifford index zero as the projection from
a linear space of dimension $p-1$ of a line bundle of clifford index $p$.  A computation finishes the proof. We then
take up the question of bounding when $\sec^j(C)$ is determinantally defined.  The same bounds hold as in the case
$\hi{L}=0$ but it is much harder to prove, that for $j> \cliff(C,L)$ that the two varieties, the rank locus and the 
secant variety differ.  Ironically this means that for special line bundles one is far more likely to have a stronger
result than for line bundles with $\hi{L}=0$. 

 Finally in the last section , we consider the question of the 
relationship between these results and Green's conjecture.  The story here is incomplete, as one would like to be 
able to use these results to prove the conjecture.  What we can say is that, the classes we  produce in 
Theorem \ref{T:2} for $D$ a base point free divisor, 
  do give rise to non-trivial Koszul cohomology classes.  This applies in particular to divisors which calculate 
the Clifford index of $C$.  Further these classes are 'decomposable' 
 (see section 9 for details) and that there are no such 'decomposable' Koszual cohomology classes in 
$K_{p,2}(C,K_C)$ for $c< \cliff(C)$.    This hardly settles the matter though.

I want to acknowledge the help I have received over the years from numerous people. 
 Firstly, I want to express my deepest thanks to my thesis advisor, Prof. Spencer Bloch.  
Spencer was unstinting in
his time and advice.  Spencer's remark, `You know I'm retiring soon, so if you want to write a thesis you should probably
get it done now rather than later.' provided the final encouragement to get things done. 
 Spencer was a great friend 
and advisor.  I also want to thank Professor  Brendan Hassett  who first brought my attention the question of the relationship between the Clifford index and determinantal embeddings. Brendan spent numerous hours discussing the problem with me and proofreading
this thesis.    My good friend Professor  Mohan Kumar  discussed this paper with me and read various version offering numerous helpful suggestions .  In addition, over the years we have had so many conversations about so many mathematical topics, it is fair to say
that I could/would never have kept my interest in mathematics alive without his help, 
advice, and information.   Nicholas Passell and Mihnea Popa also read preliminary versions of the thesis offering numerous helpful suggestions.  

I want to thank my wife Carol Lind for providing an environment that was very conducive to writing this thesis.  I want to thank my children David Sidney Ginensky and Katherine Miriam Ginensky for being my children.  
I want to thank WH Trading for all their consideration in providing me with very flexible working hours while I was writing.
 I must also thank Mr. Tom Carrideo for his enthusiastic demonstrating of the 'wobbly H'.
  It helped to provide the  inspiration 
to get this thesis done.

\section{The Clifford Index} \label{S:CI}

Let $C$ be a smooth curve of genus $g \geq 3$. The canonical map $\Phi_1:C
\to \P^{g-1}$ has been studied since the $19^{\mbox{th}}$ century. Two of
the most important theorems with regard to this map are

\begin{itemize}
\item[A:)] (Max Noether) $\Phi_1$ is a projectively normal embedding
unless $C$ is hyperelliptic.
\item[B:)] (Petri) Assume $\Phi_1$ embeds $C \hookrightarrow \P^{g-1}$,
then $C$ is
cut out by the quadrics through $C$ unless $C$ is trigonal or a plane
quintic. 
\end{itemize}

We wish to rephrase these theorems in terms of the Clifford index of C.
Recall that the Clifford index of $C$, Cliff($C$) is defined as follows:

\begin{defn} 

Let L be any line bundle.\\ $\cliff(L) = \deg(L) - 2(\h{C,L}-1)$

\end {defn}

\begin{defn}

The Clifford index of C written  $\cliff(C)$  is \\
$ \min \{ \cliff(L) | \h{L}$ and$ \,  \hi{L} \geq 2\} $

\end{defn}

 Then Cliffords theorem may be restated as Cliff($C$)$\geq 0$
with equality if and only if $C$ is hyperelliptic. Further one has  that
Cliff(C) = 1 if and only if $C$ is trigonal or a plane quintic (see \cite{ACGH}) 
, that is to say $C$ has a  $g_3^1$ or $g_5^2$. This leads to the restatement of Thm A as
$\Phi_1$ is a projectively normal embedding if and only if Cliff($C$)$>0$; and Thm B
as $\Phi_1(C) \subset \P^{g-1}$ is cut out by quadrics if and only if Cliff($C$)$>1$.

Mark Green has proposed algebraic generalizations of these statements to
arbitrary Clifford index which  state that the Clifford index
should control the syzygies of the ideal sheaf of $C \subset \P^{g-1}$.

In this section we will give a geometric generalization of Thms A
and B that  characterizes the Clifford index. 

To state our theorem we must recall some more notation. Recall that if
$\xi \in H^1(C,T_C)$ then via cup product we get a map
$\xi:H^0(C,K_C) \to H^1(C, \cO_C)$ called the Kodaira-Spencer map.  
This induces a map
\begin{center}
\[\kappa:H^1(C,T_C) \to \mbox{Hom}(H^0(C,K_C),
H^1(C,\cO_C)) ).\] 
\end{center}
 Recalling that $\Phi_2 : C
\to \P^{3g-4} = \P(H^0(K_C^{\otimes 2})^\vee)$ we get a commutative
diagram:\\

$$
\xymatrix{
 &  \mathbb{P}^{g-1}=\mathbb{P}(H^0(K_C)^{\ast}) \ar[dr]^-{\nu}&\\
C \ar[ur]^-{\Phi_1} \ar[dr]^-{\Phi_2}
&&\mathbb{P}^{g^2-1}=\mathbb{P}(H^0(K_C)^{\otimes 2\ast})\\
&\mathbb{P}^{3g-4}=\mathbb{P}(H^0(K_C^{\otimes 2})^{\vee})
 \ar[ur]_-{i}&}
$$
Here $\nu$ is the Veronese map and $i$ is induced from the map
$H^0(K_C) \otimes H^0(K_C) \to H^0(K_C^{\otimes 2})$ which
is the Kodaira-Spencer map suitably dualized. In otherwords $i(\xi)$ is the 
symmetric matrix representing $\xi:H^0(K_C) \to
H^1(\cO_C)$ . Identifying \linebreak $ \Hom^{sym}(H^0(C,K_C),
H^1(C,\cO_C))$   with    $\sym^2(\Hd{K_C})$,
 we see that the maps $\kappa$ and
$i$ are the same map.  As long as $C$ is not hyperelliptic, the maps
$i$ and $\kappa$ are injective and we abuse notation and write $\xi$ when
we mean $\kappa(\xi)$ or $i$($\xi$).  We have
a filtration $0 \subset R^1 \subset R^2 \ldots \subset R^g$ of $\P^{3g-4}$
via $R^k(C) = R^k = \{ \xi \in H^1(T_C) \, | \, \mbox{rank}( \xi ) \leq
k \}$. Via Shiffer variations we can see that $\Phi_2(C) \subset
R^1(C)$ (Roughly $p \in C$ kills $H^0(K_C(-p))$ and  sends the 1
dimensional quotient to the point $\Phi_1(c) \in \P^{g-1} =
\P(H^1(\cO_C)))$ and hence that the $k$-secant variety $\mbox{Sec}^k(C)
\subset R^k(C)$. We may roughly restate Thm A as $R^0(C) = \emptyset
 \Leftrightarrow \mbox{Cliff}(C) > 0$ and Thm B as $\nu(\P^{g-1}) \cap
i(\P^{3g-4}) = i \Phi_2 (C) \Leftrightarrow R^1(C) = \mbox{Sec}^0(C)=C
\Leftrightarrow \mbox{Cliff}(C) > 1$. As such the generalization we
propose is 
\begin{thm}[Geometric Characterization of the Clifford Index]~\label{T:mthm}

 Let C be a smooth curve then, $\Sec^{j-1}(C) =
R^j(C)$ for $j< \mbox{Cliff}(C)$ and $\mbox{Sec}^j(C) \subsetneq R^j(C)$
for $ j \geq \mbox{Cliff}(C)$.

\end{thm}

This result is than  weaker, but connected to, Green's conjecture.
We will discuss the exact relationship in the last section of this thesis.
   Notice in the case $\cliff(C)=0$  our theorem merely
states that $\cliff(C) > 0 \Leftrightarrow H^o(K_C) \otimes
H^0(K_C) \imp H^0(K_C^{\otimes 2})$ is surjective and in the case
$\mbox{Cliff}(C) = 1$  we get only set theoretic and not scheme
theoretic results. The theorem only asserts a set theoretic equality.  Scheme
theoretic equality is true, but will be discussed later as it involves some different ideas.  \\

\noindent
{\bf Notation}

$C$ will always be a smooth curve of genus at least three 
 defined over an algebraically closed field.
 $D = \sum_i^d p_i$ will be a divisor defining a $g_d^r$ (i.e.
$h^0(C,\cO_C(D)) = r+1$). We will almost always use the convention
that $\P(V)$ is the space of lines in $V$.  That is because we will be
dealing with classes that naturally live in $ H^1$ of various line bundles.

The first order of business will be to define Shiffer variations 
and
Shiffer cohomology classes. Shiffer variations were classically defined
analytically as deformations that only disturbed the complex structure at
one  point $p \in C$. If $z$ is a local analytic coordinate near $p$ then
$\frac{1}{z} \frac{\partial}{\partial z}$ defines a C\v{e}ch 1 cocycle of
$T_C$ and hence a class in $H^1(T_C)$. Call it $\tau_p$. Notice $\tau_p$
depends
upon the choice of local parameter, but $\tau_p \in \P(H^1(T_C)) =
\P(H^0(K_C^{\otimes 2})^\vee)$ is independent of the choice of local
parameter. This is the Shiffer variation associated to $p \in C$. For
convenience we recast this definition in an algebraic mode.

\begin{defn}

Recall that $H^0(T_C(p)) = 0 \ \forall p \in C$ when $g \geq 2$. Hence in
the long exact sequence of cohomology associated to\\
\begin{center} $0 \to T_C \to
T_C(p) \to T_C(p)_{|p} \to 0$\end{center}
$\partial(H^0(T_C(p)_{|p})) \subset H^1(T_C)$ is a line and hence defines a
unique class in  $ \P(H^1(T_C(p)))$. This is the Shiffer variation
(associated to $p$).
\end{defn}

\noindent
{\bf \underline{Remark 1:}} The choice of a local parameter $z$ at $p$
amounts
to the choice of a generator $\frac{1}{z} \frac{\partial}{\partial z}$ in
$H^0(T_C(p)_{|p})$. When the choice of the local parameter is irrelevant
we may speak of the Shiffer variation $\t_p \in H^1(T_C)$.\\
{\bf \underline{Remark 2:}} The map $p \mapsto \tau_p \in \P(H^1(T_C)) =
\P(H^0(K_C^{\otimes 2})^\vee)$ is the bicanonical embedding.\\

We will need this same notion regarding classes in $H^1(\cO_C)$ so
we
note the following:\\
\begin{defn} Consider the exact sequence  
\begin{center}$0 \to \cO_C \to
\cO_C(p) \to \cO_C(p)_{|p} \to 0$.\end{center}
 The class
$\partial(H^0(\cO_C(p))) \subset \P(H^1(\cO_C))$ is the
Shiffer
cohomology class (associated to $p$) and is denoted by $\s_p$ (or $\s$ if
the choice of $p$ is irrelevant). 

\end{defn}

\noindent
{\bf \underline{Remark 3}:} As before, the assignment $p \mapsto \s_p \in
\P(H^0(K_C)^\ast)$ is the canonical embedding.

We wish to generalize the preceding definitions to reduced divisors. Let
$D
= \sum_{i=1}^d p_i$ with the $p_i$ distinct. \\
\begin{defn} $\langle \s_i \rangle \subset H^1(\cO_C)$ is the
vector space spanned by the $\s_{p_i} = \s_i$ i.e. pick local parameters
$z_i$ around $p_i$ and let $\s_i = \partial \left(\frac{1}{z_i}\right) \in
H^0(\cO_C(p)_{|p})$,  then $\langle \s_i \rangle = \left\{ \sum_{i=1}^d
a_i \s_i \, | \, a_i \in k \right\}$. 
\end{defn}
 Similarly let $\t_i = \partial
\left( \frac{1}{z_i} \frac{\partial}{\partial z_i} \right) \in
H^0(T_C(p)_{|p})$ then \\

\begin{defn}
 $T(D) = \left\{ \sum_{i=1}^d a_i \t_i \, | \, a_i
\in k
\right\}$. $T(D)$ is the set of Shiffer \linebreak variations supported on $D$.
\end{defn}
Notice with this notation
$$\langle \s_i \rangle = \partial(H^0(\cO_D(D))) \subset H^1(\cO_C)
\ \ and\ \ T(D) = \partial(H^0(T_C(D)_{|D})) \subset H^1(T_C)$$
Unless the degree of $D$ is large, and in particular if
$\deg(D) \leq 2g-3$, then $H^0(T_C(D))$ \linebreak $ =0$. 
Hence in this case,  $\dim T(D) = \deg D$.

Recall that any $\xi \in H^1(T_C)$ induces $\xi:H^0(K_C) \to
H^1(\cO_C)$. In terms of Shiffer variations we can describe the action
of $T(D)$ as follows:\\

\begin{lemma}\label{L:1.2}
 Pick local coordinates $z_i$ around $p_i$ and let
$\t_i
= \frac{1}{z_i} \frac{\partial}{\partial z_i} \in H^1(T_C)$, $\s_i =
\frac{1}{z_i} \in H^1(\cO_C)$. Suppose $\omega \in H^0(K_C)$ has a
local representation $f(z_i)dz_i$ with $f(0) = a_i$. Then, for $\t =
\sum_{i =1}^d \t_i$
$$ \t(\omega) = \sum_{i=1}^d a_i \s_i \in \langle \s_i \rangle \subset
H^1(\cO_C)$$
\end{lemma}
\begin{proof} Both \, $\s_i $ and \,  $ \t_i$ are defined in terms of
boundaries \, (Say $\s_i = \partial (\widetilde{\s_i}) \ and\ \t_i = \partial
(\widetilde{\t_i})$) $\partial: H^0(\cO_D(D)) \to H^1(\cO_C) \ and\
\partial:H^0(T_C(D)_{|D}) \to H^1(T_C)$. As such it is clear we have a
commutative diagram for any $\omega \in H^0(K_C)$ 

$$
\begin{CD}
H^0(T_C(D)_{|D}) @>{\vee \omega}>> H^0(\cO_D(D)) \\
@VV\partial V				@VV\partial V \\
H^1(T_C) @>{\cup \omega}>> H^1(\cO_C) \\
\end{CD} 
$$\\

\noindent
But $\t_i \cup \omega = \frac{1}{z_i} \frac{\partial}{\partial z_i}
\cup
f(z_i) dz_i = \frac{f(z_i)}{z_i} = \frac{a_i}{z_i} + $ holomorphic fnc. 
$f(z_i)/z_i = \frac{a_i}{z_i}$ in $H^0(\cO_{p_i}(p_i))$ so denoting
by
$\widetilde{\tau_i} = \frac{1}{z_i} \frac{\partial}{\partial z} \in
H^0(T_C(p_i)_{|p_i})$ \\
$\t_i(\omega) = \partial(\widetilde{\t_i} \vee
\omega ) = \partial \left( \frac{a_i}{z_i} \right) = a_i \s_i$. Since
$T_C(D)_{|D} \ and\ \cO_D(D)$ are skyscraper sheaves it follows that
$\t(\omega) = \sum_{i=1}^d \s_i(\omega) = \sum_{i=1}^d a_i \s_i$ as
claimed.
\end{proof}

\begin{cor} 
ker $(\t_i) = H^0(K_C(-p_i))$ and for 
$\t
\in
T(D)$ ker$(\t) \supset H^0(K_C(-D))$.
\end{cor}

\begin{proof}
 $\t_i(\omega) = a_i \s_i $ where $ \omega = f(z_i)
dz_i \
and\ f(0) = a_i$. Hence $\t_i(\omega) = 0 \Leftrightarrow a_i = 0
\Leftrightarrow \omega \in H^0(K_C(-p_i))$. If $\t = \sum b_i \t_i$
then
$$ker(\t) \supset \bigcap_{i=1}^d \ker(b_i \t_i) =
\bigcap_{i=1}^d \ker(\t_i) = \bigcap_{i=1}^d H^0(K_C(-p_i)) =
H^0(K_C(-D))f.$$

\end{proof}

\begin{defn}
 The rank of $\xi \in H^1(T_C)$ is the rank of $\xi:
H^0(K_C) \to H^1(\cO_C)$. $R^k(C) = \{ \xi \in H^1(T_C) \, | \,
\mbox{rank}(\xi) \leq k \}$.
\end{defn}

Since the rank of $\xi \in \Hi{T_C} $ doesn't change when $\xi$ is multiplied by a scalar,
$R^k(C)$ makes sense in $\PP^{3g-4}$.

\noindent
If $ \xi = \sum_{i=1}^k b_i\t_i$ where $\t_i$ are Shiffer variations then
$\mbox{rank}(\xi) \leq k$ since each $\t_i$ has rank 1; i.e.
$\Sec^{k-1}(C) \subset R^k(C)$ where $\mbox{Sec}^{k-1}(C)$ is the $k$ secant
variety to $\Phi_2(C) \subset \P^{3g-4}$.  $\mbox{Sec}^k(C)$ is the
closure in $\P^{3g-4}$  of all $k$ planes ${\langle \t_i
\rangle}^k_{i=0}$ where $\t_i$ are distinct Shiffer variations. To give a
complete description we need to describe what happens to $T(D)$ when $D$
has points of multiplicity greater than one. \\

\begin{defn}
 Suppose $E = kp$ is a divisor. Then we define
generalized Shiffer variations as follows: let  $\t_p^j =
\partial (\widetilde{\t}_p^j) \ 1 \leq j \leq k$ : where ${\widetilde{\t_p}}^j =
\frac{1}{z^j} \frac{\partial}{\partial z} \in H^0(T_C(kp)_{|kp})$. 
 If $D = \sum_{i=1}^{d'} k_ip_i$ then
$$T(D) = \left\{\sum_{i=1}^{d'} \sum_{j=1}^{k_i} b_{ij} \t_i^j \, | \,
\t_i^j \mbox{ are generalized variations and } b_{ij} \in \C \right\}$$
\end{defn}
From the definition it follows that $T(D)$ is
 the linear space spanned by the divisor D.
Thus the following is clear.
\\

\noindent
{\bf \underline{Claim}:} $\displaystyle{\mbox{Sec}^j(C) = \bigcup_{D \,
\in \,
\rm{\scriptstyle{Sym}}^j(C)} T(D)}$\\
{\bf \underline{Proof}:} We just mean the set theoretic statement that if
$\varphi_j:\sym^j(C) \to \linebreak  G(j-1, 3g-4)$ is the map which takes $j$
points $p_1, \ldots ,p_j \in \Phi_2(C)$ to the $j-1$ plane they span in
$\P^{3g-4}$ and if $U_j \subset \mbox{Sym}^j(C) \times \P^{3g-4} = \{(D,x)
\, | \, D = \sum_{i=1}^j p_i 
 \ and\ x \in \langle p_1 \ldots p_j \rangle \}$
then $\pi_2(U_j) = \mbox{Sec}^j(C)$ . Since $T(D)$ is the linear span of
the divisor D we are done. In fact we are only interested in the case 
where $d \leq 2g-2$ and hence $d$ points always define a $d-1$ plane.

The purpose of this thesis  is to show that the relationship between 
$\Sec^{j-1}(C)$ and $\ R^j(C)$ is controlled by the Clifford index. This
will follow from an analysis of Shiffer variations and classes. To that end
we need to understand for any $\t \in T(D)$ what the possible kernels and
images can be. Since $\mbox{ker}(\t) \supset H^0(K_C(-D))$ in any
event set $W_D = W = H^0(K_C) / H^0(K_C(-D))$ and $ S =
{\langle \s_i \rangle}_{p_i \in D}$ so $\t:W \to S$. Note $\dim(W) = g -
(g - d + r) = d - r$ since $D$ defines a $g^r_d$. We calculate
$\dim(S)$:\\

\begin{lemma}

\begin{itemize}
\item[i)] $\dim(S) = d-r$ 
\item[ii)] $\sum a_i \s_i = 0 \Leftrightarrow \exists f \in H^0(\cO_C(D))$
s.t. in the same local coordinates $z_i$ used to define $\s_i$ one has
$f(z_i) = a_i/z_i$ + holomorphic function.
\end{itemize}
\end{lemma}
\begin{proof}

 Consider the exact sequence 
$$ 0 \to \cO_C \to \cO_C(D) \to \cO_D(D) \to 0 .$$
It gives rise to:
$$ 0 \to H^0(\cO_C) \to H^0(\cO_C(D)) \stackrel{\rho}{\to} H^0(\cO_D(D))
\stackrel{\partial}{\to} H^1(\cO_C)$$
since $\partial (H^0(\cO_D(D))) = \langle \s_i \rangle$ so i) follows
from
$h^0(\cO_D(D)) = d$ and$\dim(\mbox{Im}(\rho )) = r$ ($D$ is a $g_d^r$). 
 ii) This is just an explication of i). Choosing local coordinates $z_i$
around $p_i$ and identifying $H^0(\cO_{p_i}(p_i))$ with $\frac{1}{z_i}$.
Then if $f = \frac{a_i}{z_i}$ + holomorphic function near $p_i$ then $\rho(f) =
\bigoplus_{i \in D} \frac{a_i}{z_i}$ so $\partial(\rho(f)) = 0 = \, \sum
a_i \s_i$.
\end{proof}

\noindent
{\bf \underline{Remark}:} If for $D = \sum_{i=1}^{d'} k_i p_i$ with $\sum
k_i =
d$ and we define $\widetilde{\s}_i^j \ \ 1 \leq j \leq k_i$ to be the class
of $\frac{1}{z_i^j} \in \cO_{k_i p}(k_i p) \ \ \s_i^j =
\partial(\widetilde{\s}^j_i)$ then the lemma extends in a straight forward
manner to non-reduced divisors.\\

\noindent
{\bf Main Result.}

The main theorem of this section is Thereom \ref{T:mthm} which
characterizes the rank filtration in terms of the Clifford index.
I recall the  statement.\\

\noindent
{\bf \underline{Main Theorem}} (Theorem \ref{T:mthm})

$\mbox{Sec}^j(C) = R^j(C)$ for $ j < \mbox{Cliff}(C)$

$\mbox{Sec}^j(C) \subsetneq R^j(C)$ for $j \geq \mbox{Cliff}(C)$.\\

This theorem is a geometric characterization of the Clifford index.
Griffiths \cite{Gri} has commented upon the importance of understanding the rank
filtration in terms of doing Hodge theory, but perhaps there is more to be
said algebraically.

The main tool used is a theorem that characterizes the possible ranks of a
generalized Shiffer variation supported on $D$ in terms of data about $D$,
in particular the Clifford index $d-2r$ of $D$. A calculation of the
dimension of $\mbox{Sec}^j(C)$ which allows one to bound the number of
Shiffer variations needed to express any element of $H^1(T_C)$ finishes
the proof. The relationship between the rank filtration and the Clifford
index is governed by:\\

\begin{thm} \label{T:2}
 Let $\t = \sum_{p_i \in D} a_i \t_i \, \in \,
T(D) \
a_i \neq 0$ be a generalized Shiffer variation. Then:
\begin{itemize}
\item[i)] $d - 2r \leq \mbox{rank}(\t) \leq d -r$
\item[ii)] The upper bound is always achieved and the lower bound is
achieved if $D \ and \linebreak \ K_C(-D)$ are base point free.
\end{itemize} 
\end{thm}

\noindent
{\bf \underline{Remarks}:} Notice that if $D$ computes the Clifford index
of $C$
then $D$ satisfies ii). Also note the choice of local coordinates $z_i$ is
irrelevant as different choices of $z_i$ just scale $\s_i $ and $ \t_i$
differently.\\

\begin{proof}
 Since $\t: W \to S $,\,$\rk(\t) \leq
\dim(S)
= d-r$. Let $p_1, \ldots, p_{d-r} \in D$ be such that  $h^0(\cO_C(p_1, \ldots ,
p_{d-r})) = 1$. That is to say that the points
 $p_1, \ldots, p_{d-r}$ are linearly independent.
 One can check easily by induction that such points exist.
If we take $\t = \sum_{i=1}^{d-r} \t_i$ then $\mbox{rank} \, \t =
d-r$ because if $ \omega \in \H{K_C} \, $ then $\t(\omega) =
\sum_{i=1}^{d-r} b_i \s_i$ where $b_i \in k$ is the value of $\omega$
at $z_i$, that is to say that locally $ \omega = f_i(z_i)dz_i $ where
$ f_i(0)= b_i$. Thus $\t(\omega) = 0 $ if and only if $b_i = 0$ for all $i$,
since the $\s_i$ were constructed to be linearly independent.
Hence $\t =
\sum_{i=1}^{d-r} \t_i$ achieves the upper bound. 

Next we show the lower
bound $d-2r \leq \mbox{rank}(\t)$. Recall that the action of a
Shiffer variation $\t$ on one forms can be calculated by considering the
action of $\widetilde{\t} \in H^0(T_C(D)_{| D})$ representing $\t$. In
fact we have the following commutative diagram.\\

$$
\xymatrix{
& & H^0(\cO_C(D)) \ar[d] \\
H^0(K_C)/H^0(K_C(-D)) = W \ar[r]^-{r} \ar[drr]^-{\t} &
H^0(\omega_{C|D}) \ar[r]^-{\widetilde{\t}} & H^0(\cO_D(D))
\ar[d]^-{\partial} \\
& & H^1(\cO_C) }
$$

Here $r$ is the restriction map and $\widetilde{\t}$ is such that
$\partial(\widetilde{\t}) = \t$. The key point is if $D = \sum_{i=1}^j k_i
p_i $ and $ \widetilde{\t} = \sum_{i=1}^j \sum_{l=1}^{k_i} b_{il} \t_i^l$
then $\widetilde{\t}$ is an isomorphism if and only if $b_{ik_i} \neq 0 \ \forall i$
; i.e. the
highest pole order terms are nonzero. This means that all coefficients are
non-zero if $D$ is reduced. To see this its clearly enough to check at any
point $p_i$ that the map $H^0({K_C}_{|k_i p})
\stackrel{\widetilde{\t}_i}{\to} H^0(\cO_{k_ip} (k_ip))$ is an isomorphism
$\iff b_{ik_i} \neq 0$. Working in local coordinates since $z_i^j \t_i^k =
\t_i^{k-j}$ one sees that $\mbox{rank} \, \t_i^j = j \ and\ \mbox{Im} \,
\t_i^j = z_i \mbox{Im}(\t_i^{j+1})$ so  $\sum_{j=1}^{k_i} b_{ij}\t_i^j$
has
$\mbox{rank} \, m \iff b_{i_m}  \neq 0$ and $b_{i_j} = 0 \ j>m$ and hence
$\mbox{rank} \, \widetilde{\t_i} = k_i \iff b_{ik_i} \neq 0$.

Now returning to our diagram, the map $\widetilde{\t}$ is an isomorphism,
and $r$ is injective, so $\mbox{ker}(\t) = \mbox{ker}(\partial \circ
\widetilde{\t}
\circ r) = \mbox{ker}(\partial_{| \rm{\scriptstyle{Im}}(\widetilde{\t}
\circ r)})$.
Since  $\dim(\mbox{ker} \, \partial) = r \geq
\dim(\mbox{ker}(\partial_{|\rm{\scriptstyle{Im}}(\widetilde{\t} \circ
r)}))$ so $\dim \, \mbox{ker}(\t) \leq r$ and hence $\mbox{rank} \, \t =
\dim(W) - \dim(\mbox{ker}(\t)) \geq (d-r)-r = d-2r$. 

To finish we must
show that if $K_C(-D)$ is base point free there exists a $\t$ with
$\mbox{rank}(\t) = d-2r$. Because  $D$ moves and is base
point free, ($h^0(\cO(D)) = 1$ is the 
case $r =0$) we may
find $D'$ linearly equivalent to $D$ and reduced, $D' = \sum_{i=1}^d p_i$
where $p_i \neq p_j$ for $i\neq j$. We may set $D=D'$ since we are only interested in
existence. Now since $K_C(-D)$ is base point free there exists 
 $\omega
\in H^o(K_C(-D))$ such that  $\omega$ has simple zeroes at $p_i$ 
for all $i$.
 Pick local coordinates s.t. $z_i$ on $U_i \ni p_i'$ with $\omega(z_i)
=
z_idz_i$ and set $\s_i = \partial\left( \frac{1}{z_i} \right) \in
H^1(\cO_C) $ and $ \t_i = \partial \left( \frac{1}{z_i}
\frac{\partial}{\partial z_i} \right) \in H^1(T_C)$. Let $f_1 \ldots f_r$
be a basis for $H^0(\cO_C(D)) / H^0(\cO_C)$. From our choices
we
have $f_i \omega = \eta_i \in H^0(K_C) \ and \ \t_i(\eta_j) =
a_i^{(j)}$ where $f_{j|U_i} = a_i^{(j)}/z_i$+hol. function.  Note that $\eta_i$
are linearly independent elements of $H^0(K_C)$ because if
$\sum_{j=1}^r R_j \eta_j = 0$ for $R_j \in k $ then $\left( \sum_{j=1}^r
R_j f_j \right) \omega = 0$. But $\omega$ is a holomorphic one form so
this can only happen if $\sum_{j=1}^n R_j f_j = 0$ which contradicts the
linear independence of $f_j$. Now let $\t = \sum_{i=1}^d \t_i$.
$\t(\eta_j) = \sum_{i=1}^d \t_i(\eta_j) = \sum_{i=1}^d a_i^{(j)} \s_i = 0$
(because  $\t_i(\eta_j)= a_i^{(j)}$ by Lemma  \ref{L:1.2}). 
 Hence we have exhibited $r$ linearly independent elements
of $\mbox{ker} \, \t$ i.e. $\mbox{rank}(\t) \leq d -2r$ hence
$\mbox{rank}(\t) = d-2r. $
\end{proof}

\noindent
{\bf Remark:} The construction of the second part is related to classical
ideas
about constructing quadrics containing the canonical curve. In fact
general $\omega \in H^0(K_C(-D))$ have simple zeroes on $D$ and if
${\omega}' \ and \ \eta_i'$ are another such choice then $\eta_i{\omega}' -
\eta_i'\omega \in H^0(K_C)^{\otimes 2}$ is a quadric containing $C$
i.e. $\eta_i {\omega}' = f_i \omega \cdot {\omega}' = f_i {\omega}' \cdot
\omega = \eta_i' \omega$ so $\eta_i {\omega}' - \eta_i' \omega \in
\mbox{ker}(H^1(K_C)^{\otimes 2} \to H^0(K_C^{\otimes 2}))$.

Clearly this is close to Theorem \ref{T:mthm}. We need to see only that any $\t \in
H^1(T_C)$ can be written as an element of $T(D)$ with $D$ special (or
degree $D$ small). This is a consequence of a calculation of the dimension
of the secant variety. The fact that the secant variety of a curve is non-degenerate is well-known.
We include a proof for completenes.  The theorem and the proof have nothing
to do with the dimension of the projective space involved.  Nonetheless for concreteness
we consider only the case of $\PP(V) = \PP(\Hi{T_C})     $\\

\begin{thm}\label{T:secdim}
 $\dim(\mbox{Sec}^j(C)) = \min(2j+1, 3g-4)$ \\
\end{thm}

\begin{proof}

 Recall that for $j \leq 2g-3 $ and any divisor $D$ of $\deg(D)= j$
  we have $ \ H^0(T_C(D)) = 0$ . Hence, for
 any divisor $D \in \mbox{Sym}^{j-1}(C)$,$ \ $ $D$ spans a $j-1$ plane in
$\P^{3g-4}$ which we will denote by $\langle D \rangle \subset \P^{3g-4}$.
Thus the map $D \mapsto \langle D \rangle $ is a map $\varphi_j :
\mbox{Sym}^{j-1}(C) \to G(j-1,3g-4)$. Let $U_j \subset \mbox{Sym}^{j-1}(C) \times
\P^{3g-4} = \{ (D,p) \, | \, p \in \langle D \rangle \subset \P^{3g-4} \}$
be the incidence correspondence relation. Notice $\p_1:U_j \to
\mbox{Sym}^{j-1}(C)$ is smooth with fibers $\P^{j-1} \ and \ \p_2(U_j) =
\sec^j(C)$ so its enough to show $\p_2$ is quasifinite.

To show $\p_2$ is quasifinite it is enough to show:\\
\\
\noindent
$(\ast) \ \exists U \subset \mbox{Sym}^j(C)$ open and non-trivial s.t. for $D
\in U \ \exists$ an open subset $V_D \subset \mbox{Sym}^j(C)$ such that 
$D \in V_D$ and
$ \forall E \in V_D \ \ H^0(T_C(D+E)) = 0.$\\

Statement $(\ast)$ may be translated as saying for general $D \in
\mbox{Sym}^j(C)$ and  any nearby $E \subset \mbox{Sym}^j(C)  \langle D
\rangle \cap \langle E \rangle = \emptyset$.  But
by $(\ast)$ if $D_1$ and $D_2 \in U$ are distinct divisors then for $p_i 
\in \pi_i^{-1}(D_i)$ for $i=1,2$ then $\pi_2(p_1) \neq \pi_2(p_2)$.  Hence
$\pi_2$ is quasifinite on the open set $\pi_1^{-1}(U)$.
However $(\ast)$ is clear.   Since $j \leq \frac{3g-5}{2}$,
$\deg(T_C(E+D)) \leq g-3$ for $D,E \in \sym^j(C)$ and the general
divisor of degree $g-3$ is not effective. By semi-continuity for 
general $D,E$ $\h{\cO_C(D+E)} = 0$ once it is true for one special 
$D,E$.
\end{proof}

Finally we will show how Theorem \ref{T:secdim} implies 
Theorem \ref{T:mthm}.\\

\begin{proof}
 By Theorem \ref{T:secdim}, if $(2j+1) \geq
3g-4$ or equivalently that $j \geq \frac{3g-5}{2}$ then 
$\sec^{j}(C)= \PP(\Hi{T_C})$
and hence every Shiffer variation can be written as
the sum of  at most 
$\frac{3g-5}{2}$  rank one variations.
In other words every $\tau \in \Hi{T_C}$ is in a T(D) with 
$\deg(D) \leq \frac{3g-5}{2}$.  
If $D$ is eligible to compute $\cliff(C) $, then by Theorem 
\ref{T:2}, any $\tau \in T(D)$ has rank $\geq \cliff(C)$.  
If $D$ doesn't move, that is $\h{\cO_C(D)}=1$, 
then $\rk(\tau) = \deg(\tau)$ for any $ \tau \in T(D)$ 
so we may assume that $D$ moves and hence, since D is
not eligible to compute $\cliff(C)$ 
 that $\hi{\cO_CC(D)} =1$. However 
$\cliff(D) = \cliff(K_C(-D)) =  \deg(K_C(-D))$, since 
$\hi{\cO_C(D)} =1$  implies  $\h{K_C(-D) } = 1$.  However we have 
$deg(D) \leq \frac{3g-5}{2}$ and hence  $ \deg(K_C(-D) 
\geq \frac{g+1}{2}$. But  $K_C(-D)$ doesn't move, so
we have that $\rk(\tau) = \deg(K_C(-D)$ for all $\tau \in
T(D)$ that have all non-zero coefficients.  
But we always have $\cliff(C) \geq \frac{g+1}{2}$.  Thus
in this case too, if $\sec^{j-1}(C) \subsetneq R^j(C) $ then
$j \geq \cliff(C)$. 
\end{proof}

\section[Geometric Riemann-Roch]{Geometric Riemann-Roch and the Definition of the Clifford Index of a General Line Bundle}\label{S:GRR}

In the previous section we have proved that the Clifford index of a curve can be
characterized in terms of the geometry of the curve, specifically the geometry of
the bicanonical embedding. The purpose of this section is to generalize these
ideas to a much wider class of line bundles.  The Clifford index can be given
a more geometric interpetation which allows one to make sense of the notion of 
a Clifford index for any very ample line bundle.  In fact the definition makes 
sense for any line bundle, but it doesn't seem to be useful unless the line bundle
is very ample.

The main result  we are aiming to prove is  one that compares secant varieties of 
curves to certain rank loci for non special line bundles of large degree. 
While secant varieties have nice geometric properties, rank loci have the 
important property that their equations are (by definition!) determinants of
a given degree.  For example to say that a curve in some given embedding is a
rank one locus is  to say that the curve is defined by the vanishing
of two by two minors of some matrix. This  brings to the forefront the
issue of whether the two scheme structures defined by the secant structure 
and the rank loci structure coincide.  We show that this is true in a large
range of circumstances.  Finally we relate this to earlier work of Eisenbud, Koh, 
and Stillman on determinantal presentations of curves and their secant varieites.

The key idea is as follows.  First use Geometric Riemann-Roch to define a general Clifford
index. For a very ample line bundle L interpet $L^{\otimes 2} $ as giving an embedding of C in a 
space of matrices. Finally  use the Clifford Index to bound from below the rank of a matrix.
This allows one to say that up to a given integer $d$, $\sec^d(C)$ is set-theoretically 
a determinantal  locus.  Scheme theoretic equality then comes from a different  argument.

Let $C$ be a smooth curve, $L$ a very ample line bundle. Set $P \equiv \PP(
\Hd{C,L})$ and let $D$ be a divisor of degree $d$.

\begin{defn} 
  
  Denote by $\overline{D}$ the span of $D$ in $P$. If $D$ has no
  multiple points this is clear. In general if $V \subset \H{C,L}$ is of
  codimension $m$ the $\H{C,L} \to \H{C,L}/V$ determines a $m-1$ dimensional
  subspace.  $\overline{D}$ is the space corresponding to $V=\H{L(-D)}$. 

\end{defn}

\begin{thm}[Geometric Riemann-Roch]~

  \begin{enumerate}

    \item[(i)] $\overline{D} \simeq \PP^{d-1-r}$ if and only $\h{L(-D)} =
      \h{L} - d + r$; i.e.  $D$ imposes $d-r$ conditions on $L$.

    \item[(ii)] If $s \in \H{\OO_C(D)}$ is a section which vanishes on $D$
      (i.e. $"1" \in \H{\OO_C} \hookrightarrow \H{\OO_C(D)}$) then under the
      natural multiplication map $\OO_C(D) \otimes L(-D) \to L$ we can
      identify $I(\overline{D}) = \set{ \xi \in \H{L} \,|\, \xi
      |_{\overline{D}} = 0}$ with $s \otimes \H{L(-D)}$.

  \end{enumerate}

  \begin{proof}~
    
    \begin{enumerate}

      \item[(i)] This is the definition since $\h{L(-D)} = \h{L} - d + r$ means
	that codimension of $\H{L(-D)}$ in $\H{L}$ is $d - r$.

      \item[(ii)] We have an exact sequence 

	\[ 0 \to \OO_C \stackrel{s}{\to} \OO_C(D) \to \OO_C(D)|_D \to 0 .\] 

	Tensoring this with $L(-D)$ and taking global sections give us (ii).

    \end{enumerate}

  \end{proof}

\end{thm}

If $L = K_C$ then the criterion that a divisor $D$ be eligible to compute
$\cliff(C)$ is that i) $\h{\OO_C(D)} \ge 2$ and ii)  $\hi{\OO_C(D)} \ge 2$.
By Serre Duality ii) is equivalent to 
 $\h{K_C(-D)} \ge 2$. The
first condition is that the points of $D$ do not impose independent conditions
on $\PP = \PP \left( \Hd{C,L} \right)$, and the second condtion is that the span of 
$D$ has codimension at least 2. This second condition is necessary because if
we take any hyperplane $H$, $H \cap C$ is a divisor of $\deg(L)$ and $\deg(L) > \h{L}$.
Therefore we can always find divisors supported on hyperplanes that fail to impose
 independent conditions. In other words, such divisors aren't really special.

Thus we may rephrase the definition of the Clifford index as \[\cliff(C) =
\min \set{d - 2r \, | \, D \text{ is a divisor of degree } d} \] where D is 
 such that $\overline{D} = \PP^{d-1-r}$ with $r > 0$ and $\overline{D}$ is not a
hyperplane in $\PP^{g-1}$. This suggests the following definitions for L an arbitrary
very ample line bundle.

\begin{defn}

  $r_L(D) = r \ge 0$ is the integer such that $\overline{D} = \PP^{d-1-r}$.

\end{defn}

\begin{defn}

  $\cliff(L,D) = d - 2 r_L(D)$. 
  
\end{defn}

And finally,

\begin{defn}
  
  $\cliff(L,C) = \min \{\cliff(L,D)\, | \, r_L(D) > 0 \text{ and the span of } D \\
\text{is of codimension two or greater}\}$ 
\end{defn}

{\bf Remark} It is a standard notation  that $\cliff(D) =
\cliff(K_C,D)$. This definition is specific for curves. As far as  I can
tell, the important point is that $D$ consists of points not divisors. We now
record the basic properties of $r_L(D)$ and $\cliff(L,D)$.

\begin{lemma}~\label{L:2.2}
  
  \begin{enumerate}

    \item[(i)]  $r_L(D) = \codim \left( \im \left( \H{L} \to \H{L,D} \right)
      \right)$.

    \item[(ii)] $r_L(D) = \h{ K_C \otimes L^{-1}(D)} - \h{ K_C \otimes L^{-1}
      }$.

    \item[(iii)] $\cliff(L,D) = \cliff(L(-D)) - \cliff(L)$.

    \item[(iv)] $\cliff(L,D) = \cliff(L, L^{\otimes 2} \otimes K_C^{-1}(-D) )$.
  
  \end{enumerate}

\end{lemma}

\begin{proof}

  $r_L(D)$ is defined by $\h{L(-D)} = \h{L} - d + r_L(D)$. Using \[ 0 \to \OO(-D)
  \to L \to L|_D \to 0 .\]

  We get a long exact sequence \[ 0 \to \H{L(-D)} \to \H{L} \to \H{L|_D} \to
  \Hi{L(-D)} \to \Hi{L} \to 0 \] (since $\Hi{L|_D} = 0$) and  $\h{L|_D} = d$, we
  see that $\dim \left( \im \left( \H{L} \to \H{L|_D} \right) \right) \linebreak = d -
  r_L(D)$ which is (i).

  \begin{enumerate}

  \item[(ii)] From (i) we get $r_L(D) = \hi{L(-D)} - \hi{L}$. Applying Serre
    duality the result follows.

  \item[(iii)] $\cliff(L,D) = d - 2 r_L(D) = \left( d - r_L(D) \right) -
    r_L(D)$, which by (i) and (ii) $= \left( \h{L} - \h{L(-D)}) - (\hi{L(-D)} -
    \h{L} \right) = \left( \h{L} + \hi{L} \right) - \linebreak  \left( \h{L(-D)} + 
    \hi{L(-D)} \right)$. (iii) now follows from the well known

  \end{enumerate}

  \begin{lemma} 
    
    $\h{L} + \hi{L} = g + 1 - \cliff(L)$.

  \end{lemma}
    
  \begin{proof}

    If $\h{L} = r + 1$, $\hi{L} = \h{K_C \otimes L^{-1}} = g -
    \ell + r$, where $\ell = \deg(L)$, so $\h{L} + \hi{L} = r + 1 + g - \ell +
    r = g + 1 - (\ell - 2r) = g + 1 - \cliff(L)$.
  
  \end{proof}

  \begin{enumerate}

    \item[(iv)] By (iii), $\cliff(L,D) = \cliff(L(-D)) - \cliff(L) =
      \cliff(K_c \otimes L^{-1}(D)) - \cliff(L)$, since $\cliff(L) = \cliff(K_C
      \otimes L^{-1})$.  Now $K_C \otimes L^{-1}(D) = L\otimes\left( L^{\otimes 2}
      \otimes K_C(-D) \right)^{-1}$ and the result follows.

  \end{enumerate}

\end{proof}

We now compute some examples of the Clifford index $\cliff(L,C)$ to
give a feeling for the ideas involved and to show how it is related to the
geometry of the embedding of $C$ in $\PP(\Hd{L})$. A key idea is that if
$\hi{L} = 0$, then a divisor $D$ fails to impose independent conditions 
on the line bundle $L$ 
$\Longleftrightarrow$ $\hi{L(-D)} > 0$, which by Serre duality means $\h{K_C
\otimes L^{-1}(D)} > 0$.  This  means that there exists an injection $\OO
\stackrel{i}{\hookrightarrow} K_C \otimes L^{-1}(D)$. If $i$ vanishes
 on the divisor $E$ then $L(-D) \cong K_C(-E)$, that is to say,  $L \cong
K_C(D-E)$. We exploit this representation as well as information about how
special $E$ can be for $K_C$ -- which is given to us by Clifford's Theorem.

We first consider the case when $\deg(L) = 2g - 2$. Depending on how ``far
away'' $L$ is from $K_C$, its behavior becomes more ``very ample''. ``Far
away'' means more general in a sense to be explained below. I am sure that 
most, if not all, of this material is well known to the experts.  I could not  find  references
to these specific statements, so I am including the proofs.

\begin{thm} 

Suppose $deg(L)$ = $2g-2$ and $L \ne K_C$

  \begin{enumerate}

    \item[i.] $L$ is base-point free unless $L = K_C(p-q)$, $p, q \in C$.
      $K_C(p-q)$ has a unique base point unless $C$ is hyperelliptic.

    \item[ii.] Suppose $L$ is base-point free, then $L$ is very ample unless
      $L = K_C(D-E)$, $D, E \in \sym^2(C)$.

    \item[iii.] Suppose $L$ is very ample, then $C$ is not defined by
      quadrics if $L = K_C(D_1 - D_2)$, where $\deg(D_i) = 3$ and $D_i$ is
      general. That is, $L$ has a 3-secant line.

    \end{enumerate}

\end{thm}

\begin{proof} 

  \begin{enumerate}

    \item[i.] Suppose $p \in C$ is a base point. Then $\H{C,L} = \H{C,L(-p)}$
      and hence $\hi{L(-p)} = 1$. As mentioned above, this means $L(-p) =
      K_C(-E)$ (with $E$ effective). It follows that $\deg(E) = 1$, so $L(-p)
      = K_C(-q)$ and $L = K_C(p-q)$. 

      For any $p, q \in C$, $K_C(p-q)$ has a base point at $p$. Furthermore,
      if $p_1 \ne p$ is another base point, then $K_C(p-q-p_1) = K_C(-q_1)$,
      so $\OO_C \cong \OO_C(p-q+q_1-p_1)$ for some choice of $p_1 \in C$  which means that there exists a
      function on $C$ with two poles, i.e. $C$ is hyperellipic.

    \item[ii.] If $L$ is not very ample, there is  a divisor $D$ of degree
      2, such that $\h{L(-D)} > \h{L} - 2$. This is just the usual criterion 
      that a divisor is very ample if and only if it separates any two points
      , including infinitely near points, on $C$ (see \cite{ha} p.152) . Since $L$ is base-point free
      if $D = p_1+p_2$, \[\h{L(-D)} = \h{L(-p_1)} = \h{L(-p_2)} = g - 2,\]
      hence $\hi{L(-D)} = 1$  and hence  $L(-D) = K_C(-E)$ with $\deg(E) = 2$
      and $E$ effective, so $L = K_C(D-E)$. 

    \item[iii.] The condition that $L = K_C(D_1 - D_2)$ with $D_i$ general
      is equivalent to a line which intersects $C$ in 3 points. This is because
      $\h{L} = g - 1$ and \[\h{L(-D_1)} = \h{K_C(-D_2)} = g - 3\] (since $D_2$
      is general). So denoting $\overline{D}_1$ the span of $D_1$ in
      $\PP(\Hd{C,L})$ we see that \[\h{\OO_{\overline{D}_1}(1)} = \h{L} -
      \h{L(-D_1)} = 2,\] so $D_1$ is a line. Finally, any quadric containing
      $C$ contains $\overline{D}_1$, since if $Q$ is any such quadric
      $\#(\overline{D}_1 \cap Q) \ge 3$. 

  \end{enumerate}

\end{proof}

\textbf{Remarks}: 

\begin{enumerate}

  \item[i.] The map $C^d \times C^d \to Pic^{2d}(C)$ given by $(E,D) \mapsto E -
    D$ has $2d$ dimensional image in $P_{2d}(C)$ (\cite{ACGH}). So for $g \ge 3$,  a
    general line bundle is base-point free, and for $g \ge 5$,  a general bundle
    of degree $2g-2$ is very ample. This is because the above theorem shows that 
the  space of line bundles with a base point is two dimensional  and 
 the space of non ample line bundles is of dimension 4.

  \item[ii] One can check that for an embedding $C \hookrightarrow
    \PP(\H{C,L})$, where $\deg(L) = 2g - 2$ and $C$ is general, $C$ is defined by 
quadrics if and  only if there does not exist a 3-pointed line. Once one has defined
Shiffer variations for an arbitrary line bundle $L$, then one sees that, just as
in the case of the canonical embedding, that the rank one locus, consisting of
 divisors $D$, such that $\cliff(L,D)=1$ corresponds to the intersection of
all the quadrics through $C$.  Because  $\cliff(L)=2$ in this case one sees by 
Lemma \ref{L:2.2} part  \emph{iii} that a divisor with $\cliff(L,D)$ corresponds to
a divisor $E$ with $\cliff(E)=3$.  On a general curve the only such divisors are 
non-special and then $L \cong K_C(D-E)$ where $L$ and $E$ are divisors of degree 3.
Then on $L$, $D$ corresponds to a 3 pointed line.

Notice how this corresponds to the case of
    $L = K_C$. For the canonical embedding a non-hyperelliptic  curve is the
    intersection of the quadrics containing $C$ if and only if $C$ is neither
    a trigonal or a plane quintic. Trigonal means that there exists a divisor
    $D$ of degree 3 with $\h{\OO_C(D)} = 2$. That is, $D$ spans a line in
    $\PP(\Hd{C,K_C})$.

\end{enumerate}

Before moving on to the the Clifford index of line bundles of degree $\ge 2g-2$,
I would like to give one more application of this idea. 

One says a divisor $D$ of degree $d$ defines a $d$-pointed $j$ secant if $D$
is of degree $d$ and spans a projective space of dimension $j$. By Geometric
Riemann-Roch, $D$ spans a $d - 1 - r$ plane with $r \ge 0$. A general set of
points has $ r = 0$. One can ask what is the smallest $d$ such that there
exists a $d$-pointed $d - 1 - r$ plane. \cite{ACGH}  gives the formula that there
exists a $d$-pointed $d - 1 - r$ plane if: $d \ge r (\h{L} - d + r)$
This holds for any very ample $L$ such that $\hi{L} = 0$. The techniques used in \cite{ACGH} are  sophisticated.

For $r = 1$ and $\deg(L) = 2g - 2$, we can prove this very simply. In
this case, the result is

\begin{prop} If $d \ge \h{L} - d + 1$, then there exists a $d$-pointed $d - 2$
  plane.

\end{prop}

\begin{proof}

  Rearranging terms,  we need to show that if $2d \ge \h{L} + 1$, then there exists a
  divisor $D$ of degree $d$, such that $\hi{L(-D)} > 0$. Using the same
  argument of \cite{ACGH} as before, if $g$ is even,  the map $\sym^{\frac{g}{2}}(C)
  \times \sym^{\frac{g}{2}}(C) \to Pic^g(C)$, or if $g$ is odd, the map
  $\sym^{\frac{g+1}{2}}(C) \times \sym^{\frac{g+1}{2}}(C) \to Pic^g(C)$ ($D, E$)
  $\mapsto D - E - p_0$ for a fixed $p_0$ are surjective. In either case we
  get that for any divisor $L_0$ of degree 0, we may write  $L_0 =
  \OO_C(D - E)$, with $\deg(D) \le \frac{g+1}{2}$. Since $L = K_C \otimes
  L_0$ for some $L_0$, $L(-D) = K_C(-E)$, and hence for the embedding given by $L$,  $D$ is
  a $d$-secant $d - 2$ plane.

\end{proof}

We now consider the case where $\deg(L) \ge 2g+1 $. We distinguish between the case when L contains
$K_C$ as a subsheaf and when L doesn't.  Of course if deg(L) $\ge 3g-2 $, then the canonical bundle will
always be a subsheaf.  These results are used in calculating the bounds for which one can say that curves 
of high degree,  and certain of their secant varieties,  are determinantally defined.

\begin{thm}\label{T:clifford_calc_1}
  
  If $L = K_C(D)$ with $D$ effective of degree $= d \ge2$, then $\cliff(L,C) = d - 2$.
  Unless $C$ is hyperelliptic, $D$ uniquely achieves this bound.

\end{thm}

\begin{proof}

  $\h{L} = g - 1 + d$ and $\h{L(-D)} = \h{K_C} = g$, so $D$ spans a $d$ secant
  $d-2$ plane, that is,  $r_L(D) = 1$ and  $\cliff(L,D) = d - 2$. If $E$ satisfies
  $r_L(E) > 0$ then $\hi{L(-E)} > 0$, i.e. $L(-E) = K_C(-E_1)$. Let $e =
  \deg(E)$, $e_1 = \deg(E_1)$, and set  $\h{\OO_C(E_1)} = r_1 + 1$. Then $\hi{L(-E)}=
  \hi{K_C(-E_1)} = g - e_1 + r_1=g+d-e+r_1 =g+d-1-e+(r_1+1)$, and hence $r_L(E) = r_1 + 1$. Hence we get
  $\cliff(L,E) = \deg(E) - 2r_L(E)$. $d + e_1 - 2(r_1 + 1) = d - 2 + e_1 -
  2r_1 = d - 2 + \cliff(E_1)$. By Clifford's theorem: $\cliff(E_1) \ge 0$, with
  equality $\Leftrightarrow E_1 = 0$, $K_C$, or $C$ is hyperelliptic and $E_1 =
  ng_2^1$ a multiple of the  $g_2^1$. $E_1 = 0$ means $E = D$; $E_1 = K_C$ means $E
  = L$. Clearly $E = D + n g_2^1$ will give a divisor with $\cliff(L,E) = d -
  2$.

\end{proof}

\begin{thm}\label{T:clifford_calc_2}
  
  Suppose $\deg(L) = 2g - 2 + d$, $L = K_C(D)$, $\deg(D)=d > 1$ and $\h{\cO_C(D)}=0$ , then $\cliff(C,D) \geq  d - 1$ unless $C$ is hyperelliptic, in which
  case $\cliff(C,L) = d - 2$ is achieved.

\end{thm}

\begin{proof}

  If $r_L(E) > 0$, then $\Hi{L(-E)} > 0$, so $L(-E) = K_C(-E_1)$, with $E_1$
  effective. Again let $\deg(E) = e$, $\deg(E_1) = e_1$, and suppose
  $\h{\OO_C(E_1)} = r_1 + 1$, then $r_L(E) = r_1 + 1$ and $\cliff(L,E) = e -
  2r_L(E) = d + e_1 - 2(r_1 + 1) = d - 2 + (e_1 - 2r_1)$. Again by Clifford's
  theorem, $e - 2r_1 > 0$ unless $C$ is hyperelliptic. If $C$ is 
hyperelliptic, then taking $L =K_C(D-E)$ with $D$ general and $E$ a 
multiple of the $g_2^1$ will produce $L$ with $\cliff(C,L) = d-2$.

\end{proof}

\begin{cor}\label{C:compute_cliff}

  If $L = K_C(D)$, $\deg(D) = d > 0$, then $\cliff(L,C) \ge d - 2$. 

\end{cor}

\section{Generalized Shiffer Variations}\label{S:GSV}

The goal of this section is to generalize Shiffer variations to an arbitrary very ample line bundle.
We wish to prove the same sort of theorem for arbitrary line bundles as we have proven for $K_C$.  That is 
we wish to define Shiffer variations and relate their ranks to Cliff(C,L).  In this section we will show how
Shiffer variations may be defined in general and discuss the geometric information that is necessary to relate
a curve and its secant varieties to rank loci.

As with the case $L = K_C$ it is convenient to use the geometric version of
projective spaces of lines in the dual vector space. That is to say, we
consider the space of  lines in $\Hd{C,L} = \Hi{C,K_C \otimes L^{-1}}$, rather than
hyperplanes in $\H{C,L}$.

Consider the exact sequence \[ 0 \to K_C \otimes L^{-1} \to K_C \otimes
L^{-1}(D) \to K_C \otimes L^{-1}(D)|_D \to 0.\] This gives rise to a boundary
map: \[ \partial_D : \H{K_C \otimes L^{-1}(D)|_D} \to \Hi{K_C \otimes L^{-1}}
\]

\begin{thm}\label{T:linspace}

  $\partial_D \left( \H{K_C \otimes L^{-1}(D)|_D} \right) \subset \Hi{K_C
  \otimes L^{-1}}$ corresponds to the linear subspace $\overline{D} \subset
  \PP \left( \Hd{L} \right)$.

\end{thm}

\begin{proof}

  $\overline{D}$ is defined by 
  \[ \begin{aligned}
    \overline{D} & = \set{ x\vee \in \PP \left( \Hd{L}  \right) \, | \,
    x\vee |_{\H{L(-D)}} = 0 }\\ 
    & = \set{ x\vee \in \Hi{K_C \otimes L^{-1}} \, | \, x\vee |_{\H{L(-D)}} = 0 }\\
    & = \set{ x\vee \in \Hi{K_c \otimes L^{-1}} \, | \, x \to 0 \text{ in }
    \Hd{L(-D)} = \Hi{K_C \otimes L^{-1}(D)} }\\
    & = \set{ x \in \Hi{K_C \otimes L^{-1}} \, | \, x = \partial_D (y) \text{
    in some } y \in \H{K_C \otimes L^{-1}(D)|_D} }\\
  \end{aligned} \]
  where the last statement follows by the exactness of the long exact sequence.

\end{proof}

We now turn our attention to defining Shiffer variations in general. Just as
elements of $\Hi{T_C}$ can be considered as elements of \[\Hom\left(\H{C,K_C},
\Hi{C,\OO_C}\right)\] (via the cup product), we can consider elements of
$\Hi{C,K_C \otimes L^{-2}}$ $= \Hd{C, L^{\otimes 2}}$ as elements of $\Hom
\left( \H{C,L},\Hi{K_C \otimes L^{-1}} \right)$. To make sense of the geometry
we must assume that the map \[ A : \Hi{K_C \otimes L^{-2}} \to \Hom \left(
\H{L}, \Hi{K_C \otimes L^{-1}} \right) \] is injective. Notice:

\noindent \textbf{Fact}: $A$ is injective $\Longleftrightarrow \H{C,L}$ is
quadratically normal, that is $\H{L} \otimes \H{L} \to \H{L^{\otimes 2}}$ is
surjective.

\begin{proof}

  $\Hom \left( \H{L}, \Hi{K_C \otimes L^{-1}} \right) = \Hd{L} \otimes \Hd{L}$
  by Serre duality. The natural map $\Hi{K_C \otimes L^{-2}} \to \Hd{L}
  \otimes \Hd{L}$ is injective $\Longleftrightarrow \text{ the dual map }
  \H{L} \otimes \H{L} \to   \Hid{K_C \otimes L^{-2}} = \H{L^{\otimes 2}}$ is surjective.

\end{proof}

Consider the exact sequence \[ 0 \to K_C \otimes L^{-2} \to K_C \otimes
L^{-2}(D) \to K_C \otimes L^{-2}(D)|_D \to 0 .\] Let $T_L(D) = \partial \left(
\H{K_C \otimes L^{-2}(D)|_D} \right) \subset \Hi{K_C \otimes L^{-2}}$.

\begin{defn}

  $T_L(D)$ are the Shiffer variations (for $L$) supported on $D$.

\end{defn}

\textbf{Remark}: By Theorem \ref{T:linspace}  the vector space $T_L(D)$ corresponds to
the linear space spanned by $D$ in $\PP \left(\Hi{K_C \otimes L^{-2}}\right)$.
By an abuse of notation we will frequently make no distinction between the
affine elements of $T_L(D) \subset \H{K_C \otimes L^{-2}}$ and the projective
elements $T_L(D) = \overline{D} \subset \PP \left( \Hi{K_C \otimes L^{-2}}
\right)$. If $p \in D$, and $z$ is a local parameter, and $\ell_z$ a local
generator of $L$ at $p$, then $\frac{\partial}{\partial z} \otimes
\ell_z^{\otimes 2}$ is an element of $T_L(D)$ supported at $p$. Any other
representative differs from this one by a scalar.

Our calculations relate to rank and are independent of the choice of local
parameter.  We will generally proceed by choosing a local parameter and making
our calculation ``locally''.  We need one more piece of notation.

\begin{defn}
$S_L(D) \subset \Hi{K_C \otimes L}$ is the affine cone over
the span of $D$ in $\PP \left(\Hi{K_C \otimes L}\right)$
\end{defn}

The next theorem shows how to compute Shiffer deformations. Just as in the case
of $L = K_C$ they can be computed locally. and hence the exact same proof
applies.

\begin{thm} \label{T:localcalc}

  Let $\xi \in T_L(D) \subset \Hi{K_C \otimes L^{-2}}$. Denote by $\rho$ the
  natural restriction $\H{C,L} \stackrel{\rho}{\maps} \H{C,L|_D}$. Denote by
  \[\partial_1 : \H{K_C \otimes L^{-1}(D)|_D} \to \Hi{K_C \otimes L^{-1}}\]
  the boundary map in the long exact sequence of cohomology coming from the
  exact sequence \[ 0 \to K_C \otimes L^{-1} \to K_C \otimes L^{-1}(D) \to K_C
  \otimes L^{-1}(D)|_D \to 0, \] and denote by $\partial_2 : \H{K_C \otimes
  L^{-2}(D)|_D} \to \Hi{K_C \otimes L^{-2}}$ the boundary map in the long
  exact sequence of cohomology coming from the short exact sequence \[ 0 \to
  K_C \otimes L^{-2} \to K_C \otimes L^{-2}(D) \to K_C \otimes L^{-2}(D)|_D
  \to 0 .\]

  Let $\tilde{\xi} \in \H{K_C \otimes L^{-2}(D)|_D}$ be a lifing of $\xi$,
  i.e. $\xi = \partial_2(\tilde{\xi})$. Then the following diagram factors 
  the map
  $\cup \, \xi : \H{C,L} \to \Hi{C,K_C \otimes L^{-1}}$ as  \[ \H{C,L}
  \stackrel{\rho}{\maps} \H{L|_D} \stackrel{\cup\,\tilde{\xi}}{\maps} \H{K_C \otimes
  L^{-1}(D)|_D} \stackrel{\partial_1}{\maps} \Hi{K_C \otimes L^{-1}}.\]

\end{thm}

\begin{proof}

  We can use  C\v{e}ch cohomology to compute all the maps. Let $V_1$ be an open
  set such that $D \subset V_1$, and $V_2 = C - D$. We use this open cover.
  Write $D = \sum_{i=1}^{n} n_i p_i$.

  A class $\xi \in T_L(D)$ is exactly $\partial_2(\tilde{\xi})$, $\tilde{\xi}
  \in \H{K_C \otimes L^{-2}(D)|_D}$. Letting $s_i$ be a local section of $K_C
  \otimes L^{-2}$ around $p_i$. 

  \begin{enumerate}

    \item[(1)] $\tilde{\xi} = \displaystyle \sum_{i=1}^n s_i
       \left( \displaystyle \sum_{j=1}^{n_i}
      \frac{a_j^i}{z_i^j} \right)$ where $z_i$ is a local parameter at $p_i$.

  \end{enumerate}

  Lifting $\tilde{\xi}$ to a section of $\Gamma(V_1 \cap V_2, K_C \otimes
  L^{-2})$ gives a representation of $\xi \in \mathcal{C}^1(C, K_C\otimes
  L^{-2})$.  With this description it is clear that $\H{L} \stackrel{\cup \,
  \xi}{\maps} \Hi{K_C \otimes L^{-1}}$ factors as \[\H{L} \stackrel{\cup \,
  \tilde{\xi}}{\maps} \H{K_C \otimes L^{-1}(D)|_D}
  \stackrel{\partial_s}{\maps} \Hi{K_C \otimes L^{-1}}.\]
    
\end{proof}

From this and \ref{T:linspace} we get:

\begin{cor}\label{c:shifferdescr}

  Let $\xi \in T_L(D)$. Then $\Ker(\xi) \supset \H{L(-D)}$ and $\im(\xi) \subset
  S_L(D)$, the affine cone over $\overline{D}$. 
  
\end{cor}

When  $L = K_C$, Theorem \ref{T:mthm}   characterized the
Clifford index as the smallest integer $j$  for which $\Sec^{j-1}(C)$ is not
the full rank $j$ locus.
 Here is the set-up for general $L$.  We assume for now
only that $L$ is very ample  and quadratically normal.

\begin{displaymath}
  \xymatrix{
  & \PP \left( \Hi{K_C \otimes L^{-1}} \right) \ar[dr]^{v} & \\ 
  C \ar[ur]^{\phi_L} \ar[dr]^{\phi_{2L}} &  
  & \PP \left( \sym^2 \left( \Hi{K_C \otimes L^{-1}} \right) \right) \\
  & \PP \left( \Hi{K_C \otimes L^{-2}} \right) \ar@{^{(}->}[ur]^{i} & }
\end{displaymath}

%\begin{displaymath} 
%  \xymatrix{ 
%  C \ar[r]^{\phi_L} \ar[d]_{\phi_{2L}} & 
%  \PP \left( \Hi{K_C \otimes L^{-1}} \right) \ar[d]^{v} \\ 
%  \PP \left( \Hi{K_C \otimes L^{-2}} \right) 
%  \ar@{^{(}->}[r]_{i}		       & 
%  \PP \left( \sym^2 \left( \Hi{K_C \otimes L^{-1}} \right) \right) } 
%\end{displaymath}

$\phi_L$ and $\phi_{2L}$ are the maps determined by the linear system $|L|$
and $|2L|$ and  $i$ is the inclusion $\PP \left( \Hi{K_C \otimes L^{-2}} \right)$
as a linear subspace, which is true since $L$ is quadratically normal. Finally, $v$ is
the Veronese map. $v$ is usually defined as a map to $V \otimes V$, but on the
algebra level it is $x \in V \longmapsto x \otimes x \in V \otimes V$ which is
clearly in the $\sym^2(V)$.  In other words, $\nu$ embeds $\PP(V)$ as the rank
one symetric matrices in $\PP(V\otimes V)$
 Our general set-up is summarized by

\begin{thm}

  Let $V$ be any vector space, $M$ a linear  subspace of $\sym^2(V)$ and $X$ a
  subvariety of $\PP(V)$ such that $X_2$ = $\phi_{2L}(X) \subset \PP(M)$.

  \begin{enumerate}

    \item[(i)] Viewing $\sym^2(V)$ as a linear subspace of   $\Hom (V^*,V)$, then $v(\PP(V))$ is
      exactly the rank 1 matrices in $\PP \left( \sym^2(V) \right)$.
      $\Sec^k(v(\PP(V))) $, which is  the closure of the variety swept out by
      linear spans of ($k+1)$       points of
      $\Sec^0(\PP(V)) = v(\PP(V))$,  is exactly the rank $(k+1)$ locus.

    \item[(ii)] $\Sec^k(X_2)  \subset \PP(M) \cap
      \Sec^k(v(\PP(V)))$; in short, the variety of rank $(k+1)$ matrices 
       in $ \PP(M)$ contains the k secants to $X_2$. 

  \end{enumerate}

\end{thm}

\begin{proof} $(i)$ This is classical- elements of $\sym^2(V)$ may be viewed
as quadrics.  The only invariant of a quadric is its rank (see for example
\cite{GH} p.792).  This means given a 
quadric $Q$ of rank $k+1$ there is a basis $\langle q_1, \dots , q_n \rangle$
of $V$ such that
 in these coordinates $Q= \sum_{i=1}^{k+1} a_iq_i^2$. 
This shows that $Q$ is in $\sec^k(v(\PP(V)))$ as each $q_i^2 = v(q_i)$ is
in $v(\PP(V))$.

  $(ii)$ is clear because any element of $v(X)$ is of rank 1. As many
  people have observed the sum of $k$ rank 1 matrices is of rank $\le k$.

\end{proof}
  
By definition $\sec^0(X)=X$. So that if points of $X$
correspond to rank one matrices, then elements of $\sec^k(X)$ correspond to sums
of $k+1$ elements and hence to matrices of rank $k+1$ or less.
Because  $\Sec^k(\nu(\PP(V)))$ is exactly rank $k+1$ locus in $\PP(\sym^2(V))$ means
$\Sec^k(v(\PP(V))) \cap \PP(M)$ is the rank $k+1$ locus in $M$, and hence
the equality $\Sec^k(v(\PP(V)) \cap M = \Sec^k(X_2)$ means that every
element of $M$ of rank $k+1$ (or less) can be written as the sum of $k+1$ (or
less) elements of $v(X)$. 
In short, we have for $v(X)$ that
$\Sec^k(v(X))$, considered as a subvariety of $\PP \left( \Hom(V^*,V)
\right)$, is contained in the rank $k+1$ locus, and we wish to know when it is
exactly the rank $k+1$ locus. This rank $k+1$ locus is harder to understand
geometrically than the secant variety  but is determinantally defined. 

We are interested in when one has equality. We want to know when 
is  $\Sec^j(X_2) =
\Sec^j(v(\PP(V))) \cap \PP(M)$. We are only discussing the issue of set
theoretic equality. We use this space to record a criteria of Hassett and a
generalization.

\begin{thm}\label{T:Hassett}

Let $X \subset \PP(V)$. Assume $X$ is projective. Let $v_d : \PP(V) \to \PP
  \left( \sym^d(V) \right)$. Assume $L = \OO_X(1)$ is linearly and $d$ normal,
  i.e.  $\H{X,\OO_X(1)} = V$ and $\H{X,L}^{\otimes d} \to \H{X,L^{\otimes d}}
  = M$ is surjective. Then we get a commutative diagram 
  \begin{displaymath}
    \xymatrix{ 
    X \ar@{^{(}->}[r]^{\varphi_{_L}} \ar@{^{(}->}[dd]_{\varphi_{_{L^{\otimes d}}}} 
    & **[r]\PP(V) \ar[dd]^{v_d} \\
    & & & \\
    **[l]\PP(M) \ar@{^{(}->}[r] & **[r]\PP(\sym^d(V))} 
  \end{displaymath}

Let $X_d$ be the image of $X$ under $\varphi_{_{L^{\otimes d}}}$ and let $ R^j = 
\Sec^{j-1}(v_d(\PP(V)) \cap \PP(M)$ ($j\geq 1$).  Suppose that $ m<d$ and that
$R^j = \Sec^{j-1}(v_d(\PP(V)) \cap \PP(M)\;  \forall \, j < m $.  Then $\Sec^{m-1}(X_d) \subsetneqq
R^m \Leftrightarrow   \exists$ an $m-1$ dimensional subspace $\Lambda \subset \PP(V)$ such that
  \begin{enumerate}

    \item[(i)] $X \cap \Lambda = \emptyset$, and

    \item[(ii)] the map $\H{\PP(V),I_X(d)} \to \H{\Lambda,\OO_\Lambda(d)}$ is not
      surjective.

  \end{enumerate}

\end{thm}

\begin{proof}

  This is a set theoretic statement. First notice $\H{V,I_X(d)} \subset \H{V,\OO_V(d)}$
  corresponds to the hyperplanes $H$ on $\PP(\sym^d(V))$ such that $H \supset
  \PP(M)$, i.e. $\PP(M) = \bigcap V(H)$, $H \in \H{I_X(d)}$.  If $m=0$ the theorem states that $X$ is
  defined by degree $d$ homogeneous polynomials, that is to say $X_d \, (=\Sec^0(X_d)) = \PP(M) \cap v_d(\PP(V))$ if and only
  if $\forall \, p \notin X$ $\exists \, \varphi \in \H{V,I_X(d)}$ such that
  $\varphi(p) \ne 0$.  This is clear.
  
  The proof in the general case is exactly the same as for when $d = 2$ and
  $X$ is a curve, the case Hassett did. We include it for completeness.

  Firstly suppose $\exists \, p \in \Sec^m(v_d(V)) \cap \PP(M) \diagdown
  \Sec^m(X_d)$. Since $p \in \PP(M)$, $p \in V(H)$, $\forall \, H \in I_X(d)$.
  Further since $p \in \Sec^m$, $\exists \, p_1, \ldots, p_m \in v(\PP(V))$
  such that $p = \sum_{i=1}^m \lambda_i p_i$, i.e. each  $p_i=x_i^d$ so  $ p =
  \sum_{i=1}^m \lambda_i x_i^d$. 
  
  If we let $\Lambda = \langle p_1, \ldots p_m \rangle$, then
  $\H{V,I_X(d)}$ vanished at $p$, i.e. \[\H{V,I_X(d)} \arrownot\twoheadrightarrow
  \H{\Lambda,\OO_\Lambda(d)}.\] 
  
  If $\Lambda \cap X \ne \emptyset$, say some $y \in \Lambda \cap X$, then
  writing $\sum_{i=1}^m \lambda_i x_i^d = Q + \mu y^d$ with rank  $Q \le m -
  1$, say $Q = \sum_{i=1}^{n-1} \tilde{\lambda_i} \tilde{y_i}$, we would have
  $\langle y_1, \ldots, y_{m-1} \rangle \in \Sec^{m-1}(\PP(V)) \cap \PP(M) /
  \Sec^{m-1}(X_d)$, contradicting our hypothesis. 
  
  Conversely, given $\Lambda$ such that $Q \in \H{\Lambda,\OO_\Lambda(d)}$ not
  in the image of $\H{\PP(V),I_X(d)}$, such a $Q$ can be written as $\sum_{i=1}^m
  \lambda_i x_i^d$, then $\sum_{i=1}^d \lambda_i v_d(x_i)$ is a point in
  $\Sec^m(v_d(V)) \cap \PP(M)$. Clearly it is not in $\Sec^m(X_d)$ since
  $\Lambda \cap X = \emptyset$ by hypothesis. 

\end{proof}

We now return to Shiffer variations. For $L = K_C$ we used the Clifford index
to put lower bounds on the rank of elements of $T_L(D)$. Namely we had
$\rk(\tau) \ge d - 2r$ if all the coefficients were non-zero. That translated
to, after eliminating the possibilities of using elements of $T_L(D)$ where
$D$ is not eligible to compute $\cliff(C)$,  the fact that the $j$ secant
varieties of $C$ in $\PP \left( \Hd{C,K_C^{\otimes 2}} \right)$ are exactly
the rank $j$ varieties. The first part goes through  unchanged.  

\begin{thm}\label{T:rankcalc}

  Let $\tau \in T_L(D)$, $D = \sum_{i=1}^n n_i p_i$, and let $z_i$ be a local 
  parameter at $p_i$ and let $l_i$ be a local generator for $L$ 
at $z_i$ . Suppose that $\tau$
  corresponds to $\tilde\xi \in \left( \sum_{i=1}^n \sum_{j=1}^{n_i}
  \beta_{ij} z_i^{-j} \right) \otimes l_i$, with all $\beta_{ij} \ne 0$,
  $\forall \, i, j$, or at least $\forall \, i$, $\beta_{in_i} \ne 0$.
  Then $d - 2r \le rk(\tau) \le d -r$, where $r = \cliff(L,D)$.

\end{thm}

\begin{proof}

   The condition
on the coefficients of the Shiffer variations being non-zero is  exactly the condition
that the Shiffer variation is supported on D, but not on some subdivisor of D.
By Corollary \ref{c:shifferdescr}   $\im(\tau) \subset |D| \cong \PP^{d-r-1}$ and hence $\rk(\tau)
  \le d - r$. Recall that by Theorem \ref{T:localcalc}
we have the factorization \[ \H{L} \to \H{L|_D}
  \stackrel{\cup \, \tilde\xi}{\maps} \H{K_C \otimes L^{-1}(D)|_D} \to \Hi{K_C
  \otimes L^{-1}}.\] Assume temporarily  that $\cup \, \tilde\xi$ is an
  isomorphism of $V = \H{L|_D}$ with $\H{K_C \otimes L^{-1}(D)|_D} = V\vee$,
  then $\tilde\xi$ restricted to $\im \left( \H{L} \to \H{L|_D} \right)$ maps
  to $\Hd{L} = \Hi{K_C \otimes L^{-1}}$ and the result follows from the linear
  algebra fact:

  \begin{lemma}

    Let $\varphi : V \to V^{\vee}$ be a linear map which is an isomorphism. Let $W \subset V$ be of
    codimension $r$, then $\varphi|_W : W \to W^{\vee}$ is of rank $\ge d - 2r$.

  \end{lemma}

  \begin{proof}

    Because $\varphi$ is an isomorphism, $\varphi_1: W \to V\vee$ has rank $d
    - r$ and since $\pi : V\vee \to W\vee$ has an $r$ dimensional kernel,
    $\varphi|_W = \pi \circ \varphi_1$ has rank at least $(d - r) -r = d -
    2r$. 

  \end{proof}

  Hence we are reduced to showing $\tilde\xi$ is an isomorphism. Since all the
  sheaves are skyscraper sheaves the question is local and we may assume
  that $D = n p$. Our diagram now looks like

$$\H{C,L|_{np}}  \stackrel{\cup \, \tilde\xi}{\maps} \H{K_C \otimes L^{-1}(np)|_{np}}.$$

 Since we are local
  over $p$ we can identify $\H{C,L|_{np}}$ with $k[z]/z^n$, where $k$ is a
  field of definition for $C$ and we can identify $\H{K_C \otimes L^{-1}(np)|_{np}}$ with
  $\bigoplus_{j=0}^{n-1} z^{-j} k$. Under this identification $\xi =
  \sum_{j=0}^{n-1} \beta_j z^j$ corresponds to the matrix
  \[ \begin{pmatrix}
    \beta_{n-1} & \beta_{n-2} & \cdots & \beta_0\\ 
    & \beta_{n-1} & \cdots & \beta_1 \\ 
    & & \ddots & \vdots \\ 
    & & & \beta_{n-1} 
  \end{pmatrix} \]  
  which has determinant $(\beta_{n-1})^n$. 

\end{proof}

As mentioned in Theorem \ref{T:rankcalc} the condition $\beta_{i,{n_i}} \neq 0$ is exactly
the condition that $\tau \in T_L(D)$ is not an element of $T_L(D')$ for some $D' 
\subset D$.  That is to say that $\tau \in \sec^{d-1}(C) / \sec^{d-2}(C)$.
Obviously elements of $T_L(D)$ can have low rank by lying in a low
dimensional secant variety.  For example  $\rk(\tau_p) =1$ for any
$p \in C$.  Since we are only interested in the ranks of  generic
elements of $T_L(D)$ we make the following definition.

\begin{defn} $T_L^*(D) = \{\tau \in T_L(D)\vert \beta_{i,{n_i}} \neq 0 \forall i \}
$ That is $T_L^*(D) = T_L(D) \cap (\sec^{d-1}(C) /\sec^{d-2}(C))$.

\end{defn}

When $L = K_C$ one checked, using Shiffer variations, that one had equality
$\sec^{j-1}(C) = R^j$ for $j<$ Cliff(C).  The same sort of result is true in
general, but the results have different flavors depending on whether $\hi{L} >
0$ or  $\hi{L} = 0$.  Roughly speaking for $\hi{L} > 2$, I cannot say anything
general. For $\hi{L} = 1$ a version of the theorem is true, but weaker, as the
lower bound does not generally occur. If $\hi{L} = 0$ and $\deg(L) \ge 2g - 2$.
one can get strong results with the strongest results being for $\deg(L) \geq
2g+1$.  We first consider the case $\hi{L}  = 0$ and $\deg(L) > 2g +1$.

\section[Line Bundles of Large Degree]{Geometric Characterization of the Clifford Index for Line Bundles of Large Degree}\label{S:large_cliff}

Throughout this section L will be a very ample line bundle of degree $\ge 2g+1$. We restrict to
this case because in this range $L$ is always a very ample, quadratically normal line bundle.   Recall that in
section \ref{S:GRR} we have calculated the Clifford index of any line bundle of degree $\ge 2g+1$.  We recall
this and explain its relationship with secant varieties now.   

\begin{thm}\label{T:mainthm}

  Suppose $\deg(L) = 2g - 2 + d$ with $d \ge 3$.  Then $\cliff(C,L)\ge d-2$. Let $C_2 = \phi_{2L}(C) \subset \PP \left(
  \H{C,L^{\otimes 2}} \right)$, then for $j< \cliff(C,L)$, $  Sec^{j-1}(C) = R^j$. 
\linebreak For $ c= \cliff(C,L)$ and generic $L$, $R^c \varsupsetneqq \Sec^{c-1}$.

\end{thm}

By Theorem  \ref{T:clifford_calc_1} we know that if $L= K_C(D)$ with $D$ effective,
 then Cliff(C,L)= $d-2$ and by Theorem \ref{T:clifford_calc_2}  generically
 Cliff(C,L) $= d-1$.  In both cases by Lemma \ref {L:d-2} we can
find special Shiffer variations of rank equal to $\cliff(C,L)$.
  We are only claiming set theoretic equality (and inequality) at the moment.
We will break the proof down into two pieces, the equality and the inequality.

Recall our basic setup: L is very ample.
\begin{displaymath}
  \xymatrix{
  & \PP \left( \Hi{K_C \otimes L^{-1}} \right) \ar[dr]^{v} & \\ 
  C \ar[ur]^{\phi_L} \ar[dr]^{\phi_{2L}} &  
  & \PP \left( \sym^2 \left( \Hi{K_C \otimes L^{-1}} \right) \right) \\
  & \PP \left( \Hi{K_C \otimes L^{-2}} \right) \ar@{^{(}->}[ur]^{i} & }
\end{displaymath}

%begin{displaymath}
%    \xymatrix{ 
%    C \ar[r]^{\varphi_{_L}} \ar[d]_{\varphi_{_{2L}}} 
%    & \PP\left( \H{K_c \otimes L^{-1}} \right) \ar[d] \\
%    \PP \left( \Hi{K_C \otimes L^{-2}} \right) \ar@{^{(}->}[r] 
%    & \PP\left( \sym^2 \left( \Hi{K_C \otimes L^{-1}} \right) \right) } 
%  \end{displaymath}

That this diagram commutes follows if $L$ is very ample and quadratically
normal.

Recall that via the inclusion \[\PP \left( \Hi{K_C \otimes L^{-2}} \right)
\hookrightarrow \PP\left( \sym^2 \left( \Hi{K_C \otimes L^{-1}} \right)
\right)\] we can endow $\PP \left( \Hi{K_C \otimes L^{-2}} \right)$ with two
filtrations $\Sec^{j-1}(C) = \cup$ (span of $j$ points of $C$) and $R^j$, which is
the set of all elements of $\PP \left( \Hi{K_C \otimes L^{-2}} \right)$ which are of rank
$\le j$ when viewed as elements of $\PP \left( \sym^2 \left( \Hi{K_C \otimes
L^{-1}} \right) \right) \hookrightarrow \PP \left( \Hom \left( \H{L}, \Hd{L} \right)
\right)$. Since every $p \in C$ corresponds to rank one elements, namely the
rank one Shiffer deformation $\sigma_p^2$, and since the sum of $j$ rank one
elements is of rank $\le j$, we have \[\Sec^j(C) \subset R^j(C) \subset \PP
\left( \Hi{K_C \otimes L^{-2}} \right).\]

\begin{thm}

  Suppose  that $ j< Cliff(C,L)=c$. then as sets $\Sec^{j-1}(C) = R^j(C)$. Further if $L$ is generic then
  $\Sec^{c-1}(C) \subsetneqq R^c(C)$.

\end{thm}

\begin{proof}

  Let $E$ be any divisor on $C$. Using the factorization for the map  $\xi \in
  T_L(E)$ given by theorem \ref{T:localcalc} we can factor $\xi \in T_L(E)$ as,
 \[ \H{L}/\H{L(-E)}
  \hookrightarrow \H{L|_E} \stackrel{\cup \tilde{\xi}}{\to} \H{K_C \otimes
  L^{-1}(E)|_E} \stackrel{\partial_E}{\to} \Hi{K_C \otimes L^{-1}}.\]
 
  If $\hi{L(-E)} = \h{K_C \otimes L^{-1}(E)} = 0$, we see $\partial_E$ is
  injective, and hence if $\xi  \in T_L(E)$ is represented by $\xi =
  \partial(\tilde{\xi})$, with $\tilde{\xi} \in \H{K_C \otimes L^{-2}(E)|_E}$,
  then $\xi $ is injective $\Longleftrightarrow \tilde{\xi}$ is injective.

  By theorem \ref{T:rankcalc}  $\tilde\xi$ is injective if $\xi \in T_L^*(E)$.
    This is exactly the statement that $\xi  \in \Sec^{e}(C) / \Sec^{e-1}(C)$,
  with $e = \deg(E)$. If $\hi{L(-E)}>0 $ then $\rk(\xi) \ge 
\cliff(L,E) \ge
  \cliff(C,L)=c $, since even divisors ineligible to compute $\cliff(C,L)$ have Clifford index
   greater than $\cliff(C,L)$. Thus  we see that any $\xi \in \Hi{K_C \otimes L^{-2}}$ of
  rank $<c$ must be in $T_L(E)$ with $\deg(E) = \rk(\xi)$, and all
  the highest order terms of $\xi$ are non-zero. This last statement is that
  $R^j(C) = \Sec^j(C)$, for $j <c$. $\Sec^{c} \ne R^{c}$ follows
  from two lemmas:

  \begin{lemma} 

    If $r_L(D) = r > 0$ and $r_{L^{\otimes 2}}(D) = 0$, we can find a $\tau \in T_L(D)$,
    $\rk(\tau) = d - 2$. 

  \end{lemma}

 Our assumption is that $D$ fails to impose independent conditions on the linear
system given by $L$, but does impose independent conditions of $L^{\otimes 2}$.  Hence 
 $T_L(D)$ gives rise to a $d$-dimensional subspace of $\Hi{K_C
  \otimes L^{-2}}$, and hence we have $d$ rank one elements  of $T_L(D)$ whose image
  lies in  at most a $(d-1)$ dimensional space. In fact, $\tau$ are distinct as
  elements of \[\Hom \left( \H{L}/\H{L(-D)}, \ker \left( \Hi{K_C \otimes
  L^{-1}} \right) \to \Hi{K_C \otimes L^{-1}(D)} \right),\] and so the result
  follows from:

  \begin{lemma}\label{L:d-2}

    Let $x_1,\ldots,x_d \in V$ with $\dim(V) = d-1$, $x_1,\dots,x_{d-1}$ a
    base, $x_d = \sum_{i=1}^{d-1} a_i x_i$, with $a_i \ne 0$. Then 
    identifying $\sym^2(V)$ with $\Hom^{sym}(V^{\vee,V}$, there exists  
    $\lambda \ne 0$ such that $\rk(\sum_{i=1}^{d-1} x_i^2 + \lambda x_d^2)
    \le d - 2$.

  \end{lemma}

  \begin{proof}

    Let $f(\lambda) = \det(x_1^2 + \cdots + x_{d-1}^2 + \lambda x_d^2)$, so
    $f(0) = 1$. Unless $f$ is constant there exists $\lambda$ such that
    $f(\lambda) = 0$, i.e. $\rk(x_1^2 + \cdots + x_{d-1}^2 + \lambda x_d^2) \le
    d - 2$ as desired. But as a polynomial in $\lambda$ it has leading term
    $\lambda^{d-1} \prod_{i=1}^{d-1} a_i^2$, and hence the polynomial is
    non-constant.

  \end{proof}

  \begin{lemma}\label{L:specialvar}

    Let $L = K_C(D)$. $\exists \, \tau \in T_L(D)$ such that i) $\rk(\tau) = d
    - 2$, ii) $\im(\tau) \cap C = \emptyset$.

  \end{lemma}
  
  \begin{proof}

    By the lemma above $\exists \, \tau$ of rank $< d - 1$ and $\rk(\tau) \ge
    d - 2 = \cliff(C,D)$ in any event, so $\exists \, \tau$ such that $\rk(\tau)
    = d - 2$. By Hassett's criterion(\ref{T:Hassett}, since $\Sec^{d-3} = R^{d-3}$, if $\tau
    \notin \Sec^{d-2}(C)$, then $\im(\tau) \cap C \ne \emptyset$. But if $\tau \notin \Sec^{d-2}(C)$ ,
    then $\exists \, \tau' \in T_L(D')$ with $\deg(D) < d - 2$ such that $\tau
    = \tau'$, i.e. $D' = D - R$, with $R$ effective. This would mean
    $\sigma_1^2 + \cdots \sigma_{d-1}^2 + \lambda \sigma_d^2 = \sum_{i=1}^{d'
    < d} \lambda_i \sigma_i^2$, i.e. $\H{K_C \otimes L^{-2}(D)} \ne 0$, i.e.
    $\H{L^{-1}} \ne 0$ which is absurd. (Any relation $\sum_{\sigma_i \in D}
    \lambda_i \sigma_i^2 = 0 \in \H{K_C \otimes L^{-2}}$ implies $\H{K_C
    \otimes L^{-2}(D)} \ne 0$.)
  
  \end{proof}
\end{proof}

\section{The scheme structure of $R^p$ for $p<\cliff(L,C)$.}\label{S:SS}

Throughout this section   $L$ will either  be $K_C$ or $\deg(L) \geq 2g+1$.
If $L =K_C$ then we will further assume that $g(C) \geq 3 $ and $\cliff(C)
\geq 1$.  In particular $L$ will always be very ample and quadratically normal.
 Let $n=\h{L}$, $V=\Hd{C,L}$
and let $M_2= \Hd{L^{\otimes2}}$. Denote by

\begin{equation} \label{E:rankdef}
R^p = R^p(L) 
\subset \PP\left(\Hi{K_C\otimes L^{-2}}\right) 
\subset \PP\left(\sym^2\left(\Hi{K_C\otimes L^{-1}}\right)\right)
\end{equation}

the rank $p$ locus. 
We consider $\PP\left(\sym^2 \left( \Hi{K_C\otimes L^{-1}}\right) \right) $ as a
subspace of \linebreak $\PP(\Hom(V^{\vee},V)$ and 
  follow \cite{ACGH}  in defining the scheme
structure of $R^p$. 

Let $\mathbf{H} = \Hom(V,W)$ where $V$ and $W$ are finite dimensional vector
spaces.  The space of rank p matrices in $\mathbf{H}$ is defined by the vanishing of all
the $(p+1) \times (p+1)$ minors with respect to some choice of a basis for $V$ 
and $W$.  It is denoted by $\mathbf{H}^p. $
It is known  that $\mathbf{H}^p$ is Cohen--Macaulay and smooth away
from $\mathbf{H}^{p-1}$.  
The Cohen-MacCauley statement can be found in \cite{W}  p.175 for example.  The
proof of the  smoothness statement follows from the existence of a canonical desingularization,
$\~{\mathbf{H}}^p$,  of  $\mathbf{H}^p$, and the calculation of the tangent space to $\~{\mathbf{H}}^p$
 at a general point.
For completeness we sketch the construction.  The details may be found in \cite{ACGH}.

If $\varphi \in \mathbf{H}^p$, then $\dim(\ker(\varphi)) \geq n-p$ and so we can consider
the set of all pairs $(\varphi,A) \subset \tilde{\mathbf{H}^p} \times G(n-p,n)$ such that 
$A \subset \ker(\varphi)$.  One checks that this is smooth by showing that the projection onto
$G(p,n)$ has fiber over $A$ equal to $\Hom(V/A,W)$ and hence is a vector bundle.  This is the 
definition of $\~{\mathbf{H}^p}$.
Further if $\varphi \in \mathbf{H}^p \setminus \mathbf{H}^{p-1}$ then $T_{\varphi,\mathbf{H}} =
\{\psi \in \mathbf{H} | \psi : \ker(\varphi) \to \im(\varphi) \}$.

\begin{defn}\label{D:1}  Let $ \mathbf{M} \stackrel{i} \hookrightarrow \Hom(V,W)= \mathbf{H}$ be an
inclusion of vector spaces.  The scheme $R^p(M)= R^p$ is defined as $i^{*}(\mathbf{H}^p)$.
With this scheme structure, the ideal of $R^p$ in  $M$  is the pull back of the ideal defining $\mathbf{H}^p$,  which  is all
  the $(p+1)\times(p+1)$ minors. The same definition holds in the projective case,
$ \PP(\mathbf{M})  \stackrel{i} \hookrightarrow \PP(\Hom(V,W)= \PP(\mathbf{H}))$
\end{defn}

\textbf{Remark 1:} In our case $V=W^{\vee} = \H{C,L}$ and $M= \H{C,L^{\otimes 2}}$ and in fact
$\mathbf{M} \hookrightarrow \sym^2(W) \hookrightarrow \Hom(V,W)$.  We will need this greater
generality to deal with the scheme structures that occur in section  \ref{S:eks}.

\textbf{Remark 2:} We will see later that in this case, $\mathbf{M}=\H{C,L^{\otimes 2}},
$ that $R^p \setminus R^{p-1}$ is smooth if $p<\cliff(C,L)$.

For later reference we include:

\begin{lemma}\label{L:tanspace_id}

With the notation of Definition \ref{D:1}, let $\varphi \in R^p
\setminus R^{p-1}$, then $T_{\varphi,\mathbf{M}} = 
\{\, \psi \in \mathbf{M} | \psi: \ker(\varphi) \to \im(\varphi)$\}
\end{lemma}
\begin{proof}
Let $\mathbf{T}$ denote the tangent space to $\varphi$ in $\Hom(V,W)$.
We have previously identified $\mathbf{T}$ with $\{ \psi \in \mathbf{H}|
\psi : \ker(\varphi) \to \im(\varphi) \}$.  Since $\mathbf{M}$ is a 
subspace of $\mathbf{H}$, 
 $T_{\varphi,\mathbf{M}}= \{\, \psi \in \mathbf{T} \,  | \psi(I)=0 \}$
where $ I = I(\mathbf{M})$, is the ideal of $\mathbf{M}$.  Since $\mathbf{M}$
is a linear space, the condition on $\psi$ is that $\psi \in \mathbf{M}$
\end{proof}

Recall that Theorem  \ref{T:mainthm}  identified $\sec^{p-1} $
with $R^p(L)$ as sets.  We have given a scheme structure to 
$R^p$ and to  show scheme theoretic equality we recall
the definition of the scheme $\sec^{p-1}(C)$. We are interested
in the secant varieties to $C$ in $\PP(\H{L^{\otimes 2})}$. 
However, the
definitions  apply to any embedding of C by a very
ample line bundle $M$ in $\PP(\Hd{C,M})$
The definition of the scheme structure of the Secant variety 
goes back to \cite{Sch}, but we follow \cite{B}.  In essence one considers 
the open set of points in $\sym^k(C)$ which define $k-1$ planes, 
take their image
in $\PP(V)$ and take the reduced scheme structure on the closure of this set.
The actual definition requires some notation.

We recall the notations and results of  \cite{B} 
 and  construct a rank $p$ bundle $B^{p-1}(M)$ over $\sym^p(C)$.
Informally it is
    the rank $p$ bundle $D \mapsto \H{C, M|_D}$. The actual 
construction is given below.  $M$ will denote any 
very ample line bundle with $\h{M}=n$ and
 satisfying $ \h{M(-D)}= \h{M} -d$ for all divisors $D$ with
$\deg(D)=d \leq p$. The last condition is that   $M$  separates
$p$ points.

Let $\mathcal{D}_p \hookrightarrow C \times \sym^p(C)$ be the universal
    divisor, $\pi_2: C\times \sym^p(C) \maps \sym^p(C)$, then $B^{p-1}(M) =
    \pi_{2*}\left( \pi_1^*(M)|_{\mathcal{D}_p} \right)$.

  Since $\pi_{2*}\pi_1^*(M) =
    \H{M} \otimes \OO_{\sym^p(C)}$ and $M$ separate $p$ points
, the map  $\H{M}\otimes\OO_{\sym^p(L)} \twoheadrightarrow
    B^{p-1}(M)$ is surjective.  Hence we get an inclusion of 
projective bundles, 
\begin{equation}\label{E:secform} 
\beta_{p-1} : \PP(B^{p-1}(M)) \hookrightarrow 
\PP(\pi_{2*}\pi_1^*(M))= \PP(\H{C,M)} \times \sym^P(C)\end{equation}

We now define the Secant variety.

\begin{defn} \label{D:secant}
The secant variety $\sec^{p-1}(C)$ is the scheme theoretic image of
$B^{p-1}(M)$ in $\PP(\H{C,M)}$ under the projection onto the first 
factor in equation  \ref{E:secform}.  We refer to $B^{p-1}(M)$ as the
full  secant bundle.
\end{defn}
\textbf {Remark 1)} Since $B^{p-1}(M)$ is smooth, and so reduced, the scheme
theoretic image of the projection is also reduced.  See \cite{ha}  p.92
for details.

\textbf {Remark 2)} Since $\H{M} \otimes \OO_{\sym^p(C)} \to B^{p-1}(M)$
is surjective,
 we see that $B^{p-1}(M)$ 
is  a rank p bundle over $\sym^p(C)$ 
generated by $n$ global sections.  Hence $B^{p-1}(L)$ gives rise to map
$g: \sym^p(C) \to G(p,n)$  This map takes a divisor $ D \in \sym^p(C)$ 
to the $p-1$ dimensional subspace spanned by the divisor $D$.  As such,
$B^{p-1}(M)$ is the pullback of the universal subbundle on $G(p,n)$ and 
hence is the incidence correspondence $\{ (x,\overline D)|\, x 
\in \overline D \} \subset \PP(\H{C,M}) \times \sym^p(C)$.

\textbf {Remark 3)} We have used the fact the the line bundle $L$ separates $p$ points in 
the definition.  It is possible to define a the Secant variety
without this extra condition,  see for example \cite{B}.

For the rest of this section $L$ will also satisfy $\cliff(C,L) >p$.  Notice that
if $\cliff(C,L) >p$ then $L$ separates $p$ points.  Recall that $R^p$ is defined
in equation \ref{E:rankdef}.

\begin{lemma}\label{L:secinrank}

We have a scheme theoretic inclusion
$ \sec^{p-1}(L) \hookrightarrow R^p $.
\end{lemma}

\begin{proof} 
 
To show scheme theoretic inclusion we need to show an   inclusion
of the ideal sheaves:
$\mathcal{I}_{R^p} \hookrightarrow \mathcal{I}_{Sec} $ .  Since $R^p$ is
defined by the vanishing of all the $(p+1) \times  (p+1)$ minors of 
the generic matrix, we need to show that $\sec^{p-1}(C)$ vanishes
on all $(p+1) \times  (p+1)$ minors.  Roughly speaking, any
 point $x \in \sec^{p-1}(C) $ is a linear sum of $p$ points of $C$.
Each point of C represents a rank one linear transformation. Hence
each $x \in \sec^{p-1}(C) $ has a representation as the sum of $p$
rank one transformations and hence is of rank at most $p$.  Thus 
every point  $x \in \sec^{p-1}(C) $ will vanish at all $(p+1) \times (p+1) 
$ minors and hence scheme theoretically lies in $R^p$. Because this 
point comes up frequently, I will give it a separate formal proof.
The result is well-known (see (\cite{eks}) for example).

\end{proof}
First we fix some notation.
Let $W \subset \Hom(V_1,V_2)$ be a linear space and let $R \subset W$
be a subset consisting of rank one transformations.  This means that for
every $r \in R $, the kernel of $r$ is of codimension one and that the 
image of $r$ is of dimension one.  By $\sec^{p-1}(R)$ we mean all linear combinations of
 $p$ elements of $R$.

\begin{lemma}\label{L:rankj}

Every element  $r_p \in \sec^{p-1}(R)$ is of rank at most $p$.  That is 
if we 
fix a basis for $V_1$  and $V_2$  every $(p+1) \times (p+1)$
 minor of $r_p$ vanishes. Informally: every sum of at most
$p$ rank one matrices is of rank at most $p$.
\end{lemma}
\begin{proof}
If $p+1 > \min (\dim(V_1),\dim(V_2)$ the result is trivial.  So we
assume that $p+1 \leq \min (\dim(V_1),\dim(V_2)$
Once we have picked a basis we can represent any $r \in R$ as a matrix 
$(a_{ij})$ with a unique $ a_{ij} \neq 0$.  Hence any $r_p \in \sec^{p-1}$
can have at most $p$ columns (or rows) with a non-zero entry.  If we take a
$(p+1) \times (p+1)$ submatrix and expand along a  column
(or row) with all 
zeros we see that the determinant vanishes.
\end{proof}

Since $ \sec^{p-1}(C) $ and $R^p$ agree as sets, if $R^p$ is 
reduced then since $R^p$ is a thickening of $\sec^{p-1}(C)$ both
varieties are isomorphic having the reduced scheme structure.

There is a fair amount of literature on the scheme structure of
$\sec^p(C)$.  See  \cite{B}, \cite{Sch}, \cite{V2}
 for example.  It is known to be 
normal in many circumstances.  The first theorem on the subject is due
to Bertram.

\begin{thm}

If $L$ separates $2p$ points then $\sec^{p-1}(C)$ is normal and smooth away 
from $\sec^{p-2}(C)$. 
\end{thm}

\begin{proof}
see \cite{B}

\end{proof}

\textbf {Remark:} In case $L=K_C$ or $\deg(L) \geq 2g+1$ and $p< \cliff(C,L)$, one 
easily checks 
that $L^{\otimes 2}$ separates $2p$ points and hence that $\sec^{p-1}(C)$
is normal.  

The fact that $\sec^{p-1}(C) =  R^p$, a rank locus,  may suggest
that in fact these varieties are Cohen-Macauley.  Very little is known
about this.  Recently Sidman and Vermeire \cite{S-V}  have proven
that if $\deg(L) \geq 2g+3$ then $\Sec^1(C)$ is Cohen-Macauley.
Further Vermeire in \cite{Ve2} has shown that if $\deg(L)$ is sufficiently large, 
then $\sec(C)$ is generated by cubics.  His bound is better than ours in complete analogy
to the fact that,  once $\deg(L) \geq 2g+2$ then $C$ is generated by quadrics, but one 
needs $\deg(L)$ to be about $4g+4$ before the equations defining $C$ are 'determinantally 
presented'.

Our strategy is to construct a desingularization  $\tilde R^p$ of
$R^p$, which is analagous to the canonical desingularization presented in \cite{ACGH}.  
We  show that the variety we construct, $\tilde R^p$, coincides with the full secant variety
$B^{p-1}(L^{\otimes 2})$.   It follows that 
$\tilde R^p$ is smooth.  We then  show directly 
 by analyzing the fibers of $p: \tilde R^p
\to R^p$ that $R^p$ is normal.  This involves a reasonably
complicated tangent space calculation to show that the
scheme theoretic fibers of $p$ are reduced.  Once $R^p$ is 
normal, it is reduced and must agree scheme theoretically with
$\sec^{p-1}(C)$.

The methods of this section are similar to the 'geometric technique'
of Kempf. This provides for a way to desingularize rank loci.
 I first learned of these ideas in \cite{ACGH}.  Weyman in his
book \cite{W}  proves a theorem giving one a resolution of the structure sheaf of $\tilde R^p$.  Under some circumstances this resolution may be used to construct a resolution of $R^p$.
  This may allow one to prove that $R^p$ is Cohen-Macauley.

\textbf {Remark}
If it  could be  directly  shown that $R^p$ is reduced, then one could
conclude directly the scheme theoretic equality.  Using Bertram's Theorem,
one would have normality too.  I have not been able to create a simple argument
to prove this simpler fact.  The current proof does have the advantage
of explicating the geometry of $R^p$.  Namely, it shows that the resolution
is isomorphic to the full secant variety and that it has reduced, and in fact smooth,
fibers.

The first step in the construction of $\tilde R^p$ is the fact that 
$R^p$ is smooth away from $R^{p-1}$.  We first note:

\begin{lemma}

Suppose $L=K_C$ or $\deg(L)\geq 2g+1$  ,then  $\cliff(C,L)>p $ implies
$\cliff(C,L^{\otimes 2}) > 2p$.

\end{lemma}
\begin{proof}
If $L=K_C$, then $\cliff(K_C) \leq \frac{g-1}{2}$ and   $K_C^{\otimes 2} $ separates any $g-1$ points. 
 If $ L \neq K_C$ then if  $\deg(L) = 2g-2 +d$ with $d<g$ , then
$\cliff(C,L)\leq  d-1$ and if $d\geq g$ then $\cliff(C,L)= d-2$.  
Since $\deg(L^{\otimes 2}) = 4g-4 +2d$, $\cliff(C,L^{\otimes 2}) = 2g-4 +2d-2 >
2(d-1)$ .
\end{proof}

We can now prove:
\begin{thm}\label{T:smoothaway} 
    
    If $p < \cliff(C,L)$ and $\varphi \in R^p \setminus R^{p-1}$,
    then $R^p$ is smooth at $\varphi$.  In fact 
$T_{\varphi,R^p} = \overline{2D}_{L^{\otimes 2}}$, that is, the tangent
    space to $R^p$ at $\varphi$ is the $2p - 1$ dimensional space spanned by
    the divisor $2D$.
  
  \end{thm}

  \begin{proof}

If $\varphi \in R^p \setminus  R^{p-1}$ with $p< \cliff(C,L)$ then $\varphi \in
T_L^*(D)$ with $\deg(D) = p$ for some divisor $D$.  Hence $\ker(\varphi)=
\H{C,L(-D}$ and $\im(\varphi)= S_L(D)$.  We will work with   $\tilde \varphi
$, a lift of $\varphi$ to $\Hi{C,K_C \otimes L^{-2}}$ and show that the tangent space
$\tilde T_{\tilde \varphi,R^p} \subset \Hi{C,K_C \otimes L^{-2}}$ is the correct $2p$
dimensional space.  By Theorem \ref{T:linspace} we identify $\im(\varphi)$ with
$\{ x \in \Hi{K_C \otimes L^{-1}} \, | \, x \to 0 \in \Hi{K_C \otimes L^{-1}(D)} \}$
and hence, if we denote by $r_D$ the map\\
 $r_D: \Hi{K_C \otimes L^{-2}} \to 
\Hom(\H{L(-D)},\Hi{K_C \otimes L^{-1}(D)})$\\  
we may identify 
 $\tilde T_{\tilde \varphi,R^p}$   with $\ker(r_D)$.
We can factor $r_D$ as \\
\begin{center}
$ \Hi{K_C \otimes L^{-2}} \overset{ \alpha} \to \Hi{K_C \otimes L^{-2}(2D)} 
\overset{\beta} \to $\\
$ \Hom(\H{L(-D)},\Hi{K_C \otimes L^{-1}(D)})$.\\
\end{center}
Notice that $\ker(\alpha) = \partial (\H{K_C \otimes L^{-2}(2D)_{|2D}})$ that is 
to say   $\overline{2D}_{L^{\otimes 2}}$.  Since $\ker(\alpha) \subset \ker(r_D)$'
to finish  we
need to check that $\beta$ is injective.  We have $p< \cliff(C,L)$ and hence 
$\cliff(C,L(-D)) \geq 1$ .  In particular $L(-D)$ is arithmetically normal, which 
is equivalent to $\beta$ being injective by Serre Duality.  
\end{proof}

The next step is to construct a resolution $\tilde R^p$ of $R^p$

\begin{thm}
 Let $\~R^p = \left\{ (\varphi,\lambda) | \im\left(
    \varphi \right) \subset \lambda \right\} \subset R^p\times G(p,n)$. Then
    $\~R^p$ is smooth .
\end{thm}

\begin{proof}

  $\~R^p = \left\{ \left( \varphi, \lambda \right) | \im(\varphi) \subset
  \lambda \right\}$. Of course we mean by this that if we can represent
  $\lambda$ as $\sigma_1 \wedge \dots \wedge \sigma_p \in \bigwedge^p(V)$ then
  $\im \varphi \subset \left< \sigma_1,\dots,\sigma_p \right>$. Notice that
  $\~R^p \subset R^p\times G(p,n) \subset R^p \times \PP( \bigwedge^p(V))
\subset \PP(M_2) \times \PP(\bigwedge^p(V)) $. There is a Koszul map
  \[\PP(\sym^2(V)) \times \PP(\bigwedge^p(V)) \maps \PP(V \otimes
\bigwedge^{p+1}(V))\]
given by:

  \begin{eqnarray*}\label{E:2}
    \sym^2(V) \otimes \bigwedge^p(V) 
    &\stackrel{\cup}{\maps}&
    V \otimes \bigwedge^{p+1}(V) \\
    (v_1 \cdot  v_2) \otimes \lambda
    &\longmapsto&
    v_1\otimes (v_2 \wedge \lambda) + v_2 \otimes (v_1 \wedge \lambda)
\quad(2).
  \end{eqnarray*}

The inclusions $R^p \hookrightarrow \PP(M_2) \hookrightarrow 
\PP(\sym^2(V))$  induce a map
\begin{equation} 
R^p \times G(p,n) \stackrel{\cup}{\maps} \PP(V \otimes
\bigwedge^{p+1}(V))
\end{equation}
We first observe that $\tilde R^p$ can be considered as an 
incidence variety.

\begin{lemma}\label{L:characterizerank}

With this notation $\~R^p = \left\{ \left( \varphi,\lambda \right) |
  \varphi\cup\lambda=0 \right\}$. This is true as long as we are not
  working in characteristic 2.

\end{lemma}

\begin{proof}
Fix some $\lambda$ say $\lambda = v_1 \wedge \dots \wedge v_p$
with $ v_i \in V$ linearly independent. Let $\langle v_1  
\dots v_p \rangle = W \subset V$.  From equation ( \ref{E:2}) 
 if $\varphi \in
\sym^2(W)$ then $\varphi \cup \lambda = 0$ since $\varphi =
\sum_{i=0}^p a_iv_i^2 $ with $a_i$ constants (this is true after changing our basis of $W$).
So we may assume that $\varphi$ contains no terms entirely in $W$.  
Extend $\langle v_1 \dots v_p \rangle$ to a basis $\langle v_1 \dots v_n \rangle$ of $V$. Write
$\varphi = \sum_{i=p+1}^n a_il_i\otimes v_i + \sum_{i \geq j \geq(p+1)} b_{ij}v_i \otimes v_j$ with $l_i \in W$.  
Then
\begin{equation}
\varphi \cup \lambda = 
\sum_{i=p+1}^n a_il_i \otimes v_i\wedge \lambda + \sum_{i\geq j \geq (p+1)}b_{ij}(v_i \otimes v_j\wedge \lambda + v_j \otimes v_i \wedge \lambda)
\end{equation}
Each of the terms are linearly independent and (as long as char(k)  
$\neq 2) $ $\varphi \cup \lambda$ vanishes if and only if all the 
$a_i$ and all the $b_{ij}$ are zero. If char(k)$=2$ then the coefficients
of $b_{ii}$  are zero because they are divisible by two.
\end{proof}

To proceed we  formalize the idea that we can consider $\sym^p(C)
  \hookrightarrow \PP\left( \bigwedge^p(V) \right)$, via the association of a
  degree $p$ divisor to the $p-1$ dimensional space $\overline{D}$ it spans.

  \begin{lemma}
  
    Let $g:\sym^p(C) \maps G(p,n)$ be the natural map associating $D \mapsto
    \overline{D}$; then for $p\le \cliff(L,C)$ this map is an embedding.
  
  \end{lemma}
  
  \begin{proof}
   
    Following  \cite{B}  we construct a rank $p$ bundle $B^{p-1}(L)$ over $\sym^p(C)$.
    Let $\mathcal{D}_p \hookrightarrow C \times \sym^p(C)$ be the universal
    divisor, $\pi_2: C\times \sym^p(C) \maps \sym^p(C)$, then $\mathcal{E} =
    \pi_{2*}\left( \pi_1^*(L)|_{\mathcal{D}_p} \right)$. Informally this is
    the rank $p$ bundle $D \mapsto \H{C, L|_D}$.   Since $\pi_{2*}\pi_1^*(L) =
    \H{L} \otimes \OO_{\sym^p(C)}$ and $p\le\cliff(C)$ implies $L$ is at least
    $(p+1)$ spanned the map  $\H{L}\otimes\OO_{\sym^p(L)} \twoheadrightarrow
    \mathcal{E}|_L$ is surjective. Since a map $g: \sym^p(C) \maps G(p,n)$ is given by a
    rank $p$ bundle and n sections generating the bundle, this gives a map $g:
    \sym^p(C) \maps G(p,n)$. Again because $L$ separates atleast $p+1$ points,
    $\overline{D}_1 \ne \overline{D}_2, \forall D_1, D_2 \in \sym^p(C)$ and
    so the map is set theoretically one to one. Using the identification of
    \cite{ACGH}, we identify $T_{D,\sym^p(C)}$ with $\H{C, \OO_D(D)}$ and
    $T_{D,G(p,n)}$ with $\Hom( \H{L(-D)}, \H{L|_D})$. Then  (cf. \cite{ACGH})
     the map on tangent spaces, is the map ``cup product'', i.e.
    \[
      t \in \H{C, \OO_D(D)}
      \longmapsto
      \cup t : \H{C, L(-D)}\rightarrow\H{C,L|_D}.
    \]
    Since $L$ separates $(p+1)$ points, this map is an injection as required.
    If $D$ is smooth this is well known. We prove the case $D=k q$ for
    completeness. Let $z$ be a local parameter at $q$. $\H{\OO_D(D)} = \left<
    z^{-1},\dots,z^{-k} \right>$, $\H{L|_D} = \left< l, zl, \dots,
    z^{k-1}l \right>$ where $l$ is a local section of $L$ at $q$. Since $L$ is
    at least $(k+1)$ ample, there exists a section of $L$ which locally at $q$
    looks like $z^kl$ and the cup product is now multiplication. Since
    $z^{-i}z^kl = z^{k-i}l$ for $1\le i\le h$, multiplication gives rise to linearly
    independent elements of $\H{L|_D}$ the map is injective.
  
  \end{proof}

\begin{cor} \label{C:characterizesec}
Let $\lambda_D \in \bigwedge^p(V)$ represent the image of $D \in \sym^p(C)
$ in $G(p,n)$.
Then $B^{p-1}(L) \subset \PP(V)\times \PP(\bigwedge^p(V)) =
\{(v,\lambda_D)\,|\, v\wedge \lambda_D=0  \in 
\bigwedge^{p+1}(V)          \}. $
\end{cor}
\begin{proof}
$B^{p-1}(L)$ is the pullback to $\sym^p(C)$ of the universal subbundle on
$G(p,n)$.  That is $B^{p-1}(L)= g^*(\mathcal{S})$ where $\mathcal{S} =
\{(x,\lambda) \in V \times \wedge^p(V) \, |\, x \wedge \lambda =0
\in \wedge^{p+1}(V) \}$ . Since $B^{p-1}(L)$ is just the restriction of 
$\mathcal{S}$ to the embedding of $\sym^p(C)$ in $G(p,n)$ this is clear.

\end{proof}

\textbf{Remark:} this map is \emph{not} the usual Gauss map. We can include
  $C$ in $\sym^p(C)$ via the diagonal, i.e. $q\mapsto p q$. As Voisin proved
  cf. \cite{V1}  $\H{C, \bigwedge^p(B^{p-1}(L))} = \bigwedge^p(\H{C,L})$
  whereas the rank $p$ bundle associated with the Gauss map is $P^{p-1}(L)$
  (the jet bundle) and for example for $p=2$ $\bigwedge^2\left( P^1(L) \right)
  \simeq L^{\otimes 2}\otimes K_C$ and so the two bundles have different global sections.

To finish the proof that $\tilde R^p$ is smooth,we will identify  
$\tilde R^p$ with  $B^{p-1}(L^{\otimes 2})$ as schemes. We have characterized 
$B^{p-1}(L^{\otimes 2})$ in Corollary \ref{C:characterizesec} so by Lemma
\ref{L:characterizerank} it is enough to prove the following.

\begin{lemma}
Let $\tilde  \lambda_D$ denote the image of $D \in \sym^p(C)$ in 
$G(p,V)$ and let $\lambda_D $ denote the image of $D \in \sym^p(C)$
in $G(p,M_2)$.  Then for $\varphi \in M_2$, we have $\varphi
\cup \tilde \lambda = 0$ if and only if $\varphi \wedge \lambda =0$.
\end{lemma}
\begin{proof}
Suppose $D = \sum n_iq_i$ where $n_i$ are integers such that
$\sum n_i = p$ and $q_i \in C$.  Then $\varphi \wedge \lambda_D=0$
if and only if $\varphi \in T_L(D)$.  Since $\varphi$ is symmetric,
this is true  if and only if im($\varphi$)
$\subset \tilde \lambda_D$
\end{proof}

We now have a scheme theoretic identification of $B^{p-1}(L^{\otimes 2})$
with $\tilde R^p$.  Since $B^{p-1}(L^{\otimes 2})$ is smooth, so is
$\tilde R^p$.
\end{proof}

To finish the proof we show that $R^P$ is normal.  We use a result that we
learned in \cite{Ve4} namely:

\begin{lemma}
Let $f:X \to Y$ be a proper surjective morphism of irreducible varieties over an
algebraically closed field, with reduced and connected fibers.  If $X$ is normal, then
$Y$ is normal.
\end{lemma}
\begin{proof}
See \cite{Ve4} for details.  
\end{proof}

We will actually prove that the fibers of $p: \tilde R^p \to R^p$ are smooth.

\begin{thm}
 Let $\varphi \in R^k \setminus  R^{k-1}$ , then $p^{-1}(\varphi) $ is scheme theoretically
isomorphic to $\sym^{p-k}(C)$
\end{thm}

\begin{proof}

Because $k \leq p < \cliff(C,L)$ we can write $\varphi = \sum_{i=1}^k \sigma_{p_i}^2 $ where 
$\sigma_{p_i}^2$ is a Shiffer variation supported at $p_i \in C$. Denote by $D$ the divisor spanned by $
\langle p_1,
\dots p_k \rangle $.
 We analyze $p^{-1}(\varphi)$
as follows:

$$ \cup \varphi : \bigwedge^P(V) \longrightarrow \bigwedge^{p+1} \otimes V $$
is the map on points $\lambda \to 2(\sum_{i=1}^k \sigma_i \otimes (\sigma_i \wedge \lambda))$
First extend $\sigma_1 , \dots , \sigma_k$ to a basis $\sigma_1, \dots , \sigma_k, \dots \sigma_n$
of $V$, and set $ V' = \langle \sigma_{k+1}, \dots ,\sigma_n \rangle$, so $\ker(\varphi)
\backsimeq \sigma_1 \wedge \dots \wedge \sigma_k \otimes \bigwedge^{p-k}(V') = \tilde V$.
We also have a map of schemes:

$$ \sym^{p-k}(C) \stackrel{\varphi_D} \longrightarrow \sym^p(C) E \longmapsto E+D$$

$\varphi_D$ identifies $\ker(\cup \varphi) \cap \sym^p(C)$ with $\sym^{p-k}(C)$ because 
$ if \lambda \in \sym^p(C)$, considered as a subset of $G(p,n)$ then $ \sum_{i=1}^k \sigma_{p_i}^2
\cup \lambda = \sum_{i=1}^k \sigma_i \otimes \sigma_i \wedge \lambda =0$ if and only if $\sigma_i \wedge \lambda =0$
for all $i$.  We need to see that this is an equality of schemes.  
By induction it will be enough to consider
the case $k=1$ because we can write a general $\varphi_D$ as a composition of $\varphi_{\sigma_{p_i}}$ for the different
$p_i \in D $.  So we assume that $\varphi = \sigma^2 $ for a specific $\sigma = \sigma_{p_i} $ with $p_i \in C $.

Recall the relevant identifications of tangent spaces.
$T_{\sym^p(C),D} = H^0(\cal{O}(D)|_{D})$ and $T_{G,D} = \Hom (\H{L(-D)},\H{L|_{D}})$.  Suppose $p$ occurs with 
multiplicity $k$ in $D$.  Then in fact our fiber is the scheme theoretic intersection of $\sym^p(C)$ and $\PP_{\sigma}$ where $\sigma$  is any element in the span of the divisor $D$ that is not in the span of $D-p$ and $\PP_{\sigma}$ denotes $\{ x \in \bigwedge^p(V) | \sigma \wedge x = 0\}$ . To show that the intersection is smooth, we must find an element of the tangent space of $\sym^p(C)$ which is not in the tangent space of $\PP_{\sigma}$.  This forces the tangent space of the intersection to be at most of dimension $p-1$ and hence smooth, since the intersection is set theoretically $\sym^{p-1}(C)$.
If $\varphi $ is in 
$T_{\PP_{\sigma}} \cap T_{G(p,n)}$  then $\varphi \in \Hom(\H{L(-D)}), \H{L|_{D}})$,  is also  in $T_{\PP_{\sigma}}$ if and only if $ \sigma \notin \varphi(\H{L(-D)} \subset \H{L|_{D}}$. 
The element $\frac{1}{z^k} \in \H{\cal{O}_D(D)}$  contain such $\sigma$ in its image and hence is not in
the tangent space of the fiber $p^{-1}(\varphi)$.

\begin{cor}
As schemes, $\sec^{p-1}(C) = R^p$.
\end{cor}
\begin{proof}

 $B^{p-1}(L^{\otimes 2})$ and  $\tilde R^p$
 are defined as subsets of $\PP(M_2) \times \PP(\bigwedge^p(V))$.  $\sec^{p-1}(C)$ and $R^p$ are defined as the projection onto 
$\PP(M_2)$ of these two schemes.  Since the two schemes are identified, so
are their (reduced)  projections.
\end{proof}

\textbf{Remark}:The conclusion of the theorem is probably as strong as possible. 
We know for generic $L$ with $\deg(L)<3g-2$ and for all $L$ with
$\deg(L) \geq 3g-2$ the theorem is sharp as the 
  two schemes do not even underly the same set in this case. It is probably the 
case that this is always so.

\end{proof}

\section[Clifford Index of a Pair of Line Bundles] {The Clifford Index of a Pair of Line Bundles and a Conjecture of Eisenbud, Koh, and Stillman}\label{S:eks}

In their paper (\cite{eks}),
Eisenbud, Koh and Stillman   considered the following situation. Let $\mathcal{L}_i =(L_i,V_i)\,  i=1,2$ 
be two linear series on a curve -- that is $L_i$ is a line bundle 
on $C$ and  $V_i \subset \H{C,L_i}$ is a linear subspace. 
Let $\mathcal{L}_1\centerdot
\mathcal{L}_2 $ be the linear series
$ (L_1 \otimes L_2, V = \im(V_1 \otimes V_2 \stackrel{\mu} \to \H{C,L_1 \otimes L_2})$. That is $\mathcal{L}_1\centerdot
\mathcal{L}_2 $ represents the line bundle $L_1\otimes L_2$ 
with the sub- linear series generated by $V_1\otimes V_2$.  By 
the linear series generated by $V_1\otimes V_2$
we mean the natural map $ \H{C,L_1} \otimes \H{C,L_2} \stackrel{\mu}
\to \H{C,L_1\otimes L_2}$ restricts to a map $\mu : V_1 \otimes V_2
 \to \H{C,L_1\otimes L_2}$. $\mu(V_1\otimes V_2) \subset \H{C,L_1
\otimes L_2}$ is the linear series generated by $V_1\otimes V_2$

If $\{e_i\}$ and $\{f_j\}$ are bases for $V_1$ and
$V_2$, then $M = \{\mu( e_i \otimes f_j)\}$ can be considered as a 
 matrix of linear form with entries in $\H{C,L_1\otimes L_2}$ which represents a basis for 
$\mu(V_1\otimes V_2)$ .
Writing $I_2(M)$ for the ideal of $2 \times 2$ minors of $M$ in $S = \sym(V)$,
one has that $I_2(M) \subset I_2(C)$, where $I_2(X)$ is equations of
degree 2 in the ideal of $X$. (Proof: Let $m_{ij} = \mu(e_i \otimes f_j)$
Then the equations of $I_2(M)$ are $m_{ij}\centerdot  m_{kl} -m_{ik}\centerdot  m_{jl}=0$
which  clearly vanish on $C$).   We now  assume that $V_i = \H{C,L_i}$.

\begin{defn}[\cite{eks}:]

  If $C$ is defined by quadratic equations and $I_2(M) = I_2(C)$, we say
  $\mathcal{L}_1\centerdot  \mathcal{L}_2$ is a determinantal presentation 
of $C$ and that $C$ is determinantally presented.
\end{defn}

Eisenbud-Koh-Sullivan then prove
if $\deg(L_i)$ are ``large'' then $\mathcal{L}_1\otimes\mathcal{L}_2$ is determinantedly presented.  For ease of reading we will denote 
$\mathcal{L}_1\otimes\mathcal{L}_2$ by $L_{12}$.  Notice that if $C$ is embedded in $\PP(\H{C,L})$,
the question of whether $C$ is determinantally presented depends on the factorization of $L= L_1 \otimes L_2$,
in particular there can be infinitely many such presentations.  See \cite{eks} for a fuller discussion.

We explain how this result is related to our theorem and how their result can
be extended to $\Sec^k(C)$ in an appropriate range of $k$.

We can create  a diagram for  $L_{12}$ which generalizes our 
standard diagram.

\begin{displaymath}
  \xymatrix{
  & \PP(V_1^{\vee}) \times \PP(V_2^{\vee}) \ar[dr]^{\sigma} & \\ 
  C \ar[ur]^{(\varphi_1,\varphi_2)} \ar[dr]^{\psi} &  
  & \PP(V_1^{\vee} \otimes V_2^{\vee}) \\
  & \PP(V\vee) \ar@{^{(}->}[ur]^{\iota} & }
\end{displaymath}

Here $\varphi_i : C \to \PP(V_i)$ and $\psi : C \to \PP(V)$ are the maps given
by the linear series; $\iota$ is the inclusion induced by the surjection of $V_1
\otimes V_2 \stackrel{\mu}{\to} V$, $\mu$ is surjective since $V$ is generated
by $\mu(e_i \otimes f_j)$; and finally,  $ \sigma$ is the Segre map.

If we identify $(V_1 \otimes V_2)\vee$ with $\Hom(V_2,V_1\vee)$, then
$\sigma(\PP(V_1\vee) \times \PP(V_2\vee))$ can be identified with the set of
``rank one'' matrices in $ \PP(V_1\vee \otimes V\vee)$. That is to say the matrices
such that all their $2 \times 2$ minors vanish, or alternatively the subvariety of all
 $\varphi$ such that
$\dim(\im(\varphi)) = \codim(\ker(\varphi)) = 1$.

In general, we will define the rank $j$ locus for $j \le \min\{\dim(V_1),
\dim(V_2)\}$ as the  variety defined by all the $(j+1) \times (j+1)$ minors. If
$L_1 = L_2 = L$, $V_i = \H{C,L_i}$ we recover our standard diagram.

The basic observation of \cite{eks} is that because of this diagram, $\psi(C)
\subset \PP(V\vee) \cap \sigma(\PP(V_1\vee) \times \PP(V_2\vee)) = R^1(C)$.
Since $\psi(C)$ is contained in the rank one locus, the ideal of $\psi(C)$
contains the ideal of $\sigma(\PP(V_1\vee) \times \PP(V_2\vee))$, which is
generated by determinantal quadrics! The main theorem of \cite{eks} is:

\begin{thm}[\cite{eks}:] \label{T:eks}
  
  Let $\deg(L_1)$, $\deg(L_2) \ge 2g + 1$, and if $\deg(L_1) = \deg(L_2) = 2g
  + 1$, assume $L_1$ and $L_2$ are not isomorphic. Then $\psi(C) =
  \iota(\PP(V\vee)) \cap  \sigma(\PP(V_1) \times \PP(V_2))$ as schemes, and
  so $\psi(C)$ is determinantally presented.

\end{thm}

Further in (\cite{eks}) it was conjectured:

\begin{conj} 
Let C be a curve of genus g, 
and $L$ a line bundle which can be factored as $L= L_1\otimes L_2$
for some choice of  line bundles $L_1$ and $L_2$.Then there is a constant $k_0$ depending on the genus of C and the degrees of the $L_i$ such that the variety $Sec^k(C)$ is determinantally presented for $k \leq k_0$. 

\end{conj}

Theorem \ref{T:eks} can be interpreted as a version  of our theorem for 
$L_1 \ne L_2$ and for $C= \sec^0(C) =R^1$.
 M.S. Ravi in \cite{Ravi} gave a partial answer to this conjecture.

\begin{thm}[Ravi]

  Suppose $\deg(L_1), \deg(L_2) \ge 2g + 1 + k$ and $\deg(L_1 \otimes L_2) \ge
  4g + 3 + 2k$. Then set-theoretically, $\Sec^k(C)$ is defined by
  $I_{k+2}(M)$, that is, $\Sec^k(\psi(C)) = R^{k+1}(C)$. 

\end{thm}

By generalizing  the techniques of Shiffer variations and the Clifford index of a line bundle from the case of $L^{\otimes 2} = L\otimes L$ to the case of 
$M = L_{12}$ we can improve this result.

\begin{thm} \label{T:mainresult}
  
  Suppose $\deg(L_1), \deg(L_2) \ge 2g + 1 + k$ and $\deg(L_1 \otimes L_2) \ge
  4g + 2 + 2k$. Then as a scheme $\Sec^{k-1}(C)$ is defined by $I_{k+1}(M)$. That
  is, as schemes, $\Sec^{k-1}(C) = R^{k}(C)$, the rank $k$ locus.

\end{thm}

As in the case $L_1 = L_2$, the proof proceeds in a number of steps.
We first define and prove the basic properties of Shiffer variations for $L_1
\ne L_2$.  Then one proves a set theoretic equality of the schemes
$\sec^j(C)$ and $R^{j+1}$.  As before the scheme $\sec^j(C)$ is reduced and
includes in $R^{j+1}$, so the only issue is to show that  $R^{j+1}$ is reduced.  
The proof proceeds exactly as in the case of $L_1 = L_2$, by showing
that one has two resolutions of $R^p$ that agree scheme theoretically.
  Finally since $\deg(L)$ is very large, one can 
show the existence of an appropriate factorization of $L$ so that Theorem
\ref{T:mainresult} is true for some choice $L_1$ and $L_2$ and one 
obtains,

\begin{cor}
Suppose $C \hookrightarrow \PP(\Hd{C,L)}$ is embedded
 by the complete linear system 
$|L|$ and that $\deg(L) \geq 4g+2k$. Then for any $j \leq k$, $\sec^j(C) $
is determinantally presented and so is defined by 
equations of degree $j+2$.  
\end{cor}

\textbf{Remark } More
formally, we let $S^1 \subset \PP(V_1\vee \otimes V_2\vee)$ be
$\sigma(\PP(V_1\vee) \times \PP(V_2\vee))$, and set $S^j =
\Sec^{j-1}(\sigma(\PP(V_1\vee) \times \PP(V_2\vee))) \subset \PP(V_1\vee \otimes
V_2\vee)$.   $S^j$ is defined by the vanishing of all the $(j+1) \times (j+1)$
minors. We then set $R^j(C,L_1,L_2) = S^j \cap \PP(\Hd{C,L_1 \otimes L_2})$.
When $C$, $L_1$, $L_2$, etc. are clear we will usually just write $R^j$ or
$R^j(C)$. The corollary follows by  taking $L_1$ to be any line bundle
of degree '$\frac{1}{2} \deg (L)$' and $L_2 = L \otimes L_1^{-1}$.  Notice this 
gives rise to infinitely many determinantal presentaions.

The basic idea is the same as for $L_1=L_2$. For any $p \in C$, the 
 image of $\psi(p) $ 
in \\ $\PP(\Hi{C,K_C
\otimes L_1^{-1} \otimes L_2^{-1}})$ will correspond to a rank one matrix that
has kernel $\H{C,L_1(-P)}$ and image equal to $\partial(\H{C,K_C \otimes
L_2^{-1}(P)|_P}) \subset \Hi{K_C \otimes L_2^{-1}}$. If $D$ is any divisor we
will define Shiffer variations supported on $D$ and the rank of any such matrix
will depend on $d$, $r_{L_1}(D)$, and $r_{L_2}(D)$. 

We work out an example before doing things in general.

Let $L_1$ and $L_2$ be 2 line bundles of degree $2g + 1$, and let $D = p_1 +
p_2 + p_3$, $p_i \in C$. 
Let $s_1, s_2, s_3 \in \Hi{K_C \otimes L_1^{-1}}$, $v_1, v_2, v_3 \in \Hi{C,
K_C \otimes L_2^{-1}}$, and $t_1, t_2, t_3 \in \Hi{C,K_C \otimes L_1^{-1}
\otimes L_2^{-1}}$ be elements representing $p_1, p_2, p_3$. Recall
that for any line bundle $M$ on $C$ an element of $m\in \Hi{K_C\otimes M^{-1}}
=\Hd{M}$ represents the point $ p \in C$ means that that $m$ kills
$\H{M(-p)}$.

\begin{lemma}

The  natural map $\H{C,L_1} \otimes \H{C,L_2} \to \H{C,L_1\otimes
 L_2}$ dualizes to give a map:
 $$ \Hi{C,K_C \otimes L_{12}} 
\stackrel{\mu} \to
\Hi{C,K_C \otimes L_1} \otimes \Hi{C,K_C \otimes L_2}$$
 Then as long as
$\ch(k) \ne 2$, $t_i = \mu(s_i \otimes v_i)$. 
\end{lemma}
\begin{proof}

The map $\mu$ can also be considered as a map $$\mu : 
\Hi{C,K_C \otimes L_{12}} \to 
\Hom(\H{L_1},\Hi{K_C \otimes L_2^{-1}}$$
With this description it is clear that $\mu(s_i \otimes
 v_i) $ kills $\H{L_1(-p_i)}$ and has image generated by $v_i$
which is the action of $t_i$.

\end{proof}

We consider 3 cases:

\begin{enumerate}

  \item[(i)] $\h{L_i(-D)} = (g+2) - 3 = g - 1$, for $i = 1, 2$.

  \item[(ii)] $\h{L_1(-D)} = g - 1$, $\h{L_2(-D)} = g$.

  \item[(iii)] $\h{L_1(-D)} = \h{L_2(-D)} = g$.

\end{enumerate}

Case (i) is the generic case. Case (ii) can occur if $L_1 = K_C(p_1+p_2+p_3)$,
with $L_2$ general of degree $2g - 2$. Case (iii) can occur if $L_1 = L_2 =
K_C(p_1 + p_2 + p_3)$.

In case (i), $t = t_1 + t_2 + t_3$ is of rank 3 as $p_1,p_2,p_3, \in D$ 
are linearly independent  in both
embeddings.  That is to say that the $t_i$ represent linearly
independent rank one transformations.
 In case $(ii)$, $D$ spans only a line in $\PP(\Hi{K_C \otimes
L_2^{-1}})$ and $t$ is of rank $2$, because the image of the 
 $t_i$ lie in a 2-dimensional
space. In case (iii), we cannot be sure that rank of $t$ is greater than one. 

This picture generalizes without difficulty. The notation is a bit cumbersome.
Let  $D = \sum_{i=1}^d n_i p_i$ be a divisor of
degree $d$. Let $L_1$ and $L_2$ be line bundles such that \[2g+1 \le \deg(L_1)
\le \deg(L_2).\] We have an exact sequence \[0 \to K_C \otimes L_{12}^{-1} \to
K_C \otimes L_{12}^{-1}(D) \to K_C \otimes L_{12}^{-1}(D)|_D \to 0.\] Denote
by $T_{L_{12}}(D) = \partial(\H{C, K_C \otimes L_{12}^{-1}(D)|_D}) \subset
\Hi{C, K_C \otimes L_{12}^{-1}}$.

\begin{defn}

  $T_{L_{12}}(D)$ are the Shiffer variation for $(L_1,L_2)$ supported on $D$.

\end{defn}

\begin{defn}

  $\cliff(L_1;L_2,D) = d - r_{L_1}(D) - r_{L_2}(D)$.

\end{defn}

\begin{defn}

  $\cliff(C,L_1;L_2) = \min\{\cliff(L_1;L_2,D) \,|\, r_{L_1}(D) > 0 \text{ or
  } r_{L_2}(D) > 0\}$.
  
\end{defn}

\begin{prop}\label{P:cliffbound}

  \begin{enumerate}

    \item[(i)] $\min\{\cliff(C,L_1), \cliff(C,L_2)\} \le \cliff(C,L_1;L_2) $
      
    \item[(ii)]Suppose $\deg(L_1) = \deg(L_2)$, then 
$\cliff(C,L_1,L_2) \ge d - 1$, unless $L_1 = L_2 = K_C(D)$,
      with $D$ effective or $C$ is hyperelliptic and $L_1 = K_C(D - E_1)$,
      $L_2 = K_C(D - E_2)$, with $E_1$, $E_2$ both multiples of the $g^1_2$ on
      $C$.
\item[(iii)] Suppose $deg(L_1)=deg(L_2)\leq 3g-3$ and $L_1 \neq L_2$ 
are generic line bundles, then 
$\cliff(C,L_1,L_2) \geq d-1$.

  \end{enumerate}
  
\end{prop}

\begin{proof}

  \begin{enumerate}

    \item[(i)] Suppose $\cliff(C,L_1;L_2)$ is computed by $D$. Then if
      $r_{L_1}(D) \le r_{L_2}(D)$, then $$d - 2r_{L_2}(D) \le d - r_{L_2}(D) -
      r_{L_1}(D) $$

    \item[(ii)] Follows from (i) and Corollary \ref{C:compute_cliff} since those are the 
      only cases, for which $\cliff(C,L_1) \linebreak = d - 2$. Notice that if
$L_1 = K_C(D_1)$ and $L_2= K_C(D_2)$ with $D_1 \neq D_2$ then for example,
$r_{L_1}(D_2) = 0$ so that $\cliff(C,L_1,L_2)= d-1$

\item[(iii)] We must eliminate the cases in (ii) which are possible exceptions.
The only case to consider is $C$ is hyperelliptic and $L_i=K_C(D-E_i)$
where $D$ computes $\cliff(C,L_1,L_2)$.  But $\deg(L_1) = \deg(L_2)$
forces $\deg(E_1) = \deg(E_2) $ which means that $E_1 = E_2$ since the $g_2^1$
is unique on $C$. Hence $L_1 = L_2$.

  \end{enumerate}

\end{proof}

The results about the rank of a matrix in $T_{L_{12}}$, and the relationship
between $R^j(C)$ and $\Sec^{j-1}(C)$ are the same as in the case $L_1 = L_2$.
We will state the results and sketch the proofs.

\begin{thm}
  Let  $\xi \in T_{L_{12}}(D) \subset \Hi{C, K_C \otimes L_{12}^{-1} }$.
 Denote by $\rho$ the natural restriction $\H{C, L_1} \to
  \H{C, L_1|_D}$. Denote by $\partial_1 : \H{K_C \otimes L_2^{-1}(D)|_D} \to
  \Hi{K_C \otimes L_2^{-1}}$ the boundary map in the long exact sequence 
  \[
  0 \longrightarrow 
  K_C \otimes L_2^{-1} \longrightarrow 
  K_C \otimes L_2^{-1}(D) \longrightarrow 
  K_C \otimes L_2^{-1}(D)|_D \longrightarrow 0
  \]
  and denote by $\partial_2: \H{K_C \otimes L_{12}^{-1} (D)|_D}
  \to \Hi{K_C \otimes L_{12}^{-1}}$ the boundary map in the
  definition of $T_{L_{12}}(D)$. Let $\~\xi \in \H{C, K_C \otimes 
L_{12}^{-1}(D)|D  }$ be an element lifting $\xi$, i.e.
  $\partial_2(\~\xi) = \xi$. Then  $\cup\xi:
  \H{C, L_1} \to \Hi{C, K_C \otimes L_2^{-1}}$ factors as:
  \[
  \H{C, L_1} \stackrel{\rho}{\longrightarrow}
  \H{C, L_1|_D} \stackrel{\cup\~\xi}{\longrightarrow}
  \H{K_C \otimes L_2^{-1} (D)|_D} \stackrel{\partial_1}{\longrightarrow}
  \Hi{C, K_C \otimes L_2^{-1}}.
  \]
\end{thm}

\begin{proof}
  The proof is exactly the same. Pick an affine open cover of $C$ by $V_1$ such that
  $V_1\supset D$ and $V_2 = C-D$. If one uses this C\v{e}ch cover to compute the
  cup product then $\xi \in \Gamma\left( V_1 \cap V_2, K_C \otimes
  L_1^{-1} \otimes L_2^{-1} \right)$ is given by $\hat\xi$ where $\hat\xi$ is
  a lifting of $\~\xi$ to $\Gamma\left( V_1 \cap V_2, K_c \otimes
  L_1^{-1} \otimes L_2^{-1} \right)$. So given $s_1\in \H{C, L_1}$,
  $s_1\cup\xi$ is represented by $s_1\cdot\xi\in\Gamma\left( V_1 \cap V_2, K_C
  \otimes L_2^{-1}\right)$ which is $\partial_1(\~\xi \cdot s_1)$.
\end{proof}

\begin{cor} \label{C:genrange}
  Let $\xi \in T_{L_{12}}(D)$. Then $\ker(\xi) \supset \H{L_1(-D)}$ and
  $\im(\xi) \subset S_L(D)$, the affine cone over $\overline D$ in 
$\PP(\Hi{K_C \otimes L_2^{-1}}$.
\end{cor}

\begin{proof}
  This follows from Theorem  \ref {T:linspace} and the above description.
\end{proof}

The next theorem  calculates the rank of an element
$\tau \in T_{L_{12}}(D)$. Suppose that $r_{L_2}(D) = r_2 \ge r_{L_1}(D) = r_1$.

\begin{thm}\label{T:genrankcalc}
  Let $\xi \in T_{L_{12}}^*(D)$ , then:\[ d-r_1-r_2 \le \rk(\xi) \le d-r_2  \]
\end{thm}

\begin{proof}
  The condition that  $\xi \in T_{L_{12}}^*(D) $ 
is: write  $D=\sum_{i=1}^n n_i p_i$ and
  choose  $z_i$ a local parameter at $p_i$, then  we can write a lifting of
  $\xi$ to $\H{K_C \otimes L_1^{-1} \otimes L_2^{-1} (D)|_D}$ as $\~\xi = \sum
  _{i=1}^n \left( \sum_{j=1}^{n_i} \beta_{ij} z_i^j \right)$ with $\beta_{i,
  n_i} \ne 0$ for $1 \le i \le n$. By cor \ref{C:genrange}  $\im(\xi) \subset
  S_L(D)$, which is a linear space of dimension $d-r_2$ . This gives the upper bound.
  $\cup\~\xi$ is an isomorphism by the same argument as in Theorem 
\ref{T:rankcalc} 
so we are done by the following
  algebraic fact.
\end{proof}

\begin{lemma}
  Let $\varphi:V_1\to V_2$ be a linear isomorphism between two vector spaces
  of dimension $d$. Let $W_1\subset V_1$ be a subspace of codimension $r_1$
  and let $V_2 \to W_2$ be a surjection onto a space of dimension $d-r_2$.
  Then $\overline{\varphi}: W_1 \to W_2$ has rank $\ge d-r_1-r_2$.
\end{lemma}

\begin{proof}
  Because $\varphi$ is an isomorphism, $\overline{\varphi}_1:W_1 \to V_2$ has
  rank $d-r_1$ and  $p:V_2\to W_2$ has kernel of rank $r_2$. $\Ker\left( p\circ
  \overline{\varphi}_1 \right) \subset \ker(p)$ since $\overline{\varphi}_1$
  is injective and hence $\dim\left( \ker \overline{\varphi} \right) \le r_2$
  and so $\rk\left( \overline{\varphi} \right) = \rk\left( \overline{\varphi}_1 
  \right) - \dim \ker \left( p\circ\overline{\varphi}_1 \right) \ge
  d-r_1-r_2$.
\end{proof}

We next show set-theoretic equality of the appropriate secant varieties and
rank-loci. Having established the generalized notation, the proofs go exactly
as in the case of $L_1 = L_2$. Recall our setup
\begin{displaymath}
  \xymatrix{
  & \PP \left( \Hi{K_C \otimes L_1^{-1}} \right) \times
    \PP \left( \Hi{K_C \otimes L_2^{-1}} \right) \ar[dr]^{\sigma} & \\ 
  C \ar[dr]^{\psi} \ar[ur]^{(\varphi_1,\varphi_2)} &  
  & **[l]\PP \left( \Hi{K_C \otimes L_1^{-1}} \otimes \Hi{K_C \otimes L_2^{-1}} \right) \\
  & \PP \left( \Hi{K_C \otimes L_{12}^{-1} } \right)
  \ar@{^{(}->}[ur]^{\iota} & }
\end{displaymath}

  $\varphi_i:C \to \PP\left( \Hi{C, K_C \otimes L_i^{-1}} \right)$ and
$\psi : C \to \PP\left( \Hi{C, K_C \otimes L_1^{-1} \otimes L_2^{-1}} \right)$
are the maps given by the complete linear series, $\iota$ is the inclusion
induced by the surjection
\[
  \H{C, L_1} \otimes \H{C, L_2} \twoheadrightarrow \H{C,L_{12}}
\]
and $\sigma$ is the Segre embedding.

We view $\PP\left( \Hi{K_C\otimes L_1^{-1}} \otimes \Hi{K_C\otimes
L_2^{-1}} \right)$ as the projectivization of the space of matrices 
\[
  \Hom\left( \H{C,L_1}, \Hi{C,K_C\otimes L_2^{-1}} \right)
\]
and we define $S^j \subset \PP\left( \Hi{K_C\otimes L_1^{-1}} \otimes
\Hi{K_C\otimes L_2^{-1}} \right)$ as the rank $j$ locus given by the vanishing
of all the $(j+1)\times(j+1)$ minors. Then $R^j=R^j(C)=\iota^*(S^j)$ consists
of the rank $j$ locus in $\PP\left( \Hi{K_C\otimes L_1^{-1} \otimes
L_2^{-1}} \right)$. Since $p\in C$ represents a rank one matrix with kernel
$\H{C,L_1(-p)}$ and image $\partial\left( \H{C, L_2(p)|_p} \right) \subset
\Hi{C, K_C\otimes L_2^{-1}}$, we have $C = \Sec^0(C)\subset R^1$ and hence
$\Sec^{j-1}(C) \subset R^j$  because, as stated in Lemma  \ref{L:rankj},
 a sum of $j$ rank one matrices is of rank $\le j$.

\begin{thm}
  Suppose $\cliff(C,L_1,L_2) = c \ge 2$. Then, as sets, $\Sec^{j-1}(C) = R^j$
  for $j<c$.
\end{thm}

\begin{proof}
  The proof is essentially identical to the case $L_1 = L_2$. For completeness
  we give details.

  Since $\dim\left( \Sec^j(C) \right) = 2j+1$,
  \[
  \PP\left( \Hi{K_C\otimes L_{12}^{-1}}  \right) =
  \bigcup_{D \in \sym^k(C)} T_{L_{12}}(D)
  \]
  for $k \ge \frac{\h{C, L_{12}}+1}{2}$. In other words, every element
  can be written as a Shiffer variation in some $T_{L_{12}}(D)$ for some $D$ of degree $k$.
  Let $$T^*_{L_{12}}(D) = T_{L_{12}}(D) \setminus \left( \bigcup_{D'\subset D}
  T_{L_{12}}(D') \right)$$ be the elements of maximal rank. Set $r_1 =
  r_{L_1}(D)$ and $r_2 = r_{L_2}(D)$. By  Theorem \ref{T:genrankcalc}, if $t\in T^*_{L_{12}}(D)$,
  \[
    d-r_1-r_2 \le \rk(C) \le d - \max\{r_1,r_2\}.
  \]
  If $r_1 = r_2 =0$, then $\rk(t)=d$. That is $t \in \sec^{d-1}(C)$
as desired. 
We have $r_1 = r_2 = 0$ for $\deg(D) <
  c$. Further, if $r_1>0$ or $r_2>0$, then $d-r_1-r_2 \ge c$ so any $t\in
  \Hi{K_C \otimes L_1^{-1} \otimes L_2^{-1}}$ with $\rk(t)<c$ can be written
  as the sum of $\rk(t)$ matrices of rank 1. This is the statement
  $\Sec^{j-1}(C) = R^j$ for $j<c$.
\end{proof}

Finally we need to check that the set theoretic equality is a scheme theoretic
equality.

\begin{thm}\label{T:genscheme}
  $\Sec^{j-1}(C) = R^j$ as schemes for $j<c$.
\end{thm}

\begin{proof}
  We check the same  information as before.
  \begin{enumerate}

    \item $\Sec^{j-1}(C)$ is normal and smooth away from $\Sec^{j-2}(C)$ with
      tangent space generated by the span of the divisor $2D$ for $q$ a
      general point in the span of $D$. The \emph{exact} same proof holds.

 \item By exactly the same argument as in Lemma \ref{L:secinrank} 
we get an inclusion of schemes:
 $\sec^{p-1} \hookrightarrow R^P$ .

    \item Let $\~R^p \subset R^p \times G(p,n) = \left\{ (\varphi,\lambda) |
      \im(\varphi) \subset \lambda \right\}$, then $\~R^p$ is smooth. 
      We make the minor
      modification that now $R^p \subset \PP( V_1 \otimes V_2)$, $V_i = \Hd{C, L_i}$
      and the linear map is
     \[
       v : V_1 \otimes V_2 \otimes \bigwedge^p(V_2)
       \longrightarrow
       V_1 \otimes \bigwedge^{p+1}(V_2)
     \]
     and the fibers are linear spaces. Since $p< \cliff(C,L_1,L_2)$,
we can identify $\~R^p$ with the incidence correspondence defining the 
full secant variety,$B^{p-1}(L_{12})$.  Again  as in the case
 $L_1 = L_2$ it follows that the projections are the same and that
$\sec^{p-1}(C)=R^p$.
    
\end{enumerate}
  
  \end{proof} 

Since $\sec^{p-1} = R^P$ as schemes they have the same tangent 
spaces at all points.  We present an independent calculation of
the tangent space to $R^p$ at a smooth point, which is to say,
at a point of rank exactly $p$

     \begin{thm}
       Let $p<\cliff(C,L)$
$\varphi \in R^p\setminus R^{p-1}$, so $\varphi\in T_{L_{12}}^*(D)$ 
for some
       $D$ of degree $p$. The tangent space $T_{\varphi,R_p}$ is the
       projectivization of $$\partial\left( \H{K_C \otimes L_{12}^{-1} 
       (2D)|_{2D}} \right) \subset \Hi{K_C \otimes L_{12}^{-1}} $$

     \end{thm}

     \begin{proof}
       Let $\~T \subset \Hi{K_C \otimes L_1^{-1} \otimes L_2^{-1}}$ be the
       cone over $T_{\varphi,R^p}$. Then $\~T $ is the tangent space to 
the affine rank $p$ locus.  For $\varphi \in R^p\setminus R^{p-1}$ 
this tangent space is described in ((\cite{ACGH}) see page  68) as the  
matrices which map the 
kernel of $\varphi$ into the image of $\varphi$.  That is to say:
 $\~T = \left\{ \psi \in \Hi{K_C
       \otimes L_1^{-1} \otimes L_2^{-1}} | \psi: \ker \varphi \to \im \varphi
       \right \}$.  But, $\ker \varphi = \H{C, L_1(-D)}$ and 
$\im \varphi = \left( \{       x \in \Hi{K_C \otimes L_1^{-1}} | x|_{\H{L(-D)}} = 0 \right\}$ so
       $$\~T = \ker \left( \Hi{K_C \otimes L_{12}^{-1} } \to
       \Hom \left( \H{L_1(-D)}, \Hd{L_2(-D)} \right) \right) $$
       Clearly 
       
         $$T_L(2D) \subset 
	 \ker \left( \Hi{K_C \otimes L_{12}^{-1} } \to
	 \Hi{K_C \otimes L_{12}^{-1}  (D)} \right) \subset \~ T$$
       
       so we must show
       $$
       \Hi{C, K_C \otimes L_{12}^{-1}  (2D)}
       \to
       \Hom \left( \H{L_1(-D)}, \Hd{L_2(-D)} \right)
      $$
       is injective or by Serre duality (and using $\Hom(A\vee,B) = A\otimes
       B$) that 
      $$
       \H{C, L_1(-D)} \otimes \H{nC, L_2(-D)} \to
       \H{C, L_{12} (-2D)}
      $$
       is surjective, which is true since $\deg(L_i(-D)) \ge 2g-1$.
     \end{proof}
\begin{cor}
$R^p$ is smooth at $\varphi \in R^p\setminus R^{p-1}$
\end{cor}
\begin{proof}
Since $p< \cliff(C,L)$ by the above argument $ T_{\varphi,R^p} =T_{L_{12}}(D)$
which is of dimension $2p+1$.
\end{proof}

%Recall that $\cliff\left( C,L \right) = d-1$ unless either $L = K_X(D)$ of $C$
%is hyperelliptic.

This concludes the proof of  Theorem \ref{T:mainresult},
since for any factorization $L= L_1 \otimes L_2$ with $
\deg(L_i) \geq 2g+1+k$ one has $\cliff(C,L_i) \geq k+1$. 
For the case $L_1=L_2$ the result is sharp.
By Lemma \ref{L:specialvar}   we can for $L_1 = L_2 =L=K_C(D)$ find an element  $\tau \in T_{L_{12}}^*(D)$ with 
$d-2 = \rk(\tau) < \deg(\tau)= d-1$.
In the language of Theorem \ref{T:mainresult}
$d=k+3 $ and    $\cliff(C,L)= k+1$. In
particular, for $k=0$ this gives a weaker result than  \ref{T:eks},
the result proved in (\cite{eks}). However, the theorem can be 'tweaked' to get,

\begin{lemma}
 If $L_1\neq L_2$,but $\deg(L_i)\geq 2g+k+1$ , then
$\sec^j(C) = R^j$ for $j \leq (k+1)$
then 
\end{lemma}
\begin{proof}
By Proposition \ref{P:cliffbound} (iii) we have 
$\cliff(C,L_1,L_2) \geq k+2$ .The lemma follows from 
Theorem \ref{T:genscheme}
\end{proof}

\section[Line Bundles with $h^1(L)=1$]{The Clifford Index for line bundles with  $ h^1(L)=1$}\label{S:h1}

Throughout this section we will write $L = K_C(-P)$, where $P$ is an effective
divisor of degree $p < c = \cliff(C)$. Since $\deg(P) < c$, $\h{\OO_C(P)} =
1$ (else $\cliff(C) \le p-2$), so $\hi{L} = \h{K_C \otimes L^{-1}} =
\h{\OO_C(P)} = 1$. This is the only case we will discuss as it is the only
situation in which I can say something meaningful about $\cliff(C,L)$. 

\begin{lemma}

 Suppose that  $L= K_C(-P)$ with $\deg(P) = p < c$ then

\begin{enumerate}

\item $\cliff(L,D) = \cliff(D+P) -p$.  
\item $ c-p \leq \cliff(C,L) \leq c$. 
\item  $ \cliff(C,L) = c-p$, except possibly in the case where $\cliff(C)= [\frac{g-1}{2}]$
and $p= \cliff(C)-1$

\end{enumerate}

\end{lemma}

\begin{proof}
\begin{enumerate}
\item  First notice that $r_L(D) = r_{K_C}(D+P) $ since
\begin{align*} 
    r_L(D) 
    & = \h{K_C \otimes L^{-1}(D)} - \h{K_C \otimes L^{-1}} \\ 
    & = \h{\OO_C(D+P)} - \h{\OO_C(P)} \\ 
    & = \h{\OO_C(D+P)} - 1 \\ & = r_{K_C}(D+P).
  \end{align*} 

Setting $d= \deg(D)$ we have:
\begin{align*} 
    \cliff(L,D) 
    & = d-2r_L(D) \\ 
    & =(d+p - 2r_L(D))-p \\ 
    & = \deg(D+P)-2r_{K_C}(D+P) - p \\ & = r_{K_C}(D+P)-p.
  \end{align*}

  \item If $D$ computes $\cliff(C,L)$, then   
  it spans at most a
  codimension 2 plane in $\PP(\Hd{C,L})$, so $\h{I_{\overline{D}}} = \h{L(-D)}
  \ge 2$. But $L(-D) = K_C(-P-D)$, so $D+P$ is eligible to compute
  $\cliff(C)$, and hence by 1. $\cliff(L,D)=\cliff(D+P)-p \geq c-p$.
Considering the divisor $D+P$ where $D$ computes $\cliff(C)$ gives the 
other inequality.
\item Suppose $D$ computes $\cliff(C)$ and $\deg(D) < [\frac{g-1}{2}]$.
Consider the divisor $L(-D)$.  It is effective except possibly in the case
when $\cliff(C)= [\frac{g-1}{2}]$ and $p= \cliff(C)-1$.  By 1, $\cliff(C,L(-D))=
c-p$.
 
\end{enumerate}

\end{proof}

{\bf Remark:} In general I would not expect it to be the case these divisors
give rise to Shiffer variations of low rank. This is in complete analogy to
the fact that a curve being non-arithmetically normal gives rise to a divisor
of degree $2n+2$, spanning $n$ planes, but the existence of an $2n+2$ pointed
 $n$ plane
do not necessarily imply that the embedding is not arithmetically normal.

{\bf Remark:} We postpone  the discussion of the possible pathology  that can
occur until after our one positive result. 

Notice that when $L=K_C(-P)$, $L$ is  always very ample. If $D$ was a divisor of degree 2 
such that $\h{L(-D)} \ge \h{L} - 1$, then $\hi{L(-D)} \ge 2$, and hence $\cliff(P+D) \le
p+2 - 2 = p < \cliff(C)$. Since \[\deg(P+D) = \deg(P) + 2 < c+2 <
\frac{g-1}{2} + 2 < g,\] $P+D$ is eligible to compute $\cliff(C)$, and we
would have a contradiction since $p < c$.

To apply our standard setup we need to know that $L$ is quadratically normal.
This is a special case of a theorem of Green and Lazarsfeld proven in \cite{GL} 

\begin{thm}[Green-Lazarsfeld]

  Suppose $L$ is very ample with $\deg(L) \ge 2g + 1 - 2\hi{L} -
  \cliff(C)$, then $L$ is projectively normal. 

\end{thm}

In our case $\hi{L} = 1$ and $\deg(L) = 2g - 2 - p > 2g - 2 - \cliff(C)$. We
will actually give a proof of the theorem, as the use of Shiffer deformation
and Clifford index gives (to us!) a conceptual proof of the theorem. As [G-L]
points out, cubic and higher normality follow from the base-point-free pencil
trick and the only issue is to prove quadratic normality.

A simple argument will show that the failure of quadratic normality implies
the existence of a divisor of Clifford index zero. From our point of view it
is natural to restate the inequality of the theorem as \[\deg(L) \ge (2g -
2\hi{L}) - (\cliff(C) - 1 ).\] This makes clear the basic idea, the only way
to get a divisor of Clifford index zero is as the ``projection'' from a plane
of dimension $(c-1)$ of a divisor of Clifford index $c$. This is the geometry behind the proof.
Incidentally one can check that the inequality of the theorem implies
$\hi{L} \le 1$.

\begin{proof}

  Firstly, if $L$ is not quadratically normal, then the map \[\H{L} \otimes
  \H{L} \to \H{L^{\otimes 2}}\] is not surjective; or dually the map \[\Hi{K_C
  \otimes L^{-2}} \to  \Hom(\H{L},\Hi{K_C \otimes
  L^{-1}})\] is not injective. If $\xi \in \Hi{K_C \otimes L^{-2}}$ is in the
  kernel of this map, then, viewing $\xi$ as a matrix in $\Hom(\H{L},\Hi{K_C \otimes L^{-1}})$
   we have  $\rk(\xi) = 0$. Viewing $\xi$ as a Shiffer variation, there is a divisor 
  $D$ on $C$ such that $\xi \in T_L(D)$ and setting $d
  = \deg(D)$, $r = r_L(D)$, we have $d-2r \le 0$. Adding points to $D$ if
  necessary we have a divisor $D$ such that $\cliff(L,D) = 0$.

We can bound the degree of $D$ such that $\xi \in T_L(D)$. Namely, if $L=K_C(-P)$  we have
$\deg(D) \le \frac{3g-5}{2} - p$. To prove this, we use the  same idea
as for $L = K_C$. Namely the embedding $C \hookrightarrow \PP(\Hd{C,L^{\otimes
2}})$  in a projective space of dimension $3g - 4 - 2p$. This is because
$p < \cliff(C) \le \frac{g-1}{2}$ implies $\deg(L^{\otimes 2}) = 4g - 4 - 2p
\ge 3g -3$, so $L^{\otimes 2}$ isn't special. Now $\dim(\Sec^j(C)) = 2j + 1$
implies that as long as $j \ge \frac{3g-5}{2} - p$, $\Sec^j(C) =
\PP(\Hd{C,L^{\otimes 2}})$, so any $\xi \in \Hi{K_C \otimes L^{-2}}$ is the
sum of at most $\frac{3g-5}{2} - p$ Shiffer variations.

We consider separately the cases of $\hi{L} = 0$ and $\hi{L} = 1$. If $\hi{L}
= 1$, $L = K_C(-P)$, where $P$ is a divisor of degree $p$ and our inequality
is that $p < c = \cliff(C)$. $\cliff(L,D) = 0$ means $d - 2r_L(D) = 0$ and
$r_L(D) = \h{\OO_C(D+P)} - 1 = \frac{d}{2}$, and hence $\cliff(K_C,\OO_C(D+P))
= d + p - 2\left( \frac{d}{2} \right) = p < c$. Thus $D+P$ cannot be used to
compute $\cliff(C)$ and hence must span a hyperplane. That is,
$\h{K_C(-D-P)} = 1$ and so $\cliff(K_C(-D-P) = \deg(K_C(-D-P)$
. But then \[\cliff(\OO_C(D+P)) = \cliff(K_C(-D-P)) =
\deg(K_C(-D-P)) = 2g - 2 - (d + p).\] But recall, $d \le \frac{3g-3}{2} - p$,
so \[\deg(K_C(-D-P)) \ge 2g - 2 - \frac{3g-5}{2} \ge \frac{g+1}{2} >
\cliff(C),\] and hence $\cliff(L,D) = \cliff(\OO_C(D+P) - p >\cliff(C)-p> 0$. 
This is a
contradiction.

Now assume $\hi{L} = 0$ and $\cliff(L,D) = 0$, $\hi{L(-D)} = r_L(D) > 0$ so
$L(-D) = K_C(-E)$. Let $\ell = \deg(L)$, $d = \deg(D)$, $e = \deg(E)$, and $r
= r_{K_C}(D)$, so $r + 1 = r_L(D)$. From $D = L \otimes K_C^{-1}(E)$ we get $d
= \ell + e - (2g - 2)$. Now $d - 2(r+1) = 0$ can be rewritten as $l + e -
(2g-2) - 2(r+1) = 0$, or $\ell + (e - 2r) = 2g$, or $\ell = 2g - \cliff(E)$.
If $E$ satisfies $r>0$ and $\hi{\OO_C(E)} \ge 2$, then $\cliff(E) \ge
\cliff(C)$, and $\ell \le 2g - \cliff(C)$ as desired. If $r=0$, then $d=2$ and
$L$ isn't very ample! The only case left to rule out is $\hi{\OO_C(E)} \le 1$.
That is $E = K_C(-P)$ with $P$ general.  Since $P$ is general, $ \cliff(E)= \cliff(P)=\deg(P)$.         
 We can bound $\deg(D)$ from above and
hence $\deg(P)$ from below. If $\deg(L) = \ell$, $\h{L^{\otimes 2}} = 2\ell -
g + 1$, and by the usual argument we can assume $\deg(D) \le \frac{2\ell - g
-1}{2}$. Since $K_C(-E) = L(-D)$, \[\deg(P) = \deg(K_C(-E)) = \deg(L(-D)) \ge \ell -
\frac{2\ell - g - 1}{2} \ge \frac{g+1}{2}.\] So $\cliff(E) \geq \frac{g+1}{2} \geq  \cliff(C)$
 and so \[ l = 2g - \cliff(E) \geq 2g - \cliff(C)\] \,
as desired. 

\end{proof}

\begin{cor}

  Suppose $L$ is very ample, $\deg(L) = 2g$ and $L$ is not arithmetically
  normal. Then $C$ is hyperelliptic.

\end{cor}

\begin{proof}

  From the proof of the theorem we see that $\deg(L) = 2g$ implies $e - 2r =
  0$. By Clifford's theorem, after  easily ruling out the cases of $E = \OO_C$
  or $K_C$, we see that $C$ is hyperelliptic and $E$ is a multiple of the
  $g^1_2$ on $C$.

\end{proof}

When $L = K_C$ we have seen that $\cliff(C)$ characterizes on the nose the
degree to which secant varieties are rank loci. For $L = K_C(-P)$ this is no
longer true. We can always guarantee the same bound, but unlike for $L$ with
$\hi{L} = 0$ or $L = K_C$, the existence of a divisor with $\cliff(L,D) = c$
does not seem to imply there exists a $\xi \in T_L(D)$ with $\rk(\xi) = \cliff(L,D)$.
Nonetheless, the more important lower bound always holds. Consider the standard
diagram:
\begin{displaymath}
  \xymatrix{
  & \PP \left( \Hi{C, K_C \otimes L^{-1}} \right) 
  \ar[dr]^{v} & \\ 
  C 
  \ar[ur]^{\Phi_L} 
  \ar[dr]^{\Phi_{2L}} &  
  & \PP \left( \sym^2 \left( \Hi{C, K_C \otimes L^{-1}} \right) \right) \\
  & \PP \left( \Hi{K_C \otimes L^{-2}} \right) 
  \ar@{^{(}->}[ur]^{i} & }
\end{displaymath}
As always, $v$ is the Veronese map, $i$ is an inclusion since $L$ is
quadratically normal, and $\Phi_L$, $\Phi_{2L}$ are the maps associated to the
linear systems $L$ and $L^{\otimes 2}$. We set $\Sec^j(C) =
\Sec^j(\Phi_{2L}(C)) \subset \PP(\Hi{K_C \otimes L^{-2}})$, $R^j(C) = R^j =
\Sec^j(v(C)) \cap \PP(\Hi{K_C \otimes L^{-2}})$, so $R^j$ is the rank $j$
locus.

\begin{thm}
  Suppose $j<\cliff(C,L)$; then $R^j(C) = \Sec^{j-1}(C)$ as sets.
\end{thm}

\begin{proof}
  We use the same argument as before in the case $L=K_C$. As in the proof of
  the quadratic normality, any $\xi\in\Hi{K_C\otimes L^{-2}}$ can be written
  as $\xi\in T_L(D)$ where $\deg(D) \le (3 g-5)/2 -p$. If $D$ is eligible to
  compute $\cliff(C,L)$ then $\cliff(L,D) \ge \cliff(C,L)$. If $D$ is not
  eligible to compute $\cliff(C,L)$ then either $r_L(D) = 0$ or $\hi{L(-D)}
  \le 1$. In the first case, $\xi\in \sec^{d-1}(C) \setminus \Sec^{d-2}(C)$.
  In the later case we may assume $r_L(D)> 0$. Again exactly as in the proof
  of quadratic normality theorem and recalling $K_C\otimes L^{-1} = \OO_C(P)$
  we see that
  \[
    \cliff(L,D) = d - 2 \left( \h{\OO_C(D+P)-1} \right) = \cliff(K_C, D+P) - p.
  \]
  But
  \begin{multline*}
    \cliff(K_C, D+P) 
     = \cliff(K_C, K_C(-D-P)) \\
     \ge 2g-2 -\deg(D)-\deg(P) 
     \ge 2g-2-\frac{3g-5}{2}\ge\frac{g+1}{2}
  \end{multline*}
  Thus
  \[
    \cliff(L,D) \ge\frac{g+1}{2}-p > \cliff(L,C)
  \]
\end{proof}

\textbf{Remarks:}

\begin{enumerate}
  \item Scheme theoretic equality should follow exactly as in the cases of
    $L=K_C$ or $\hi{L} = 0$. I have not checked the details.

\end{enumerate}

2. In case $L=K_C$ the existence of $\xi$ with $\rk(\xi) = \cliff(C)$ was 
delicate. If $D$ computed $\cliff(C)$ one could find an $\eta\in
\H{K_C(-D}$   such that there was a $\tau\in T_{K_C}(D)$ with the property that $\tau$
vanished on $1\otimes\H{K_C(-D)}$ as well as 
vanishing on $f_i\otimes\eta_1$ for
$1< i\le n$ and so $\rk(\tau) = \cliff(D) = \cliff(C)$.

If $D$ computes $\cliff(C)$ then $K_C(-D)$ is base point free and 
that is the key point. 

In the case $L=K_C(-P)$ it is the twisted linear system, $K_C\otimes
L^{-1}(D)$, not $\OO_C(D)$ which comes into play. One needs the base point
freeness of $L^{\otimes 2}\otimes K_C^{-1}(-D) = L(-D-P)$.

In the case of the divisor used to compute $\cliff(C,L)$ being
 $D = L(-E)$ where $E$ computes $\cliff(C)$ and
$\deg(E)$ small, $L^{\otimes 2}\otimes K_C^{-1}(-D) = E - P$ which satisfies
$\h{\OO_C(E-P)} < 1$. I cannot provide a proof, but I think that these
divisors do not give rise to Shiffer variations 
$\tau$, satisfying $\rk(\tau)=
\cliff(L,D)$. I suspect that on a general line bundle $L = K_C(-P)$ that
$\sec^{j-1}(C) = R^j$ for $j \leq \cliff(C) -2$ as opposed to holding for
$ j \leq \cliff(C)-p-1$.
 The moral is that projecting from a general point
should not affect the Clifford index, but projection from special points
should. Here is a small positive result.

\begin{thm}
  Suppose \, $\deg(E)<(g-1)$,\,  $\h{\OO_C(E)}\ge2$. Suppose , $L(-E)$ is
 base point free, $\h{\OO_C(E-P)} = 1$  and  $\deg(P)<\cliff(C)$. Let $D=E-P$, $L=K_C(-P)$. Then
  \begin{enumerate}
    \item $\cliff(L,D) = \cliff(E) -p$.
    \item There exists a Shiffer variation $\xi$, of rank $ c = \cliff(L,D)$,
but such that $ \xi \notin \sec^{c-1}(C)$.   
  \end{enumerate}
\end{thm}

\begin{proof}
  \begin{enumerate}
    \item Since $K_C\otimes L^{-1}(D) = E$, $\cliff(L,D) = \cliff(E) -p$. By
      degree considerations $D$ cannot span a hyperplane and $r_L(D) =
      r_{K_C}(E)$
    \item $L(-E)$ is base point free so the same argument as for $L=K_C$
      works. We can find $\eta\in \H{L(-E)}$ which vanishes on exactly $E$.
      If $\{ 1,f_1, \dots, f_n \}$ is a basis for $\H{\OO_C(E)}$ then since
      $E-D=P$ and any Shiffer variations in $T_L(D)$ kills $\H{L(-D)} \supset
      \H{L(-E)}$ and $\eta \otimes 
 f_i$ are not in $\H{L(-D)}$ and killed by some
      $\tau\in T_L(D)$, specifically by $\tau = \sum_{p_i \in D}\tau_{ p_i}$.
  \end{enumerate}
\end{proof}

\textbf{Remark:} I believe that one should have $L(-E)$ base point free
always when $\deg(E)<g-1$ and $E$ computes $\cliff(C)$. I do not have a proof.

In general the technique will fail since $L(-E)$ need not be base point free.
To produce this example, we also give an example of a divisor not in the Petri
 locus. That is a curve $C$ and a divisor $D$ such that $(r+1)(g-d+r) >
g$ but the Petri map is not surjective. After conversations with L. Ein and I.
Coskun it is clearly not hard to find such examples. I do not know of examples
in the literature though. Firstly we prove:

\begin{lemma}
  Suppose $D$ is base point free. The following two assertions are equivalent:
  \begin{enumerate}
    \item The Petri map $\H{\OO_C(D)}\otimes \H{K_C(-D)} \to \H{K_C}$ is not
      surjective.
    \item The natural map $\H{\OO_C(D)} \otimes \H{\OO_C(D)} \to
      \H{\OO_C(2D)}$ is not surjective.
  \end{enumerate}
\end{lemma}

\begin{proof}
  Now by the base point free pencil trick c.f. [ACGH, p126] we have an exact
  sequence:
  \[
  0 \longrightarrow \OO_C(-D) \longrightarrow \OO_C^{\oplus 2}
  \longrightarrow \OO_C(D) \longrightarrow 0
  \]
   Twisting by $K_C(-D)$ we get:
  \[
  0 \longrightarrow K_C(-2D)
  \stackrel{\iota}{\longrightarrow} K_C(-D)^{\oplus 2}
  \stackrel{m}{\longrightarrow} K_C
  \longrightarrow 0
  \]
  and $m$ is the Petri map. So $\h{m}$ is not surjective if and only if
  $\hi{\iota}$ is not injective if and only if (by Serre duality)
  $\H{\OO_C(D)}^{\oplus 2} \longrightarrow \H{\OO_C(2D)}$ is not surjective.
\end{proof}

I speculate that the curves carrying a $g^1_n$ such that the Petri map is not
surjective will be represented by a cohomology class in $M_{g,n}$ that is not
an intersection of divisors or in the cohomology ring of ordinary
Brill-Noether loci.

To construct such curves recall that if $f:C\to \PP^1$ is a $n$-gonal map then
$f^*(\OO_{\PP^1}(1)) = \OO_C(D)$ where $D$ is the $g^1_n$. If $\OO_C(2D)$ is
to have extra sections we need $\h{\OO_C(2D)} \ne \h{\OO_{\PP^1}(2)}$. But by
the projection formula $\H{C, \OO_C(2D)} = \H{C, g^*(\OO_{\PP^1}(2))}$\,$ =
\H{\PP^1, \OO_{\PP^1}(2)\otimes g_*(\OO_C)}$ \, and we can write \, $g_*(\OO_C)$\\$ =
\bigoplus_{i=0}^{n-1} \OO_{\PP^1}(-a_i)$, with $a_0=0$ and $a_i >0$ for
$i\ge 1$. For $\h{\OO_C(2D)} > 3$ we need some $a_i\le 2$. In fact $a_i>1$
because otherwise using the projection formula we see $\h{\OO_C(D)} \geq 3 $. Hence
$\h{\OO_C(2D)} = 3 + \#\left\{ i | a_i = 2 \right\}$. The simplest example of
this phenomena is to take a cyclic cover of $\PP^1$ of degree 4 branched at 8
points. By construction $g_*(\OO_C) \cong \bigoplus_{i=0}^{3}
\OO_{\PP^1}(-2i)$ (as an $\OO_{\PP^1}$ module). A simple computation with
Hurwitz's formula shows $g(C)=9$ and $\OO_C(F) = g^*(\OO_{\PP^1}(1))$
is by construction  a $g_4^1$ on $C$. That is, 
$h{\OO_C(2F)} = 4$. This already
provides an example of a curve $C$ and line bundle $F$ such that the Petri map
$\H{\OO_C(F)} \otimes \H{\OO_C(K_C - F)} \to \H{K_C}$ is not surjective.
Further $\cliff(F)=2$ and we can check that $C$ is not hyperelliptic so that
$\cliff(C) =2$. Let $F=P_1+P_2+P_3+P_4$. Then since $F$ (by construction as a
cyclic cover) moves in a base point free pencil,  $\h{\OO_C(F-P_1)}=1$ and
hence any 3 points of a fiber span a plane, hence any two points span a line
and so $\h{K_C(-P_1-P_2)} = g-2 = 7$. If $P_1$ and $P_2$ are two points in
separate fibers they clearly span a line in $\PP(\Hd{C,K_C})$ as they map to
separate points of $\PP^1$. This means that $C$ is not hyperelliptic.
For this example it is not clear how to construct a
$D$ such that $L^2\otimes K_C^{-1}(-D)$ has base points. Here is a procedure to
do this in general.

Suppose $C$ is any curve with a $g^1_n$ ($n \geq 4$), call it $F$, and such that $\h{\OO_C(2F)}
= 5$. Let $\sum_{i=1}^n P_i \in |F|$, and set $P = \sum_{i=1}^{n-3} P_i$, $D =
\sum_{i = n-2}^n P_i$, and $L = K_C(-P)$. Then in general $L$ will be very
ample and if $L$ is very ample, $\h{L} = g - n + 3$, $\h{L(-D)} = g - n + 1$,
so $D$ spans a 3-secant line. \[L^{\otimes 2} \otimes K_C^{-1}(-D) = L(-P-D) =
K_C(-F-P).\] By construction $\h{K_C(-F-P)} = g - 2n + 4 = \h{K_C(-2E)}$, so
$L(-F)$ has base points on $D$, and hence the  map 
\[\H{L(-E)} \otimes
\H{\OO_C(E)} \to \H{L}\] 
lands in $\H{L(-D)}$. For such a curve and linear
system, the method of construction of  Shiffer variations used in the case of $L = K_C$  won't work!

We now construct such $C$ and $F$. We merely iterate the previous
construction. Namely on $C$ we have the linear system $g^*(\OO(1))$, which is
of degree 32. We take the 4-fold cyclic cover of $C$ branched on $g^*(\OO(1))$
, call this $C_2$. By Hurwitz formula we calculate $g(C_2) =
\frac{4(16) + 3(32)}{2} + 1 = 81$. Let $f$ be the composed map,
\begin{displaymath}
  \xymatrix{
  & C_1 \ar[dr]^{g} & \\ 
  C_2 \ar[ur]^{h} \ar[rr]_{f} &  
  & \PP^1 }
\end{displaymath}
then $h_*(\OO_{C_2}) = \bigoplus_{i=0}^3 g^*(\OO_{\PP^1}(-2i))$, and so
\begin{multline*} f_*(\OO_{C_2}) = g_*h_*(\OO_{C_2}) = g_*\left( \bigoplus_{i=0}^3
  g^*(\OO_{\PP^1}(-2i)) \right) \\ = g_*(\OO_{C_1}) \otimes
  \bigoplus_{i=0}^3 \OO_{\PP^1}(-2i) = \OO_{\PP_1} \oplus
  \OO_{\PP^1}(-2)^{\oplus 2} \oplus \OO_{\PP^1}(-4)^{\oplus 3} \oplus
  \ldots
\end{multline*}
so setting $F = f^*(\OO(1))$ we see that $\h{\OO_C(F)} = 2$, $\h{\OO_C(2F)} =
5$. Since $F$ is a $g^1_{16}$ we get $\H{K_C(-F)} = 81 - 16 + 1 = 66$, and
$\h{K_C(-2F)} \linebreak = 81 - 32 + 4 = 53$. In particular, $F$ imposes 13 conditions on
the linear system $K_C(-F)$. Take $F = \sum_{i=1}^{16} P_i$ and $P =
\sum_{i=1}^{13} P_i$, such that $P$ imposes independent conditions on
$K_C(-F)$. Let $L = K_C(-P)$. $L$ is very ample because if not there would be
a divisor $A$ of degree 2, such that $\h{\OO_C(F-P+A)} = 2$. This would mean
these 15 points are linearly dependent. Exactly as in the case of showing
$K_C$ is very ample, the 15 points cannot lie on one fiber, and if they lie on
different fibers, they clearly are linearly independent. This completes the
proof. 

While pathological behavior occurs, the generic case is fine. For example, we
have seen that if $r_L(D) = 1$, for simple linear algebra reasons we can
construct a $\tau \in T_L(D)$ with $\rk(\tau) = \cliff(L,D) = d - 2$. So if D
computes $\cliff(C,L)$ we can construct a Shiffer variation of rank equal to 
$\cliff(C,L)$. 

\section[Connections with Koszul Cohomology]{Connections with Koszul Cohomology and  Green's Conjecture}\label{S:gc}

The standard reference for  Koszul cohomology is \cite{G1}.  However I have also profited from reading 
\cite{G2}, \cite{V1},
and \cite{E1}.  We will not define these groups in their greatest generality. We begin by reviewing with
brief proofs some of the basic facts about Koszul cohomology.
We will assume throughout that k
is an algebraically closed field of characteristic $ \ne 2,3$.

Let W be a vector space over k  and let $S= \bigoplus_{k=0}^{\infty} \sym^k(W)$ be the symmetric algebra.
Thus k is the quotient of S by the irrelevant ideal.  Let $M= \bigoplus_{q=o}^{\infty} M^q$ be a graded S module.  Define 
\begin{equation}
 \mathcal{K}_{p,q} = \bigwedge^p(W)\otimes M^q 
\end{equation}

and let 
\begin{equation}
 \delta : \mathcal{K}_{p,q} \to \mathcal{K}_{p-1,q+1}
\end{equation}
be defined by 
\begin{equation}\label{E:delta}
\delta(w_1\wedge\dots \wedge w_p \otimes m) = \sum_{i=1}^p (-1)^p (w_1 \wedge \dots
\wedge w_{i-1} \wedge w_{i+1} \dots \wedge w_p)\otimes (w_im)
\end{equation}
One checks that $\delta^2 = 0$ by  a direct computation.  Therefore we get a complex and we can compute
cohomology.

\begin{defn}

The Koszul cohomology  group  $K_{p,q}(S,M) $ is the cohomology of the complex:
\[ \mathcal{K}_{p+1,q-1} \overset{\delta}{\longrightarrow}\mathcal{K}_{p,q} \overset{\delta}{\longrightarrow} 
\mathcal{K}_{p-1,q+1} \]

\end{defn}

The most common case is when $W= \H{C,L}$ and $ M^q = \H{C,L^{\otimes q}}$.  In that case we will denote
the Koszul cohomology group as $K_{p,q}(C,L)$.   Koszul cohomology 
groups are useful because they 
 can be used to compute free resolutions of graded modules over S.
For completeness we include a brief description of this phenomena.  The information can be found
in any of the sources mentioned above, as well as Eisenbuds book,'Commutative Algebra'  \cite{E2}. 

\begin{defn} A free resolution $ F^{\mbox \textbullet}$  of M is an exact sequence of the form: 

\[ 0 \to F_n \to F_{n-1} \to \dots \to F_0 \to M \to 0 \]

with the $F_i$ free S modules.

\end {defn}

The maps $F_{i+1} \to F_i $ are given by a matrix of homogenous polynomials, call it $f_i$.  The most important case is 
when the maps are given by non constant polynomials.

\begin{defn}
 
A free resolution is said to be minimal if all the matrices $f_i$ contain no non-zero constant terms.

\end{defn}

The two basic facts are that free resolutions always exist and are (essentially unique).  Many
important properties of a module M, and geometry of the curve embedded in $\PP(W^*)$ can be read
off from the minimal free resolution.  For example, $M_0 = S$ is equivalent to the curve being
linearly normal.  If we write $M_1 = \sum_{i=1}^{i=n} S(-a_i) $ then the $a_i$'s represent the degrees
of the minimal generators of the ideal of $C$ in the projective embedding.  The connection between
the Koszul complexes and the minimal free resolution  comes as follows.

 Let $M= \sum_{i=0}^{i=\infty}M_i$ be a graded S module with a free resolution $F^{\mbox \textbullet}$ 
with $F_p = \sum_q(V_{p,q}\otimes S(-q)) $: ie $V_{p,q}$ is a vector space that keeps track of how many 
$S(-q)$'s appear in $F_p$.  

\begin{thm} $dim(K_{p,q}(S,M)) =   dim(V_{p,p+q})$  

\end{thm}

\begin{proof}

The proof of this is  well known and in all the above sources.  I will sketch it for the 
sake of completeness.  Firstly one has the Koszul resolution of k:
$$ 0\to \bigwedge^n(V) \otimes S \to \dots \to V\otimes S \to S \to k \to 0 $$
One can tensor this resolution with M and one sees by inspection that the maps are the boundary 
maps $\delta$ from equation \ref{E:delta}.  Hence one has $K_{p,q}(S,M)$ = $Tor_p^{p+q}(M,k)$.  On the other hand, we can 
take the free resolution $F^{\mbox \textbullet}$
   of $M$ and tensor it with k (viewed as the quotient of S by its 
maximal ideal).  By the commutivity of Tor, this gives the same answer as above.  On the other hand,
since $F^{\mbox \textbullet}$ 
 is minimal, when we tensor with k all the boundary maps are zero and hence we get
$V_{p,p+q}\otimes S$  is the degree $p+q$  piece of $F_p$.
\end{proof}

To relate our work to Koszul cohomology we introduce some new notation. This
is needed as we are going to work with the 
dual of the Koszul complex. This is the complex which, term by term, is the 
vector space dual of the Koszul complex.  There is another notion of a 
dual Koszul complex which involves divided powers (see the appendix to 
\cite{E2} for details).
   As long as char(k) is sufficiently large, these are the same.  
The  reason we do this is that our Shiffer variations live naturally  in the dual of the 
Koszul complex we describe .  

Let $V= W^*$, $ W = \H{C,L}$ where L is a very ample line bundle. We denote by $M_i$ the dual of 
$\H{C,L^{\otimes i}}$ so that $ M_1 = V. $ 

We also assume that the line bundle  L is quadratically normal which means that
$M_2 \hookrightarrow \sym^2(V)$ is injective.  We frequently consider $\varphi \in M_2$
 as an element of $ \sym^2(V)$ or equivalently as an element of $\Hom^{\mbox sym}(W,V)$.

The Koszul cohomology group can be calculated as  the cohomology of:
\[\bigwedge^{p+1}(W)\otimes \H{C,L} \to \bigwedge^p(W)\otimes \H{C,L^2} \to 
\bigwedge^{p-1}(W) \otimes \H{C,L^3} \]
and so the dual cohomology groups are calculated by the complex:
\[M_3\otimes \bigwedge^{p-1}(V) \stackrel{\delta_{p-1}}{\to} M_2\otimes \bigwedge^p(V) \stackrel{\delta_p=\delta}{\to} M_1\otimes \bigwedge^{p+1}(V) \]

We are generally only interested in determining if these groups are 
zero or non-zero, so
 calculating
the dual group is good enough. 
Any $ v=\varphi\otimes\lambda \in M_2\otimes\bigwedge^p(V) $ can be viewed as an element of
$\Hom(\bigwedge^p(W),M_2)$ and $\delta(v)$ as an element of $\Hom(\bigwedge^{p+1}(W),V)$. 
If $w \in \bigwedge^{p+1}(W) $ then $\delta(v)(w)$ is calculated by first contracting 
$\lambda \wedge w$ via the standard map $\bigwedge^p(V)\otimes \bigwedge^{p+1}(W) \to W$
(recall that W and V are dual) and then letting $\varphi$ act on $\lambda\wedge\nu$. 

We start out with the basic observation:

\begin{lemma}

Let $\varphi \in M_2$ and let $\lambda \in \bigwedge^p(V) $ be decomposable.  If $Im(\varphi) \subset \lambda$,
then $\delta(\varphi\otimes\lambda)=0$.

\end{lemma}

\begin{proof}

By $ \lambda $ decomposable we mean that $\lambda = v_1\wedge\dots \wedge v_p$ and hence defines a 
subspace $V_p \subset V $.  We mean that viewing $\varphi$ as an element of $\Hom(W,V)$, that 
$\varphi(W) \subset V_p$.  
 First extend $v_1 \dots , v_p$ to a basis
$\langle v_1 \dots v_n \rangle$,
 of $V$.  Let $ w_1, \dots,w_p,
\dots , w_n$ be a dual basis so that $v_i( w_j)= \delta_{ij}$.
First note that because  $\varphi$ is symmetric, $\ker(\varphi) \supset
\langle w_{p+1} \dots w_n  \rangle$.
Further, one has
$v_1 \wedge \dots v_p\wedge w_j = 0$(under the standard
contraction) if $j > p$, and so $\lambda \wedge w_{i_1} \dots
\wedge w_{i_n} = 0$ unless $\langle w_{i_1} \dots w_{i_n} \rangle 
= \langle w_1 \dots w_p,w_j \rangle$ with $j>p$.
  Then $\lambda 
\wedge \bigwedge^{p+1}(W ) 
\subset W_{>p}$, the subspace of $W$ generated by
$\langle w_{p+1} \dots w_n \rangle$ and hence is killed by $\varphi$.

\end{proof}

Since $C \subset \PP(M_3)$ is non degenerate, $C$ generates $M_3$.  That means we can
construct a basis of $M_3$ of the form $\sigma^3_{p_1} \dots \sigma^3_{p_m}$ where
the $p_i$ are points of $C$ and $m = \h{M^{\otimes 3}}$.
From this it follows that $\partial (M_3 \otimes \bigwedge^{p-1}(V))$ is generated by
elements of the form $\sigma_{p_i}^2 \otimes \lambda \wedge \sigma_{p_i}$ with $\lambda \in \bigwedge^{p-1}(V)$
Any element of $\partial (M_3 \otimes \bigwedge^{p-1}(V))$ is then a sum of elements of the form
$(\sum_{i=1}^{k \leq p} (a_i\sigma^2_{p_i}) \otimes \sigma_{p_1} \wedge \dots \wedge \sigma_{p_k} \wedge
 \lambda $ where $k \leq p$ and $ \lambda \in \bigwedge^{p-k}(V)$.
Thus any element of $\partial (M_3 \otimes \bigwedge^{p-1}(V))$ can be written as a sum,
$\sum \varphi_i \otimes \lambda_i$ where not only $\rk(\varphi) \leq p$ but
in fact $\varphi \in \sec^{p-1}(C)$.  We do not claim that this representation is unique, that is to say
the $\lambda_i$ are not necessarily a basis for $\bigwedge^P(V)$.  Nonetheless.
this means that Shiffer variations that live in $R^d $ but not $Sec^{d-1}(C)$ are candidates to give 
non-trivial Koszul cohomology classes.  To fix ideas we consider the case of $L=K_C$.

\begin{thm}\label{T:counterexample}

Let D compute Cliff(C) and let $\varphi \in T_{K_C}(D)$ be a Shiffer variation in $\Sec^{d-1}(C) / \Sec^{d-2}(C)$
with $\rk(\tau) = \cliff(C)$.   Let $Im(\varphi) = \langle \sigma_1,\dots,\sigma_c \rangle \subset V$ and set
$\sigma = \sigma_1 \wedge \dots \wedge \sigma_c $. Then $\varphi\otimes\sigma $ represents a non-trivial Koszul
cohomology class in $K_{p,2}(C,K_C)$.
\end{thm}

\begin{proof}

By construction $\delta(\varphi\otimes\sigma)=0$ since $\sigma = Im(\varphi)$.
The incoming boundary map is from $M_3\otimes\bigwedge^{p-1}(V)$.  Since $C$
is embedded in $\PP(M_3)$ by the complete linear system $|3K_C|$, the points
of $C$ span $\PP(M_3)$ and hence we can find a basis of $M_3$ 
consisting of 
$\sigma_{p_i}^3 ,\,$with $ p_i\in C$. That is to say $\sigma_{p_i}^3$ 
represents the point $p_i \in C$.
 Up to a scalar (we are not in characteristic 3!)\, $\varphi\otimes\sigma
= \delta( \sum_{i=1}^c (-1)^i\sigma_i^3 \otimes\sigma_1\wedge\dots\sigma_{i-1}\wedge\sigma_{i+1} \dots
\wedge\sigma_c)$.  This element cannot lie in $M_3\otimes\bigwedge^{p-1}(V)$ as  $\sigma_i^3 \notin M_3$
because by Hassett's criteria (see Theorem \ref{T:Hassett}), the only rank one elements
of $\PP(M_3)$ are in $C$ and
$C \cap\, \langle \sigma_1 \dots\sigma_c \rangle
 = \varnothing$. Adding a coboundary in $M_4\otimes\bigwedge^{p-2}(V)$ cannot produce
an element that lives  in $M_3$. 
\end{proof} 

We comment on how this compares to the construction of Green and Lazarsfeld in the appendix to \cite{G1}.  
  Their  idea is 
that if one can factor a line bundle L  as $L_1\otimes L_2$ ,
where $\h{L_1}= r_1+1$ and $\h{L_2}= r_2+1$ with $r_i\ge1$ then one can produce
a non-trivial class in $ K_{r_1+r_2-1,1}(C,L)$ .

 There is a duality of Koszul cohomology groups
and for $L=K_C$ the  group $K_{g-p-2,1}$ is dual to $ K_{p,2}$.  
If  $L_1= \OO_C(D)$,
then $L_2 = K_C(-D)$ and by Riemann-Roch $r_2 = g-d+r_1-1$ so that $r_1 + r_2 - 1 = g-d +2r_1 -2=
g - \cliff(D)-2$. If $D$ computes $\cliff(C)=c$ the cohomology class lives in $K_{g-c-2,1}(C,L)$
which is dual to $K_{c,2}(C,L)$.
  Their construction  amounts to the following: if $s_i \in
\H{C,L_i} $ corresponds to a divisor $D_i$ then  the linear space corresponding to $D_1 \cap D_2$
is used to construct the cohomology class in $ K_{r_1+r_2-1,1}(C,L)$.  Eisenbud (\cite{G2} ch.8)
has given a different version of their construction.  

The class I have constructed also lies
in a group dual to $ K_{c,2}$.
This is the situation of Theorem \ref{T:counterexample}.
Let $D$ be a divisor used to compute Cliff(C). 
If $\H{\OO_C(D)}= \{f_0,\dots, f_r\}$ and $\H{K_C(_D)} = \{\eta_1,\dots,\eta_{g-d+r}\}$ then we constucted
a Shiffer variation which vanished on $\{f_0\otimes\eta_i\}\, 1\leq i\leq (g-d+r)$ and$ \{f_i\otimes\eta_1\}$
$1\leq i \leq r$ which is exactly the intersection of the linear spaces spanned by $D$ and $K_C(-D)$. By
this I mean that the $ \{f_i\}$ generate the ideal of the linear space corresponding to $K_C(-D)$ and the 
$\{\eta_i \} $ generate the ideal of the linear space corresponding to $D$.

\textbf{Remark 1 } Suppose $C$ is a generic curve of genus $5$ so that $\cliff(C) = 2$ and 
 $C$ has a $g_4^1$ which computes 
which computes $\cliff(C)$, call it $D$.  By Riemann-Roch, $\h{\OO(D)} = \h{K_C(-D)} = 2$.  Let
$ \{1,f\}$  be a basis for $\H{\OO(D)}$ and let $\{\eta, \omega \}$ be a basis for $\H{K_C(-D)}$.
By a slight abuse of notation (supressing the tensor signs) we get four elements of $\H{K_C}$ which
we denote by $\{1,f,\eta, \omega  \}$ and further we have that $ \{1,f, \eta \}$ are linearly
independent in $\H{K_C}$ .  Let $p_1, \dots p_4 \in C$ be the elements of $D$ and denote 
by $\tau_i$ the corresponding Shiffer variations.  Then by Theorem \ref{T:2} some sum,
$\tau = \sum a_i\tau_i $ with $a_i \ne 0$ contains $\{ 1,f,\eta\}$ in its kernel, which is to say 
it is of rank one. Hassett's criteria (see \ref{T:Hassett}) says such a $\tau$  corresponds to a line $L
\subset \PP^4$ such that $C \cap L = \emptyset$ and $\H{I_C(2)} \to  \H{\OO_L(2)}$ is not surjective.
  Notice that  $L = I(1,f,\eta)$ and the extra quadric in $\H{I_C(2)}$ that vanishes on $L$ is 
$1\times \omega - f \times \eta$.  Eisenbud's version of the Green-Lazarsfeld construction produces a
class in $K_{1,1}(C,K_C) = \H{I_C(2)}$ and this class is by construction the quadric, 
$1\times \omega - f \times \eta$.  It appears that the two constructions are connected.

\textbf{Remark 2} Notice how  the formula: $\h{\OO_C(D)} + \h{K_C(-D)} = g+1 - \cliff(D)$ comes into play in these constructions.  
In essence, whenever we factor $K_C$ as $\OO_C(D) \otimes K_C(-D)$,
we should expect to find a linear subspace of dimension = Cliff(D) and a Shiffer variation supported on this subspace.  Our 
construction will produce such a Shiffer variation and hence such a cohomology class as long as one of the line bundles is base 
point free.  We have discussed this more extensivly in the section line bundles with $\hi{L}=1$. As long as that is the case,
we  can  non-trivial Koszul cohomology  classes. They should be essentially the same as the classes constructed
by Green and Lazarsfeld.

\textbf{Remark 3} The construction of \cite{G1} is not the only way to construct non-trivial classes.  For example  \cite{AN}  
a more general construction is presented. 
 In all these cases the natural location of these classes is in a group of the form $K_{p,1}$.  Our classes 
always live in $K_{p,2}^{*}. $ Thus all previous methods speak to the lenght of the linear strand of the 
minimal free resolution of $S_C$, whereas our methods possibly give information on where the quadratic strand
may start. Recall that the linear strand of a variety (defined by quadrics) is the piece in degree $p$ composed 
of $\OO(-p-1)$'s and the quadratic strand is the piece composed of $\OO(-p-2)$'s In the special 
case $L= K_C$ there is duality
between $K{p,2}$ and $K_{g-p-2,1}$ which doesn't exist for other line bundles allows one to translate results about
the linear strand to results about the quadratic strand.

We also do not always require that the bundle factor as the tensor product of two bundles, both of which
have at least 2 sections. This is necessary for the construction of Green and Lazarsfeld.
 For example, let $L= K_C(D)$ where D is effective.  Then by Theorem \ref{T:clifford_calc_1} $\cliff(C,L)=
d-2$ and by \ref{L:d-2} there exists a $\tau \in T_L(D)$ of rank $d-2$ such that $\tau \in \Sec^{d-1} $ but $\tau \notin \Sec^{d-2}$.
If D is general of with $\deg(D) <g$, then $\H{\OO_C(D)} = 0$ so the method of Green and Lazarsfeld doesn't produce anything interesting.
However:

\begin{thm}

The class $\tau$ produced above gives a nontrivial cohomology class in $K_{d-2,2}$.

\end{thm}

\begin{proof}

The proof goes exactly as in Theorem \ref{T:counterexample}.  Namely by Hassett's criteria, the $d-2$ plane, $Im(\tau)$
cannot meet C.  This means for any expression  $\tau = \sum_{=1}^{i=d-2} t_i^2$, the $t_i^3$ cannot lie in $M_3$ and hence 
$\tau$  cannot be a boundary.

\end{proof}

The final issue to be discussed in this thesis is the relationship between our theorems and  Green's conjecture.
Because $K_{0,2}(C,K_C)= 0$ if and only if $\cliff(C)>0$ and if $K_{0,2}(C,K_C)= 0$, then $K_{1,2}(C,K_C)=0$ iff $\cliff(C) >1$,  
Green conjectured that $K_{p,2}(C,K_C)=0$ for $p< \cliff(C)$.  
There has been a lot of progress on this issue.  After partial results mainly by Schreyer (see \cite{Schr1} and \cite{Schr2}),
Voisin  proved Green's conjecture for a generic curve in \cite{V1} and \cite{V3}. 
 Needless to say her techniques are very different from the ones of this thesis.  In particular she uses the duality between
$K_{p,2}$ and $K_{g-p-2,1}$ which is particular to the case of $L= K_C$.  She then proves the vanishing on a specific curve.
Generic vanishing follows by semi-continuity.  The techniques of this thesis work for a larger class of line bundles and prove results
that hold for all curves carrying such a line bundle.  However the results given here do not seem to prove Green's conjecture on the nose.

We do have a positive result.  Suppose $p< \cliff(C,L)$. We consider the dual of $K_{p,2}(C,L)=0$.  Recall that the dual cohomology 
group is the cohomology of the complex:
\[M_3\otimes \bigwedge^{p-1}(V) \to M_2\otimes \bigwedge^P(V) \to M_1\otimes \bigwedge^{p+1}(V) \]

\begin{thm} 

Suppose that $\varphi\otimes\lambda \in  M_2\otimes \bigwedge^P(V)$ with $\lambda$ decomposable
and that $\delta(\varphi\otimes\lambda)=0$
, then $\varphi\otimes\lambda = \delta(\mu) $ for some $\mu \in M_3\otimes\bigwedge^{p-1}(V)$.  That is to say, in the group dual to
the Koszul cohomology group, any decomposable element is trivial.

\end{thm}

\begin{proof}

Write $\lambda = \sigma_1 \wedge \dots \wedge \sigma_p$.  We first claim that $\im(\varphi) \subset 
\langle \sigma_1 \dots \sigma_p \rangle$. If not let $\tau \in \im(\varphi)$ be such that $\tau \wedge \lambda
\neq 0$ and let $\tau^{\vee} \wedge \lambda^{\vee}$ be dual to this element.  Then $\delta(\varphi \otimes \lambda)(  
\tau^{\vee} \wedge \lambda^{\vee}) = \varphi(\tau^{\vee}) \neq 0$.

Now since $\rk(\varphi)\leq p$ we can write $\varphi = \sum_{i=1}^k \sigma_{p_i}^2 $ where $k \leq p$ and 
$\sigma_{p_i}^2 $ is a Shiffer variation associated to the point $p_i \in C$.  Set $\mu = \frac{1}{3}\sum_{i=1}^k
(-1)^{k-1} \sigma_i^3 \otimes \sigma_1  \wedge \dots \sigma_{i-1} \wedge \sigma_{i+1} \dots \wedge \sigma_k$.  Then
$\mu \in M_3 \otimes \bigwedge^{p-1}(V)$ and 
$\delta(\mu) = \varphi \otimes \lambda$.
\end{proof}

At this point we are left with the words of Ludwig Bemelmans, ``And thats all there is-- there isn't anymore".

\pagebreak
%\addcontentsline{toc}{TOCspecial}{References}

\end{document}